\documentclass[ a4paper,10pt]{amsart}
\usepackage[foot]{amsaddr}

\usepackage[T1]{fontenc}

\usepackage[utf8]{inputenc}

\usepackage{lmodern}

\usepackage[english]{babel}

\usepackage{amsmath}
\makeatletter
\def\maketag@@@#1{\hbox{\m@th\normalfont\normalsize#1}}
\makeatother
\usepackage{amssymb}
\usepackage[alphabetic]{amsrefs}
\usepackage{mathtools}

\usepackage{mathrsfs}

\usepackage{xypic}

\usepackage{tikz-cd}

\usepackage{bookmark}

\usepackage{hyperref} 
\hypersetup{
	colorlinks=true,
	linkcolor=blue,
	filecolor=magenta,      
	urlcolor=cyan,
}

\usepackage{cleveref}

\usepackage{enumitem}
\usepackage{graphicx}
\usepackage{tikz}
\usepackage{tikz-cd}
\usetikzlibrary{matrix,arrows,decorations.pathmorphing}
\usepackage[all,cmtip]{xy}

\newtheorem*{thm-no-num}{Theorem}
\newtheorem*{df-no-num}{Definition}
\newtheorem{thm}{Theorem} [section]

\newtheorem{prop}[thm]{Proposition} 
\newtheorem{lm}[thm]{Lemma} 
\newtheorem{cor}[thm]{Corollary} 
\newtheorem {df}[thm]{Definition}
\theoremstyle{remark}
\newtheorem{rmk}[thm]{Remark}

\newcommand{\bA}{\mathbb{A}}
\newcommand{\Brm}{{\rm B}}
\newcommand{\PP}{\mathbb{P}}
\newcommand{\ZZ}{\mathbb{Z}}

\newcommand{\cl}[1]{\mathcal{#1}}

\newcommand*{\sheafhom}{\mathcal{H}\kern -.5pt om}

\newcommand{\Bcal}{\cl{B}}

\newcommand{\Mcal}{\cl{M}}

\newcommand{\Hcal}{\cl{H}}

\newcommand{\Xcal}{\cl{X}}
\newcommand{\Ycal}{\cl{Y}}
\newcommand{\Kcal}{\cl{K}}

\newcommand{\Gm}{\mathbb{G}_{\rm m}}
\newcommand{\Ga}{\mathbb{G}_{\rm a}}

\newcommand{\bfk}{\mathbf{k}}

\newcommand{\pr}{{\rm pr}}

\newcommand{\K}{{\rm K}}

\newcommand{\Inv}{{\rm Inv}^{\bullet}}
\newcommand{\Spec}{{\rm Spec}}

\begin{document}
\title[Invariants of algebraic stacks in positive characteristic]{Cohomological invariants and Brauer groups of algebraic stacks in positive characteristic}
	\author[A. Di Lorenzo]{Andrea Di Lorenzo}
	\email{andrea.dilorenzo@unipi.it}
	\address{Università di Pisa, Dipartimento di Matematica, Largo Bruno Pontecorvo 5, 56127 Pisa}
	\author[R. Pirisi]{Roberto Pirisi}
	\email{roberto.pirisi@unina.it}
	\address{Università degli studi di Napoli Federico II, Dipartimento di Matematica e Applicazioni, Via Cintia, Monte S. Angelo, 80126 Napoli}
\begin{abstract}
We introduce a theory of cohomological invariants with mod $p^r$ coefficients for algebraic stacks in characteristic $p$. Using these new tools we complete the computation of the Brauer group and cohomological invariants of the stack of elliptic curves over any field.    
\end{abstract}
\maketitle

\section*{Introduction}
\subsection*{Background on cohomological invariants}

Given an algebraic group $G$ over a base field $\bfk$, Serre \cite{GMS} defined the group of cohomological invariants of $G$ as the group of natural transformations from the functor
\[
{\rm T}_G : ({\rm field}/\bfk) \to ({\rm set}), \quad F \mapsto \lbrace G\textnormal{-torsors}/F \rbrace /\simeq
\]
to Galois cohomology with coefficients in some torsion Galois module.  Cohomological invariants were, and are, studied by Garibaldi, Merkurjev, Rost, Totaro and many others. 
A cohomological invariant of $G$ can be thought as an arithmetic equivalent to a characteristic class, functorially assigning to each torsor $E \to \Spec(F)$ an element in the cohomology of $F$. 

The functor ${\rm T}_G$ of $G$-torsors modulo isomorphism can be seen as the functor of points of the classifying stack ${\Bcal}G$. From this point of view, the group of cohomological invariants appears naturally as an invariant of the stack ${\Bcal}G$ rather than the group $G$. Based on this idea the second named author \cite{PirAlgStack} extended the concept to define cohomological invariants of algebraic stacks. When restricted to smooth schemes, this theory recovers the theory of unramified cohomology (see \cites{Salt, CTO}). 
Moreover, the authors showed in \cite{DilPirBr} that for a smooth quotient stack $\Xcal$ and any positive integer $\ell$ coprime to the characteristic of the base field, cohomological invariants in degree two compute the $\ell$-torsion ${\rm Br}'(\Xcal)_{\ell}$ of the cohomological Brauer group.

In \cites{PirCohHypEven, PirCohHypThree, DilCohHypOdd, DilPir, DilPirBr, DilPirRS} the authors investigated the cohomological invariants of the stacks of smooth hyperelliptic curves and their compactifications; using these computations, they obtained an explicit presentation of the Brauer group of the stacks of smooth hyperelliptic curves over any field of characteristic zero, and the prime-to-$p$ part in characteristic $p$.

\subsection*{In positive characteristic}
Most results on classical cohomological invariants (with some notable exceptions, such as \cite{EKLV}*{Appendix B}) use mod $\ell$ Galois cohomology, where $\ell$ is a positive integer coprime to the characteristic of the base field. The assumption is crucial to the theory, most notably providing homotopy invariance. This changed with some recent works of Blinstein-Merkurjev \cite{BlinMerk}, Lourdeaux \cite{Lou} and Totaro \cite{Totp} who explored classical cohomological invariants with $p$-torsion coefficients in characteristic $p$. Totaro in particular developed a complete theory and provided the first examples of full computations of mod $p$ cohomological invariants of a group $G$ in characteristic $p$. The key idea is that the coefficients should be in the ``motivic'' groups ${\rm H}^n(F,\underline{\ZZ}/p^r(j))$, where $\underline{\ZZ}/p^r(j)$ is Voevodsky's complex in the derived category of \'etale sheaves on schemes smooth over $\Spec(\bfk)$.

In this paper we extend Totaro's ideas in the following direction.
\begin{thm-no-num}
There is a theory of mod $p^r$ cohomological invariants for algebraic stacks in characteristic $p$ extending the classical theory, and moreover for a smooth quotient stack $\Xcal$ we have:
\[
{\rm Br}'(\Xcal)_{p^r} = {\rm Inv}(\Xcal, {\rm H}^2(-,\underline{\ZZ}/p^r(1)))
\]
\end{thm-no-num}
Developing the theory requires some non-trivial preliminary work, mostly to extend known results on \'etale motivic cohomology of fields and discrete valuation rings from mod $p$ to mod $p^r$ coefficients. 

We then study mod $p^r$ cohomological invariants of the moduli stack $\Mcal_{1,1}$ of elliptic curves over a field $\bfk$, completing the classification from \cite{DilPirBr}*{Theorem 3.1}. The most interesting case is the following, which we state in the notation of \Cref{sec:explicit}: 
\[
{\rm H}^{n}_{p^r}(F)={\rm H}^n(F,\underline{\ZZ}/p^r(n-1)), \quad \K^n_{p^r}(F)={\rm H}^n(F,\underline{\ZZ}/p^r(n)).
\]
\begin{thm-no-num}
Let $\bfk$ be a field of characteristic $p$. Then
\[
\begin{cases}
\Inv(\Mcal_{1,1},{\rm H}_{2^r}) = \Inv(\bA^1,{\rm H}_{2^r}) \oplus {\rm J}^{\bullet-1}_{2^r}(\bfk) & \mbox{if p}=2,\\
\Inv(\Mcal_{1,1},{\rm H}_{3^r}) = \Inv(\bA^1,{\rm H}_{3^r}) \oplus {\rm H}^{\bullet-1}_{3^r}(\bfk) & \mbox{if p}=3,\\
\Inv(\Mcal_{1,1},{\rm H}_{p^r}) = \Inv(\bA^1,{\rm H}_{p^r}) & \mbox{if p}>3.\\
\end{cases}
\]
The group ${\rm J}_{2^r}^\bullet$ above is a fiber product 
\[
{\rm J}_{2^r}^{\bullet} = {\rm H}^{\bullet}_{2^r}(\bfk) \times_{{\rm H}^{\bullet}_{2^r}(\bfk)} \K_2^{\bullet-1}(\bfk)
\]
where the first map is multiplication by $4$, and the second map comes from the natural morphism 
\[\K_2^{n-1}(\bfk) = \Omega^{n-1}_{\log,\bfk} \to \Omega^{n-1}_{\bfk} \to {\rm H}^{n}_{2^r}(\bfk).\]
\end{thm-no-num}

We remark that all the generators in the groups above are described constructively, meaning that given a smooth irreducible scheme $S/\bfk$ and a family of elliptic curves $E:S \rightarrow \Mcal_{1,1}$ their pullbacks to $\Inv(S,{\rm H}_{p^r})$ can be obtained explicitly from a Weierstrass equation for $E_\xi$, where $\xi$ is the generic point of $S$.

\subsection*{The Brauer group of $\Mcal_{1,1}$}
We apply our study of mod $p^r$ cohomological invariants of $\Mcal_{1,1}$ to the computation of its Brauer group.

While the Brauer group is a fundamental and highly studied invariant, computations of the Brauer group of moduli stacks have appeared only recently. 
The breakthrough result in this direction is Antieau and Meier's paper \cite{AntMeiEll} where they compute the Brauer group of the stack $\Mcal_{1,1}$ of elliptic curves over various bases, including $\mathbb{Q}$, all finite fields of characteristic different from $2$ and, most notably, $\ZZ\left[1/2\right]$ and $\ZZ$. The authors contributed to the topic in \cite{DilPirBr}*{Corollary 3.2}, though it should be noted that the corollary is implied by the stronger result in Meier's unpublished draft \cite{Mei}, and finally Shin obtained a somewhat surprising result in \cite{Shi}, proving that differently from all other characteristics the Brauer group of $\Mcal_{1,1}$ is not trivial over an algebraically closed field of characteristic $2$, and moreover computing it for all finite fields of characteristic $2$.

Our computation of mod $p^r$ cohomological invariants of $\Mcal_{1,1}$ over fields of characteristic $p$, together with \cite{DilPirBr}*{Corollary 3.2} and some extra mod $\ell$ computations, gives us the following description of the Brauer group of $\Mcal_{1,1}$ which holds over any field, with no assumptions on characteristic, perfection or algebraic closure.

\begin{thm-no-num}
Let $\mathcal{M}_{1,1}$ be the stack over $\Spec(\bfk)$ parametrizing elliptic curves. If ${\rm char}(\bfk)=c$ the group ${\rm Br}(\mathcal{M}_{1,1})$ is:
\[
\begin{cases}
{\rm Br}(\mathbb{A}^1_\bfk) \oplus {\rm H}^{1}(\bfk,\ZZ/12\ZZ) & \mbox{if } c \neq 2,\\
{\rm Br}(\mathbb{A}^1_\bfk) \oplus {\rm H}^{1}(\bfk,\ZZ/3\ZZ) \oplus {\rm J} & \mbox{if } c=2, \, x^2+x+1 \, \mbox{irreducible over} \, \bfk, \\
{\rm Br}(\mathbb{A}^1_\bfk) \oplus {\rm H}^{1}(\bfk,\ZZ/12\ZZ) \oplus \ZZ/2\ZZ & \mbox{if } c=2, \, x^2+x+1 \, \mbox{has a root in} \, \bfk
\end{cases}
\]
where ${\rm H}^1(\bfk,\ZZ/4) \subset {\rm J} \subseteq {\rm H}^1(\bfk,\ZZ/8)$ sits in an exact sequence
\[
0 \to  {\rm H}^1(\bfk,\ZZ/4) \to {\rm J} \to \ZZ/2 \to 0. 
\]
\end{thm-no-num}
As for cohomological invariants, our construction provides an explicit descrition of each generator.

\subsection*{Content of the paper}
In Section \ref{sec:motivic cohomology} we recall some basic facts on \'{e}tale motivic cohomology, in particular its relation to logarithmic differential forms.

In Section \ref{sec:inv p} we define mod $p^r$ cohomological invariants for smooth algebraic stacks over a base field of characteristic $p$. After proving some general properties, we restrict to invariants with coefficients in the \'etale motivic cohomology and characterize them as the sheafification of unramified cohomology.

In Section \ref{sec:low degree} we show that cohomological invariants of degree $2$ can be used to compute the cohomological Brauer group. We also give an interpretation of cohomological invariants of degree one in terms of first cohomology groups of the sheaves $\ZZ/p^r$ and $\mu_{p^r}$ (the latter regarded as a sheaf in the flat topology). 

The next two sections are the technical core of the paper. In Section \ref{sec:explicit}  we study the \'{e}tale motivic cohomology groups of fields, in particular their explicit description in terms of log differential forms and symbols. We establish several useful facts, and we recall Izhboldin's results on the tamely ramified and wild subgroups.

In Section \ref{sub:CohDVR} we study cohomology groups of DVR. The results contained here are necessary for the definitions in \Cref{sec:inv p}.

Section \ref{sec:some comp} is devoted to the computation of the cohomological invariants of stacks of the form $X\times\Brm\ZZ/n$, which will be necessary later.

In Section \ref{sec:inv M11} we compute the mod $p^r$ cohomological invariants of $\mathcal{M}_{1,1}$, dividing our analysis into three parts, depending on the characteristic of the base field.

In Section \ref{sec:mod l} we take a little detour to compute the cohomological invariants of $\mathcal{M}_{1,1}$ with coefficients in cycle modules of $\ell$-torsion, with $\ell$ coprime to the characteristic of the base field. These have already been computed in the previous works by the authors except when the characteristic of the base field is $2$ or $3$, which is what is done here.

Finally, in Section \ref{sec:brauer} we compute the Brauer group of $\mathcal{M}_{1,1}$ over any field and we describe its generators.

\subsection*{Notation}

Throughout the paper $\bfk$ will be a field characteristic $p>0$ unless stated otherwise. All schemes and algebraic stacks will be assumed to be of finite type over $\bfk$ or localizations thereof unless stated otherwise. With the notation ${\rm H}^i(\Xcal,A)$ we always mean \'etale cohomology with coefficients in $A$, or lisse-\'etale if $\Xcal$ is not a Deligne-Mumford stack. If $R$ is a $\bfk$-algebra we will often write ${\rm H}^i(R,A)$ for ${\rm H}^i(\Spec(R),A)$. In general for a functor $\rm F$ when no confusion is possible we will often write ${\rm F}(A)$ for ${\rm F}(\Spec(A))$.

\subsection*{Acknowledgements}

We are indebted to Burt Totaro for patiently explaining to us much of the theory on étale motivic cohomology that we use. This paper would not have been possible without his help. We thank Angelo Vistoli for some useful comments and Zinovy Reichstein for a conversation which inspired \Cref{prop:essential} and \Cref{rmk:essM11}.

\section{\'Etale motivic  cohomology}\label{sec:motivic cohomology}

\'Etale cohomology with torsion coefficients behaves rather differently when the coefficients are of $\ell$-torsion, with $\ell$ coprime to the characteristic of our base field, and when they are instead of $p={\rm char}(\bfk)$-primary torsion, with the latter case being significantly harder to deal with. The most basic example is the following: consider the map $\mathcal{P}:\Ga \to \Ga$ given by $a \mapsto a^p-a$. It is a surjection of \'etale sheaves with kernel the constant sheaf $\ZZ/p$. Consequently we have an exact sequence
\[0 \to \ZZ/p \to \Ga \xrightarrow{\mathcal{P}} \Ga \to 0 \]
which induces a long exact sequence in cohomology
\[0 \to {\rm H}^0(-,\ZZ/p) \to {\rm H}^0(-,\Ga) \xrightarrow{\mathcal{P}} {\rm H}^0(-,\Ga) \to {\rm H}^1(-,\ZZ/p) \to {\rm H}^1(-,\Ga) \ldots
\]
If we apply it to $\bA^1$, as it is an affine scheme we get
\[0 \to {\rm H}^0(\bA^1,\ZZ/p) \to {\rm H}^0(\bA^1,\Ga) \xrightarrow{\mathcal{P}} {\rm H}^0(\bA^1,\Ga) \to {\rm H}^1(\bA^1,\ZZ/p) \to 0
\]
which implies that ${\rm H}^1(\bA^1,\ZZ/p)=\bfk\left[t\right]/\mathcal{P}(\bfk\left[t\right])$, a non finitely generated group, even if the base field $\bfk$ is algebraically closed. This shows that mod $p$ \'etale cohomology is not homotopy invariant, thus lacking a fundamental property of its mod $\ell$ counterpart.

Nonetheless, there are many tools at our disposal in this situation as well. At the beginning of the '80s Kato \cite{Kato}*{Pg. 219} used differential forms to define cohomology groups ${\rm H}^i(\bfk,\ZZ/p^r(j))$, which turned out to work impressively well: when $i=1, j=0$ we have ${\rm H}^1(\bfk,\ZZ/p^r(0))={\rm H}_{\textnormal{ \'et}}^1(\bfk,\ZZ/p^r)$ and when $i=2, j=1$ we have ${\rm H}^2(\bfk,\ZZ/p^r(1))={\rm Br}(\bfk)_{p^r}$. To explain why this happens we need more sophisticated technology: let $\ZZ_{\rm tr}$ be Voevodsky's sheaf with transfers. Define \cite{VoeMot}*{Definition 3.1}
\[
\underline{\ZZ}(j)={\rm C}_*\ZZ_{\rm tr}(\Gm^{\wedge j})\left[-j\right]
\]
in the derived category of \'etale sheaves over schemes smooth over $\bfk$. Then $\underline{\ZZ}(0)$ is quasi-isomorphic to $\ZZ$, and $\underline{\ZZ}(1)$ is quasi-isomorphic to $\Gm\left[-1\right]$ \cite{VoeMot}*{Theorem 4.1}.

There is a natural map $\underline{\ZZ}(j) \xrightarrow{\cdot n} \underline{\ZZ}(j)$ and an exact sequence
\[
0 \to \underline{\ZZ}(j) \xrightarrow{\cdot n} \underline{\ZZ}(j) \to \underline{\ZZ}/n(j) \to 0
\]
which in particular for $j=1$ gives us the familiar looking
\[
\xymatrix {
{\rm H}^2(X, \underline{\ZZ}(1)) \ar[r]^{\cdot n} & {\rm H}^2(X, \underline{\ZZ}(1)) \ar[r] & {\rm H}^2(X, \underline{\ZZ}/n(1)) \ar[r] &  {\rm H}^3(X,\underline{\ZZ}(1))_n \ar[r] & 0\\
{\rm H}^1(X,\Gm) \ar@{=}[u] \ar[r]^{\cdot n} & {\rm H}^1(X,\Gm) \ar@{=}[u] \ar@{.>}[r] & (?) \ar@{.>}[r] \ar@{=}[u] & {\rm H}^2(X,\Gm)_n \ar@{=}[u]
}
\]
When $n$ is coprime to the ${\rm char}(\bfk)$ we know that ${\rm H}^2(X,\underline{\ZZ}/n(1))={\rm H}^2(X, \ZZ/n(1))={\rm H}^2(X, \mu_n)$, retrieving the usual Kummer exact sequence. In general if ${\rm char}(\bfk)$ does not divide $n$ then $\underline{\ZZ}/n(j)$ is quasi-isomorphic to $\ZZ/n(j)=\ZZ/n\otimes \mu_n^{\otimes j}$ by \cite{VoeMot}*{Theorem 10.2}. It remains to understand what the groups ${\rm H}^n(X, \underline{\ZZ}/p^r(j))$ are when $p = {\rm char}(\bfk)$.

Let $X/\bfk$ be a smooth scheme, and write $\Omega^n_X$ for the sheaf of differentials of $X$ over $\Spec(\ZZ)$. Inside this sheaf consider the subsheaf (of $\ZZ/p$-modules) $\Omega^n_{\rm log}$ of logarithmic differentials, i.e. the subsheaf generated by elements in the form 
\[
\frac{{\rm d}b_1}{b_1}\wedge \ldots \wedge \frac{{\rm d}b_n}{b_n}
\]
with $b_1, \ldots , b_n$ units. There are corresponding notions of de Rham-Witt differentials ${\rm W}_r\Omega^n$ and logarithmic de Rham-Witt differentials ${\rm W}_r\Omega_{\rm log}^n$. Geisser and Levine \cite{GeisLev}*{Proposition 3.1, Theorem 8.3, Theorem 8.5} proved that in the derived category of Zariski or \'etale sheaves on smooth schemes over a \emph{perfect} field of characteristic $p$ there is an isomorphism
\[
\underline{\ZZ}/p^r(j) = {\rm W}_r\Omega^q_{\rm log}\left[-j\right]
\]
so that ${\rm H}^i(X,\underline{\ZZ}/p^r(j))={\rm H}^{i-j}(X,{\rm W}_r\Omega^j_{\rm log})$. Note that their Theorem 8.5 is stated for $r=1$ but the proof works for general $r$.

Geisser and Levine's result also computes the \'etale motivc cohomology of non perfect fields thanks to Quillen's method (see the proof of \Cref{prop:KMH} or \cite{GeisLev}*{Proposition 3.1}). 

Fields of characteristic $p$ have cohomological dimension $1$ \cite{SerGal}*{Section II, 2.2}, implying that ${\rm H}^i(\bfk,\underline{\ZZ}/p^r(j))=0$ unless $i$ is either $j$ or $j+1$. When $i = j$ we have
\[
{\rm H}^j(\bfk, \underline{\ZZ}/p^r(j))={\rm W}_r\Omega^j_{\bfk,  {\rm log}}=\K^j_{\textnormal{Mil}}(\bfk)/p^r
\]
where $\K^{\bullet}_{\textnormal{Mil}}$ is Milnor's K-theory \cite{GilSza}*{Ch. 7} and the identification is originally by Bloch and Kato \cite{BlKaP}. For $i = j+1$ the description by Geisser and Levine gives
\[
{\rm H}^{j+1}(\bfk, \underline{\ZZ}/p^r(j))={\rm H}^1_{\rm Gal}(\bfk, {\rm W}_r\Omega^j_{\bfk^s, {\rm log}})
\]
where $\bfk^s$ is the separable closure of $\bfk$. This description is not particularly enlightening, but it agrees with Kato's original definition (thus explaining why his groups worked so well), and when $r=1$ the latter is rather simple. 

The group $\Omega^n_{\bfk}$ is generated additively by elements in the form $a\frac{{\rm d}b_1}{b_1}\wedge \ldots \wedge \frac{{\rm d}b_n}{b_n}$. We have a well defined additive homomorphism $\mathcal{P}=\Omega^{n}_{\bfk} \to \Omega_{\bfk}^n$ given by
\[
\mathcal{P}(a\frac{{\rm d}b_1}{b_1}\wedge \ldots \wedge \frac{{\rm d}b_n}{b_n}) = \mathcal{P}(a)\frac{{\rm d}b_1}{b_1}\wedge \ldots \wedge \frac{{\rm d}b_n}{b_n} = (a^p-a)\frac{{\rm d}b_1}{b_1}\wedge \ldots \wedge \frac{{\rm d}b_n}{b_n}
\]
and there is an exact sequence \cite{IzhK}*{Corollary 6.5}:
\begin{equation}\label{eq:omega to H}
0 \to \Omega^n_{\bfk, {\rm log}} \to \Omega^n_{\bfk} \xrightarrow{\mathcal{P}} \Omega^n_{\bfk}/{\rm d}\Omega^{n-1}_{\bfk} \to {\rm H}^{n+1}(\bfk,\underline{\ZZ}/p(n)) \to 0
\end{equation}
In other words we can see ${\rm H}^{n+1}(\bfk,\underline{\ZZ}/p(n))$ as $\Omega^n_\bfk/N$, where $N$ is the subgroup generated by exact differentials and elements in the form $(a^p-a)\frac{{\rm d}b_1}{b_1}\wedge \ldots \wedge\frac{{\rm d}b_n}{b_n}$, with $b_1,\ldots,b_n \in \bfk^*$.

The last tool we need to recall from motivic cohomology is Gros and Suwa's resolution of the sheaf of logarithmic differentials on a scheme $X$ smooth over a perfect field \cite{GrSu}*{Theorem 1.4}. Write
\[
\Hcal^n(X,\underline{\ZZ}/p^r(j)) = {\rm H}^0_{\rm Zar}(X,{\rm H}^n(-,\underline{\ZZ}/p^r(j))).
\]

Gros and Suwa construct a resolution \cite{GrSu}*{Corollary 1.6}:

\begin{equation}\label{eq:Omega}
\small
 0 \rightarrow \Hcal^n(X,\underline{\ZZ}/p^r(n))={\rm W}_r\Omega_{X,{\rm log}} \rightarrow {\rm W}_r\Omega^n_{\bfk(X), {\rm log}} \xrightarrow{\partial} \bigoplus_{x \in X^{(1)}} {\rm W}_r\Omega^{n-1}_{\bfk(x), {\rm log}} \xrightarrow{\partial} \ldots
\end{equation}

Using Gros and Suwa's results together with Quillen's method we get, over any field of characteristic $p$, the following exact sequence, which appears in appendix A of \cite{BlinMerk}:
\begin{equation}\label{eq:ramification}
 0 \rightarrow \Hcal^n(X,\underline{\ZZ}/p^r(j)) \rightarrow {\rm H}^n(k(X),\underline{\ZZ}/p^r(j)) \xrightarrow{\partial} \bigoplus_{x \in X^{(1)}} {\rm H}_{{\textnormal{\'et}}, x}^{n+1}(X,\underline{\ZZ}/p^r(j))   
\end{equation}
where the groups appearing at the end of the sequence above are defined as
\[ 
{\rm H}_{{\textnormal{\'et}}, x}^{\bullet}(X,\underline{\ZZ}/p^r(j))\overset{\textnormal{def}}{=}\varprojlim_{x\in U}  {\rm H}_{{\textnormal{\'et}},U\cap \{\overline{x}\}}^{\bullet}(X,\underline{\ZZ}/p^r(j)).
\]
The colimit runs over all the open subset containing $x$, and the groups on the right are the cohomology groups with support on a closed subscheme.

The sequence, importantly, works at the level of local rings, giving us:
\begin{equation}\label{eq:local ramification}
    \small
    0 \rightarrow \Hcal^n(\mathcal{O}_{X,s},\underline{\ZZ}/p^r(j)) \rightarrow {\rm H}^n(k(X),\underline{\ZZ}/p^r(j)) \xrightarrow{\partial} \bigoplus_{x \in X^{(1)}, s \in \overline{x}} {\rm H}_{{\textnormal{\'et}}, x}^{n+1}(X,\underline{\ZZ}/p^r(j))
\end{equation} 

\section{Invariants in characteristic $p$}\label{sec:inv p}

In this section we generalize the definition of cohomological invariants from \cite{DilPirBr}*{Definition 2.3} (which is already a more general version of the original one from \cite{PirAlgStack}*{Definition 2.2}) to include coefficients of $p^r$-torsion even when our base field $\bfk$ has characteristic $p$. While the functors we will be interested in for the rest of the paper are specifically those in the form ${\rm H}^n(-,\underline{\ZZ}/p(j))$ for $j \in \lbrace n-1,n\rbrace$, allowing for general functors still has value, at least in terms of exposition. The following is the minimun we can ask of our coefficients functor for our theory to make sense. 

\begin{df}
Let $(F,v)$ be a discretely valued field (resp. let $(R,v)$ be a DVR), let $\mathcal{O}_F=\lbrace a \in F \mid v(a) \geq 0 \rbrace$ and let $\bfk_v = \mathcal{O}_F/\lbrace a \mid v(a) > 0 \rbrace = R/{\rm m}_v$.

We say $v$ is a geometric discrete valuation if:
\begin{itemize}
    \item $F$ and $\bfk_v$ are finitely generated over $\bfk$.
    \item $v$ is trivial on $\bfk$.
    \item The transcendence degree of $F$ over $\bfk$ is one more than the transcendence degree of $\bfk_v$ over $\bfk$.
\end{itemize}
Equivalently, $(R,v)$ is a geometric DVR if it is the localization of a regular variety over $\bfk$ at a point of codimension one.
\end{df}

From now on, coherently with our notation choices, we will always assume valuations and DVRs are geometric unless stated otherwise.

Write $(\mathcal{F}/\bfk)$ for the category of finitely generated extensions of $\bfk$, and $({\rm gv}\mathcal{F}/\bfk)$ for the category whose objects are $(F,v)$, where $F$ is a finitely generated extension of $\bfk$ and $v$ is a geometric discrete valuation. 

Finally, write $(\textnormal{Ab})$ for the category of Abelian groups and $(\textnormal{gr-Ab})$ for graded Abelian groups.

\begin{df}
A $v$-functor is the data of two functors
\[\Theta:(\mathcal{F}/\bfk) \to (\textnormal{(gr-)Ab}), \quad \Theta':({\rm gv}\mathcal{F}/\bfk) \to (\textnormal{(gr-)Ab})
\]
with, for each object $(F,v)\in ({\rm gv}\mathcal{F}/\bfk)$ compatible maps
\[{\rm j}_v:\Theta(\bfk_v) \hookrightarrow \Theta'(F,v), \quad {\rm p}_v: \Theta(F) \to \Theta'(F,v) . \]
where ${\rm j}_v$ is injective. We will often write $\Theta$ for the $v$-functor $(\Theta, \Theta', j, p)$.
\end{df}

\begin{rmk}\label{rmk:ellcomparison}
This definition includes all the functors we picked as coefficients in the mod-$\ell$ case, i.e. cycle modules, as we can pick the functors $\Theta, \Theta'$ to be respectively the cycle module ${\rm M}$ and the group ${\rm M}_v$ described in \cite{Rost}*{Remark 1.6}. 

More specifically, when using \'etale cohomology with coefficients in a locally constant $\ell$-torsion Galois module $D$ we can define a $v$-functor using Gabber's theorem \cite{StPr}*{Tag 09ZI}. Define $\Theta(F)={\rm H}^n(F,D)$ and for any DVR $(R,v)$ 
\[ 
 \Theta'(\bfk(R),v)={\rm H}^{n}(\bfk(R^{\rm h}),D), \, {\rm p}_v=\pi_{\rm h}^*, \, {\rm j}_v=(i')^* \circ (i^*)^{-1}
\]
where $R^{\rm h}$ is the Henselization, $\pi_{\rm h}:\Spec(R^{\rm h}) \to \Spec(R)$ is the projection, the map $i: \Spec(\bfk_v) \to \Spec(R)$ is the inclusion of the closed point, which induces an isomorphism on cohomology, and $i': \Spec(\bfk(R^{\rm h})) \to \Spec(R^{\rm h})$ is the inclusion of the generic point.
\end{rmk}

Using the notion of a $v$-functor we can give a broad definition of cohomological invariants. Given an algebraic stack $\Xcal/\bfk$, let ${\rm Pt}_{\Xcal}$ be the functor 
\[
{\rm Pt}_{\Xcal}: ({\rm fields}/\bfk) \to (\mbox{sets}), \quad {\rm Pt}_{\Xcal}(F) = \Xcal(\Spec(F))/{\simeq}.
\]

\begin{df}
Let $\Theta$ be a $v$-functor. A cohomological invariant with coefficients in $\Theta$ of an algebraic stack $\Xcal/\bfk$ is a natural transformation
\[
\alpha: {\rm Pt}_\Xcal \longrightarrow \Theta
\]
satisfying the following continuity condition: for any DVR $(R,v)$ with a map $\Spec(R) \to \Xcal$ there exists a finite \'etale extension 
$(R', v') \to (R,v)$
such that $\bfk_{v'}=\bfk_v$ and
\[{\rm p}_{v'}(\alpha(\bfk(R)))={\rm j}_{v'}(\alpha(\bfk_{v'})).\]
\end{df}

\begin{rmk}
Using \Cref{rmk:ellcomparison} it's easy to see that this definition retrieves the general definition for mod $\ell$ coefficients given in \cite{DilPirBr}*{Definition 2.3}
\end{rmk}

The more practical way to see a cohomological invariant $\alpha$ is as a way to functorially assign to each point $x: \Spec(F) \to \Xcal$ an element $\alpha(x) \in \Theta(F)$, with the continuity condition requiring that it behaves well with respect to specialization.

Cohomological invariants of $\Xcal$ with coefficients in $\Theta$ clearly form a group, graded if $\Theta$ is, i.e. if it maps fields (resp. valued fields) to graded Abelian groups. We call it ${\rm Inv}(\Xcal,\Theta)$. If $\Theta$ is graded we write $\Inv(\Xcal,\Theta)$ for the whole group and ${\rm Inv}^n(\Xcal,\Theta)$ for ${\rm Inv}(\Xcal,\Theta^n)$.

There is an obvious pullback map: a morphism $f:\Ycal \to \Xcal$ induces a natural transformation $f:{\rm Pt}_{\Ycal} \to {\rm Pt}_\Xcal$ and given an element $\alpha \in \Inv(\Xcal,\Theta)$ we compose to obtain $f^*\alpha \in \Inv(\Ycal, \Theta)$.

In other words, given $y: \Spec(F) \to \Ycal$ we define
\[
(f^*\alpha)(y)=\alpha(f\circ y).
\]

Now recall the following definition \cite{PirAlgStack}*{Definition 3.2}: a \emph{smooth-Nisnevich} morphism $\Ycal \to \Xcal$ is a smooth, representable map of algebraic stacks such that for any field $F$ and any morphism $\Spec(F) \to \Xcal$ we have a lifting

\[
\xymatrix {
 & \Ycal \ar[d]^{f} \\
 {\rm{Spec}}(F) \ar[ur] \ar[r] & \Xcal.
}
\]

Typical examples are a quotient map $X \to \left[X/G\right]$ where $G$ is a \emph{special} smooth affine algebraic group, in the sense that $G$-torsors are always Zariski locally trivial (e.g. ${\rm GL}_n, \Ga$), as well as vector or affine bundles.

Smooth-Nisnevich morphisms define a Grothendieck topology, which we refer to as the smooth-Nisnevich topology, on the category of representable smooth morphisms $\Ycal \to \Xcal$. We call the resulting site the smooth-Nisnevich site of $\Xcal$. If $\Xcal$ is a scheme then any smooth-Nisnevich morphism has a (\'etale) Nisnevich section, so on schemes the smooth-Nisnevich site is equivalent to the ordinary Nisnevich one \cite{PirAlgStack}*{Proposition 3.3}.

In \cite{PirAlgStack}*{Proposition 3.6} the second author proves that an algebraic stack with affine stabilizers has a smooth-Nisnevich cover if the base field is infinite, and relies on a finer topology, the $m$-smooth Nisnevich topology, for finite fields. Using result by Aizenbud and Avni we can easily show that smooth-Nisnevich covers of finite type exist over any base field.

\begin{prop}
Let $\Xcal/\bfk$ be an algebraic stack with affine geometric stabilizers. Then there exists a scheme $X$ and a smooth-Nisnevich cover $X \to \Xcal$ of finite type.
\end{prop}
\begin{proof}
The proof of \cite{PirAlgStack}*{Proposition 3.6} can be applied verbatim to show that there is a scheme $X_1$ with a smooth, finite type map $X_1\to \Xcal$ lifting all points over infinite fields. On the other hand, \cite{AizAvn}*{Theorem A} shows that there exists a scheme $X_2$ with a smooth, finite type map $X_2 \to \Xcal$ lifting all points over finite fields. Then $X_1 \sqcup X_2 \to \Xcal$ is the smooth-Nisnevich cover we are looking for.
\end{proof}

Given an algebraic stack $\Xcal$ we can consider the functor $\Inv(-,\Theta)$ on the smooth-Nisnevich site of $\Xcal$.

\begin{thm}\label{thm:sheaf}
The functor $\Inv(-,\Theta)$ is a smooth-Nisnevich sheaf.
\end{thm}
\begin{proof}
Let $\Ycal \to \Xcal$ be a smooth-Nisnevich covering, and assume an element $\alpha \in \Inv(\Ycal,\Theta)$ satisfies ${\rm pr}_1^* \alpha= {\rm pr}_2^* \alpha$. For each morphism $p: \Spec(F) \to \Xcal$ define $\beta(p)$ to be $\alpha(p')$ where $p'$ is any lifting of $p$. This is well defined as given a second lifting $p''$ there is a map $q: \Spec(F) \to \Ycal \times_{\Xcal} \Ycal$ such that ${\rm pr}_1 \circ q = p', {\rm pr}_2 \circ q = p''$, so the values of $\alpha$ at the two points must be the same. Functoriality is immediate.

Finally, given a map from the spectrum of a DVR $(R,v)$ to $\Xcal$ we can always find a finite extension $(R',v')$ with $\bfk_{v'}=\bfk_v$ which lifts to $\Ycal$. This is because the map from the Henselization $(R^{\rm h},v)$ always lifts. This shows that we can check the continuity condition ${\rm p}_{v'}\beta(\bfk(R'))={\rm j}_{v'}\beta(j(\bfk_{v'}))$ on the lifting $\Spec(R') \to \Ycal$, concluding our proof.
\end{proof}

A direct consequence of the continuity condition is that a cohomological invariant of an irreducible smooth stack is determined by its ``generic'' value.

\begin{lm}\label{lm:injreg}
Let $\Xcal/\bfk$ be an algebraic stack and let $(A,m)$ be a regular local ring with a map $\Spec(A) \to \Xcal$, and let $\alpha\in\Inv(\Xcal,\Theta)$. Let $\xi, x$ be respectively the image of the generic and closed point of $\Spec(A)$. Then if $\alpha(\xi)=0$ we have $\alpha(x)=0$.
\end{lm}
\begin{proof}
First assume that $A$ is a DVR. We may pass to an extension $A'$ such that $k_{v'}=k_v$ and ${\rm j}_{v'}$ is injective. Then the fact that $\alpha(\xi)=0$ implies that $\alpha(\bfk(A'))=0$ which in turn implies that ${\rm j}_{v'}(\alpha(x))={\rm p}_{v'}(0)=0$. As ${\rm j}_{v'}$ is injective we conclude that $\alpha(x)=0$.

Now assume that the dimension $d$ of $A$ is greater than $1$, an let $a_1,\ldots,a_d$ be a regular sequence defining the maximal ideal $m$. By the one-dimensional case we know that the statement holds for the DVR $A_{(a_1)}$. This shows that if $\alpha$ is zero at $\xi$ it is zero at the generic point $\xi'$ of $\Spec(A/(a_1))$. As $A/(a_1)$ is a regular local ring of dimension $d-1$ by induction we have that $\alpha(\xi')=0$ implies $\alpha(x)=0$.
\end{proof}

\begin{prop}\label{prop:generic}
Let $\Xcal/\bfk$ be an irreducible smooth algebraic stack, and let $\mathcal{U} \subset \Xcal$ be a non-empty open substack. Then the pullback 
\[
\Inv(\Xcal, \Theta) \to \Inv(\mathcal{U},\Theta)
\]
is injective.
\end{prop}
\begin{proof}
Consider a smooth-Nisnevich covering $X \to \Xcal$, and let $U \to \mathcal{U}$ be the pullback to $\mathcal{U}$. Let $\alpha$ be a cohomological invariant of $\Xcal$. Given $p: \Spec(F) \to \Xcal$, take a lifting $p': \Spec(F) \to X$ and let $\xi$ be the generic point of the corresponding connected component of $X$. Then $\xi$ belongs to the image of $U$ as $U \to X$ is dominant on every component. Now if $\alpha_U = 0$ we must have $\alpha(\xi)=0$ which implies $\alpha(p')=\alpha(p)=0$.
\end{proof}

\subsection{Invariants with coefficients in \'etale motivic cohomology}

For the rest of this section we will restrict our attention to $\Theta = {\rm H}^n(-,\underline{\ZZ}/p^r(j))$. For now we will not see these as graded functors: note that as we observed that only the cases $j=n-1,n$ produce nonzero groups the obvious grading coming from cohomological degree does not produce an interesting functor. 

First we need to show that they are $v$-functors. This will require some facts that will be proven in \Cref{sec:explicit} and \Cref{sub:CohDVR}, which we postpone as they require some work. In \Cref{sec:explicit} we will also introduce the correct graded functors, and describe them explicitly.

We can define a $v$-functor structure for $\Theta = {\rm H}^n(-,\underline{\ZZ}/p^r(n))$ by following the idea in \cite{Rost}*{Remark 1.6}.  Recall that for a field $F/\bfk$ we have 
\[{\rm H}^n(F,\underline{\ZZ}/p^r(n))=\K^n_{\textnormal{Mil}}(F)/{p^r}.\]

For a DVR $(R,v)$ define $\K^{\bullet}_{v}=\K^{\bullet}_{\textnormal{Mil}}(F)/\lbrace 1 + m_v\rbrace \K^{\bullet}_{\textnormal{Mil}}(F)$. We set $\Theta'(\bfk(R),v) = \K^{n}_v/p^r$. There are maps 
\[{\rm p}_v: {\rm H}^n(\bfk(R),\underline{\ZZ}/p^r(n)) \to \K_v^n/p^r,\quad {\rm j}_v: {\rm H}^n(\bfk_v,\underline{\ZZ}/p^r(n)) \to \K_v^n/p^r\]
given respectively by the projection and
\[
{\rm j}_v\lbrace b_1,\ldots, b_n \rbrace \to {\rm p}_v \lbrace \tilde{b}_1, \ldots \tilde{b}_n \rbrace
\]
where $\tilde{b}_1, \ldots, \tilde{b}_n$ are liftings of $b_1, \ldots, b_n$. We have an exact sequence
\[
0 \to {\rm H}^n(\bfk_v,\underline{\ZZ}/p^r(n)) \xrightarrow{{\rm j}_v} \K^{n}_v/p^r \xrightarrow{\partial} {\rm H}^n(\bfk_v,\underline{\ZZ}/p^r(n)) \to 0 
\]
where $\partial$ is the ramification map defined in \Cref{subsec:ramif}.

To obtain a structure of $v$-functor on $\Theta={\rm H}^{n+1}(-,\underline{\ZZ}/p^r(n))$ we use \Cref{prop:DVR} and the identification obtained in Section \ref{sub:CohDVR} of the unramified subgroup of ${\rm H}^{n+1}(\bfk(R),\underline{\ZZ}/p^r(n))$ with the cohomology group ${\rm H}^{n+1}(R,\underline{\ZZ}/p^r(n))$ (see \Cref{prop:Hunr}).

Consider a DVR $(R,v)$. If we take the Heselization $(R^{\rm h},v)$ we get
\[{\rm H}^{n+1}(R^{\rm h},\underline{\ZZ}/p^r(n))={\rm H}^{n+1}(\bfk_v,\underline{\ZZ}/p^r(n)) \xrightarrow{(i')^* \circ (i^*)^{-1}} {\rm H}^{n+1}(\bfk(R^{\rm h}),\underline{\ZZ}/p^r(n))\] 
where $i$ is the inclusion of the closed point of $\Spec(R^{\rm h})$ and $i'$ is the inclusion of the generic point. This shows that we can set 
\[
\Theta'(\bfk(R),v)={\rm H}^{n+1}(\bfk(R^{\rm h}),\underline{\ZZ}/p^r(n)), \quad
{\rm p}_v = \pi_{\rm h}^*, \quad {\rm j}_v=(i')^* \circ (i^*)^{-1}.
\]

Note that we have a map
\[
{\rm H}^{n}(\Xcal,\underline{\ZZ}/p^r(j)) \to {\rm Inv}(\Xcal, {\rm H}^{n}(-,\underline{\ZZ}/p^r(j)))
\]
given by restriction, i.e. $h \in {\rm H}^{n}(\Xcal,\underline{\ZZ}/p^r(j))$ maps to the element 
\[\tilde{h} \in {\rm Inv}(\Xcal, {\rm H}^{n}(-,\underline{\ZZ}/p^r(j)))\]
defined by setting, for a point $x: \Spec(F) \to \Xcal$
\[\tilde{h}(x)=x^*h \in {\rm H}^{n}(F,\underline{\ZZ}/p^r(j))
\]
It's immediate that $\tilde{h}$ is functorial and satisfies the continuity condition. The following lemma shows that thanks to the continuity condition, when we restrict to DVRs this map is \emph{Nisnevich-locally} surjective.

\begin{lm}\label{lm:local}
Let $j \in \lbrace n-1,n \rbrace$, and let $(R,v)$ be a DVR. In both cases, we have 
\[{\rm p}_v^{-1}( {\rm j}_v({\rm H}^{n}(\bfk_v,\underline{\ZZ}/p^r(j))))=(i')^*({\rm H}^{n}(R,\underline{\ZZ}/p^r(j)).\]
\end{lm}
\begin{proof}
The case $j=n$ is a direct consequence of \Cref{prop:KMH}. The case $j=n-1$ is shown in \Cref{prop:DVR}, \Cref{cor:unramcompl} and \Cref{prop:Hunr}.
\end{proof}

Consider a smooth irreducible scheme $X/\bfk$, with generic point $\xi$. As a consequence of \Cref{prop:generic} we have an inclusion 
\[
{\rm Inv}(X,{\rm H}^n(-,\underline{\ZZ}/p^r(j))) \subseteq {\rm H}^n(\xi,\underline{\ZZ}/p^r(j)).
\]
Recall from \Cref{eq:ramification} that for any point $x \in X^{(1)}$ we have a ramification map
\[
\partial_x:{\rm H}^n(\xi,\underline{\ZZ}/p^r(j)) \to {\rm H}^{n+1}_{x}(X,\underline{\ZZ}/p^r(j)).
\]

We claim that for any $\alpha \in {\rm Inv}(X,{\rm H}^n(-,\underline{\ZZ}/p^r(j)))$ the element $\alpha(\xi)$, which belongs to ${\rm H}^n(\xi,\underline{\ZZ}/p^r(j))$, must lie in the kernel of $\partial_x$.

\begin{lm}\label{lm:unraMerk}
Let $X/\bfk$ be a smooth irreducible scheme with generic point $\xi$. Then for any $\alpha\in {\rm Inv}(X, {\rm H}^n(-,\ZZ/p^r(j)))$ and any $x \in X^{(1)}$ we have
\[
\partial_x(\alpha(\xi))=0
\]
\end{lm}
\begin{proof}
By \Cref{lm:local} the continuity condition implies that after passing to a Nisnevich neighbourhood $(X',x')$ of $x$ the element $\alpha(\xi')$ belongs to $ {\rm H}^n(\mathcal{O}^{\rm h}_{X',x'},\underline{\ZZ}/p^r(j))$. But e.g. \cite{Mil}*{Lemma 1.16, Proposition 1.27} implies that 
\[
{\rm H}_{x}^{\bullet}(X,\underline{\ZZ}/p^r(j))={\rm H}_{x'}^{\bullet}(X',\underline{\ZZ}/p^r(j))={\rm H}_{x}^{\bullet}(\mathcal{O}^{\rm h}_{X,x},\underline{\ZZ}/p^r(j)).
\]
Thus we have a commutative diagram
   \[
     \xymatrix{
     {\rm H}^n(\bfk(X),\ZZ/p^r(j)) \ar[r]^{\partial_x} \ar[d] & {\rm H}_{x}^{n+1}(X,\ZZ/p^r(j)) \ar[d]^{\simeq} \\
     {\rm H}^n(\bfk(X'),\underline{\ZZ}/p^r(j)) \ar[r]^{\partial_{x'}} & {\rm H}_{x'}^{n+1}(X',\underline{\ZZ}/p^r(j))
     }
   \]
and applying the local exact sequence \Cref{eq:local ramification} at $(X',x')$ we conclude that $\partial_{x'}(\alpha(\xi'))=0$, proving our claim.
\end{proof}

We are ready to extend the description in \cite{PirAlgStack}*{Theorem 4.4} to invariants with coefficients in ${\rm H}^n(-,\ZZ/p^r(j))$; we will do it in two steps. First we describe the cohomological invariants of a smooth, irreducible scheme.

\begin{prop}\label{prop:SheafHSch}
Let $X/\bfk$ be a smooth, irreducible scheme. Then 
\[
{\rm Inv}(X,{\rm H}^n(-,\underline{\ZZ}/p^r(j)))={\rm H}^0_{\rm Zar}(X,{\rm H}^n(-,\underline{\ZZ}/p^r(j))).
\]
\end{prop}
\begin{proof}
There is an obvious map 
\[
h:{\rm H}^0_{\rm Zar}(X,{\rm H}^n(-,\underline{\ZZ}/p^r(j))) \to {\rm Inv}(X,{\rm H}^n(-,\underline{\ZZ}/p^r(j)))
\]
as the map ${\rm H^i}(X, \underline{\ZZ}/p^r(j)) \to {\rm Inv}^i(X,{\rm H}^n(-,\underline{\ZZ}/p^r(j))) $ factors through the smooth-Nisnevich sheafification, which in turn factors through the Zariski sheafification. Now note that that the map ${\rm Inv}(X,{\rm H}^n(-,\underline{\ZZ}/p^r(j))) \to {\rm H}^n(\bfk(X),\underline{\ZZ}/p^r(j))$ is injective by \Cref{lm:injreg} and factors through the subgroup of unramified elements by \Cref{lm:unraMerk}. Now we have maps

\[
\xymatrix {
 &  {\rm Inv}(X,{\rm H}^n(-,\underline{\ZZ}/p^r(j))) \ar[d]^{f} \\
 {\rm H}^0_{\rm Zar}(X,{\rm H}^n(-,\underline{\ZZ}/p^r(j))) \ar[ur]^{h} \ar[r]^{g} & {\rm Ker}(\partial) \subseteq {\rm H}^n(\bfk(X),\underline{\ZZ}/p^r(j)) 
}
\]

where $f$ is the value at the generic point and $g$ is the pullback to the generic point, an isomorphism thanks to \Cref{eq:ramification}. Now $f \circ h = g$ is an isomorphism and $f$ is injective, so we conclude that $h$ and $f$ must both be isomorphisms as well.

\end{proof}

As an immediate consequence we obtain:

\begin{thm}\label{thm:SheafHStack}
Let $\Xcal/\bfk$ be a smooth algebraic stack. Then the functor 
\[
{\rm Inv}(-,{\rm H}^n(-,\underline{\ZZ}/p^r(j))): ({\rm Sm}/\Xcal) \to (\mbox{abelian groups})
\]
is the smooth-Nisnevich sheafification of ${\rm H}^{n}(-,\underline{\ZZ}/p^r(j))$.
\end{thm}
\begin{proof}
This is an immediate consequence of the fact that ${\rm Inv}(-,{\rm H}^n(-,\underline{\ZZ}/p^r(j)))$ is a Nisnevich sheaf (\Cref{thm:sheaf}), the existence of a map ${\rm H}^n(-,\underline{\ZZ}/p^r(j)) \to {\rm Inv}(-,{\rm H}^n(-,\underline{\ZZ}/p^r(j)))$ and \Cref{prop:SheafHSch}.
\end{proof}

In the next section we will show how cohomological invariants with coefficients in \'etale motivic cohomology retrieve some standard cohomology groups and invariants; while we don't have such a natural interpretation for higher degree invariants of a stack it is easy to see that they provide a lower bound for its \emph{essential dimension} \cites{BuRe, BeFa}, here conjugated in the sense of \cite{BrReVi}, which roughly speaking, measures the minimum number of independent variables necessary to define all the objects parametrized by it. 

Recall (e.g. \cite{Mil}*{Chapter V, Lemma 1.12}) that if $f:Y \to X$ is a finite \'etale map of constant degree $d$ and $A$ is an \'etale sheaf on $X$ there is a transfer morphism $f_*:{\rm H}^{n}(Y,A_{\mid Y}) \to {\rm H}^{n}(X,A)$ and $f_*f^*\alpha = d \alpha$.

\begin{prop}\label{prop:essential}
Let $\Xcal$ be a smooth algebraic stack, and $\alpha \in {\rm Inv}(\Xcal,{\rm H}^n(-,\underline{\ZZ}/p(j)))$. Assume that $\alpha$ is not zero when pulled back to $\Xcal_{\overline{\bfk}}$. Then we have
\[
{\rm ed}(\Xcal) \geq {\rm ed}_p(\Xcal) \geq n.
\]
\end{prop}
\begin{proof}
We may assume that $\bfk = \overline{\bfk}$ as passing to $\overline{\bfk}$ cannot increase the essential dimension and $p$-essential dimension.

We will follow the idea in \cite{TotSp}*{Lemma 3.1}. The two relevant cases are $j=n$ and $j=n-1$. We know, respectively by \cite{GeisLev}*{Theorem 8.3} and the fact that, in Geisser and Levine's notation, we have $\nu^n(X)=0$ for $n > {\rm dim}(X)$ for $j=n$ and by \cite{KaKu}*{12, Section 3, Corollary 2} for $j=n-1$, that in both cases the fact that $\alpha(x) \neq 0$, where $x: \Spec(F) \to \Xcal$, implies that the transcendence degree of $F$ over $k$ is at least $n$. This proves that ${\rm ed}(\Xcal) \geq n$.

To show the inequality for ${\rm ed}_p(\Xcal)$ observe that if $d=\left[ L:F \right]$ is coprime to $p$, and $f:\Spec(L) \to \Spec(F)$ then for an element $\alpha \in {\rm H}^n(F,\underline{\ZZ}/p(j))$ we must have $f_*f^*\alpha=d \alpha$, which shows that if $\alpha$ is nonzero it must stay so after pulling back to $L$. So if we have an object $x':\Spec(E) \to \Xcal$ such that $x, x'$ are equal after passing to an extension $F'/F$ that is finite of degree coprime to $p$ we must have that $\alpha(x') \neq 0$, which shows that the transcendence degree of $E$ over $\bfk$ must be at least $n$.
\end{proof}

\section{Invariants in low degrees}\label{sec:low degree}

The purpose of this section is to prove that, in analogy to their mod $\ell$ counterparts, when $\Xcal/\bfk$ is a smooth quotient stack its cohomological invariants compute the ordinary \'etale cohomology groups ${\rm H}^1(\Xcal,\ZZ/p^r)$ and the flat cohomology group ${\rm H}_{\rm fl}^1(\Xcal,\mu_p)$ in degree one, and the $p$-primary torsion in the cohomological Brauer group in degree two. The ideas are the same as for the mod $\ell$ invariants, with some updates to make up for the lack of homotopy invariance. We begin with the easier case of degree one invariants. In this case we get a more complete result than what we have in \cite{DilPirBr}*{Lemma 2.18}, as we have no requirement on the group action. First we show that for smooth schemes the Nisnevich cohomology is trivial when the coefficients are constant.

\begin{lm}\label{lm:NisConstant}
Let $A$ be a finite abelian group, and let $X/\bfk$ be a smooth connected scheme. Then ${\rm H}^i_{\rm Nis}(X,A)=0$ for $i >0$.
\end{lm}
\begin{proof}
Let $\xi:\Spec(\bfk(X)) \to X$ be the generic point of $X$. Then as $X$ is smooth we have $\xi_*\xi^*\ZZ/p^r = \ZZ/p^r$. The Leray-Cartan spectral sequence applied to the map $\xi$ reads
\[
{\rm H}^i_{\rm Nis}(X,{\rm R}^j\xi_*\ZZ/p^r) \Rightarrow {\rm H}^{i+j}_{\rm Nis}(\Spec(\bfk(X)),\ZZ/p^r)
\]
but the Nisnevich site of the spectrum of a field is trivial, so all the terms in the abutment except the first and all the ${\rm R}^i$ for $i>0$ vanish, which shows that ${\rm H}^i_{\rm Nis}(X,\ZZ/p^r)=0$ for $i>0$.
\end{proof}

On smooth schemes we can compare the flat cohomology group ${\rm H}^1_{\rm fl}(X,\mu_p)$ with ${\rm Inv}(X,{\rm H}^1(-,\underline{\ZZ}/p(1))$ thanks to results by Illusie, Gros and Suwa.

\begin{lm}\label{lm:mupflat}
Let $X/\bfk$ be a smooth scheme. Then 
\[{\rm H}^1_{\rm fl}(X,\mu_{p^r}) = {\rm H}^1(X,\underline{\ZZ}/p^r(1))\]
As a consequence, if $\Xcal$ is a smooth algebraic stack an element $\alpha \in {\rm H}^1_{\rm fl}(\Xcal,\mu_{p^r})$ induces a cohomological invariant by pullback.
\end{lm}
\begin{proof}
The first statement is \cite{VoeMot}*{Remark 4.10}. Now given $\alpha \in {\rm H}^1_{\rm fl}(\Xcal,\mu_{p^r})$ all we have to check is that the association 
\[
(x:\Spec(F) \to \Xcal) \mapsto x^*\alpha \in {\rm H}^1_{\rm fl}(\Spec(F),\mu_{p^r})={\rm H}^1(F,\underline{\ZZ}(p^r(1))
\]
satisfies the continuity condition, but this is clear as given any Henselian DVR $R, v$ with a map $\Spec(R) \to \Xcal$ we can see the pullback of $\alpha$ to $\Spec(R)$ as an element of ${\rm H}^1(R,\underline{\ZZ}/p^r(1))$.
\end{proof}

Recall \cites{Tot, EG} that, given a scheme $X/\bfk$ with an action by an affine smooth group scheme $G/\bfk$, an \emph{equivariant approximation} $X'$ of $\left[ X/G\right]$ is obtained by taking a representation $V$ of $G$ such that $G$ acts freely on an open subset $U \subset V$ and ${\rm codim}(V \smallsetminus U, V) > 1$. Such a representation can always be found under these hypotheses \cite{EG}*{Lemma 9}.

\begin{lm}\label{lm:smNisconstant}
Let $X/\bfk$ be a smooth scheme with an action of a smooth algebraic group $G/\bfk$, and write $\Xcal=\left[X/G\right]$. Then 
\begin{itemize}
    \item ${\rm H}^i_{\textnormal{sm-Nis}}(\Xcal,\ZZ/p^r) = 0$ for $i >0$.
    \item ${\rm H}^i_{\textnormal{sm-Nis}}(\Xcal, \mu_{p^r})=0$ for $i>0$.
\end{itemize}
\end{lm}
\begin{proof}
First, observe that $\mu_{p^r}$ is a constant sheaf on the Nisnevich site of any smooth scheme. We write $A$ for either $\ZZ/p^r$ or $\mu_{p^r}$.

Pick an equivariant approximation $X'=(X\times U)/G \xrightarrow{\pi} \Xcal$, and consider the Leray spectral sequence associated to the morphism $f:X'\to \Xcal$:
\[
{\rm H}^i_{\textnormal{sm-Nis}}(\Xcal,{\rm R}^j f_*A) \Rightarrow {\rm H}^{i+j}_{\textnormal{sm-Nis}}(X',A).
\]
We claim that the sequence collapses on the first row and all the terms in the abutment are zero except for the first. To see this, first note that on schemes the smooth-Nisnevich and Nisnevich cohomology are equal \cite{PirAlgStack}*{Proposition 3.3}. 
Then by \Cref{lm:NisConstant} all the terms in the abutment except for the first are zero. Moreover, the fibers of $X' \to \Xcal$ are smooth and connected, which means all the ${\rm R}^q$ are zero for $q > 0$. Finally, ${\rm R}^0 f_*A = A$, concluding our proof.
\end{proof}

With this, all that is left is to apply the appropriate spectral sequence.

\begin{prop}\label{prop:H1}
Let $X/\bfk$ be a smooth scheme with an action of a smooth algebraic group $G/\bfk$, and write $\Xcal=\left[X/G\right]$. We have
\begin{align*}
{\rm Inv}(\Xcal, {\rm H}^1(-,\underline{\ZZ}/p^r(0))) = {\rm H}^1(\Xcal,\ZZ/p^r) \\
{\rm Inv}(\Xcal, {\rm H}^1(-,\underline{\ZZ}/p^r(1))) = {\rm H}_{\rm fl}^1(\Xcal,\mu_{p^r})
\end{align*}
\end{prop}
\begin{proof}
We begin by proving the first equality. Consider the category ${\rm Sm}/\Xcal$ whose objects are representable smooth morphisms $\Ycal \to \Xcal$, with morphisms the commutative squares. On this category we can consider both the smooth and smooth-Nisnevich topology, and we write $({\rm Sm}/\Xcal)_{\rm Sm}$ and $({\rm Sm}/\Xcal)_{\textnormal{sm-Nis}}$ for the respective Grothendieck sites. The identity $({\rm Sm}/\Xcal) \to ({\rm Sm}/\Xcal)$ is continuous from the smooth-Nisnevich to the smooth sites, and by \cite{StPr}*{TAG 00X6} it induces a morphism of sites $\pi:({\rm Sm}/\Xcal)_{\rm sm} \to ({\rm Sm}/\Xcal)_{\textnormal{sm-Nis}}$. Consequently there is a Leray spectral sequence
\[
{\rm H}^i_{\textnormal{sm-Nis}}(\Xcal,{\rm R}^j\pi_*\ZZ/p^r) \Rightarrow {\rm H}^{i+j}_{\rm sm}(\Xcal,\ZZ/p^r).
\]
By \Cref{lm:smNisconstant} the first row of the page is zero except for the first group, showing that 
\[{\rm H}^0_{\textnormal{sm-Nis}}(\Xcal,{\rm R}^1\pi_*\ZZ/p^r)={\rm H}^1_{\rm sm}(\Xcal,\ZZ/p^r).\]

Now, smooth cohomology with coefficients in $\ZZ/p^r$ is equal to \'etale cohomology, and $\underline{\ZZ}/p^r(0)$ is quasi-isomorphic to $\ZZ/p^r$, so the left term is the same as ${\rm H}^0_{\textnormal{sm-Nis}}(\Xcal,{\rm H}^1(-,\underline{\ZZ}/p^r(0)))$, and the second is equal to ${\rm H}^1(\Xcal,\ZZ/p^r)$, allowing us to conclude.

The second equality is proven exactly in the same way: using the Leray spectral sequence for the inclusion of the smooth-Nisnevich site into the flat site we show that 
\[
{\rm H}_{\rm fl}^1(\Xcal,\mu_{p^r})={\rm H}^0_{\textnormal{sm-Nis}}(\Xcal,{\rm R}^1\pi_*\mu_{p^r})
\]
which shows that ${\rm H}_{\rm fl}^1(\Xcal,\mu_{p^r})$ is a smooth-Nisnevich sheaf for smooth quotient stacks. But by \Cref{lm:mupflat} we know it is equal to cohomological invariants on smooth schemes, showing that the two are equal on all smooth quotient stacks.
\end{proof}

Now we deal with degree $2$. First we show that ${\rm Br}'$, and in fact the whole group ${\rm H}^2(-,\Gm)$, is a smooth-Nisnevich sheaf on smooth stacks.

\begin{lm}\label{lm:Brsheaf}
Let $X/\bfk$ be a smooth scheme, and let $G/\bfk$ be an affine smooth algebraic group acting on it, and set $\Xcal=\left[X/G\right]$. Then 
\[
{\rm H}^2(\Xcal, \Gm)={\rm H}^0_{\textnormal{sm-Nis}}(\Xcal, {\rm H}^2(-, \Gm)).
\]
In other words, ${\rm H}^2(-, \Gm)$ is a smooth-Nisnevich sheaf over $\Xcal$.
\end{lm}
\begin{proof}
Consider again the morphism of sites $\pi:({\rm Sm}/\Xcal)_{\rm sm} \to ({\rm Sm}/\Xcal)_{\textnormal{sm-Nis}}$. it induces a Leray-Grothendieck spectral sequence
\[
{\rm H}^i_{\textnormal{sm-Nis}}(\Xcal, R^j \pi_*\Gm) \Rightarrow {\rm H}^{i+j}_{\rm sm}(\Xcal, \Gm).
\]

Note that $R^1 i\pi_* \Gm$ is the smooth-Nisnevich sheafification of the Picard group, which is zero on a smooth stack because the Picard group of a local regular ring is zero. Then if we prove that ${\rm H}^i_{\textnormal{sm-Nis}}(\Xcal, \Gm)=0$ for $i>1$ we have proven our claim, as the smooth and \'etale cohomology of $\Gm$ coincide on $\Xcal$.

Now we use the fact that $\Xcal = \left[X/G\right]$. Let $X'=(X \times U)/G$ be an equivariant approximation of $\Xcal$, and consider the Leray-Grothendieck spectral sequence induced by the projection $X' \xrightarrow{f} \Xcal$
\[
{\rm H}^i_{\textnormal{sm-Nis}}(\Xcal, R^j f_*\Gm) \Rightarrow {\rm H}^{i+j}_{\textnormal{sm-Nis}}(X', \Gm).
\]

For any smooth quasi-separated algebraic space $Y$ we have $H^n_{\textnormal{sm-Nis}}(Y,\Gm)=H^n_{\rm Nis}(Y,\Gm)=0$ when $n>1$ \cite{PirAlgStack}*{Proposition 7.5}. This necessarily implies that ${\rm H}^{i}_{\textnormal{sm-Nis}}(X', \Gm)=0$ for $i>1$ and $R^q f_* \Gm = 0$ for $q > 1$. Now observe that $R^1 f_* \Gm = 0$ for the same reason as above, and $R^0 f_* \Gm = \Gm$ as a vector bundle induces an isomorphism on $\mathcal{O}^*$. Thus the spectral sequence collapses and we can conclude.
\end{proof}

Now that we know the cohomological Brauer group is a sheaf in the right topology, all we need is a morphism from the correct cohomology group onto it to conclude.

\begin{thm}
Let $\Xcal=\left[X/G\right]$ be as above. Then we have 
\[
{\rm Br}'(\Xcal)_{p^r}={\rm Inv}^2(\Xcal, {\rm H}^n(-,\underline{\ZZ}/p^r(1))).
\]
\end{thm}
\begin{proof}
Observe that thanks to \Cref{lm:Brsheaf} we have a morphism 
\[
{\rm Inv}^2(\Xcal, {\rm H}^n(-,\underline{\ZZ}/p^r(1)))={\rm H}^0_{\textnormal{sm-Nis}}(\Xcal,{\rm H}^n(-,\underline{\ZZ}/p^r(1))) \to {\rm Br}'(\Xcal)_{p^r}.
\]
As this is a map of sheaves it suffices to prove our claim when $X$ is a smooth scheme. Due to the identification $\underline{\ZZ}(1) = \Gm\left[ -1 \right]$ in the \'etale derived category of $X$ and the exact sequence $0 \to \underline{\ZZ}(1) \xrightarrow{\cdot p^r} \underline{\ZZ}(1) \to \underline{\ZZ}/p^r \to 0$ we get the functorial exact sequence
\[ 
{\rm H}^2(X, \underline{\ZZ}(1))={\rm Pic}(X) \xrightarrow{\cdot p^r} {\rm Pic}(X) \to {\rm H}^2(X, \underline{\ZZ}(1)/p^r)
\]
\[\to {\rm H}^2(X, \Gm) \xrightarrow{\cdot p^r} {\rm H}^2(X, \Gm)={\rm H}^3(X, \underline{\ZZ}(1)).
\]

Passing to the Zariski sheafification allows us to conclude immediately as the Picard group of a smooth scheme is Zariski-locally trivial.
\end{proof}

\section{\'{E}tale motivic cohomology of fields}\label{sec:explicit}
This section is devoted to the study of \'{e}tale motivic cohomology of fields. In particular, we are interested in the explicit description of these groups via (log) differential forms and symbols. We also recall the definition of ramified, tamely ramified and unramified elements in these cohomology groups, and we discuss some related results.
\subsection{Motivic cohomology, log differentials and symbols}
As explained earlier, the functor $F \mapsto {\rm H}^{n}(F,\ZZ/p^r(j))$ is nonzero only for $j = n,n-1$. When $r=1$ we have an explicit description of the functor, but in general the descriptions we gave for $r>1$ are not quite simple enough for computations.

We will introduce another description of the functors ${\rm H}^{n}(F,\ZZ/p^r(j))$ for $j=n,n+1$ in terms of symbols that allow us to do (somewhat) explicit computations. We begin with the simpler case $j=n$ and $r=1$. In this case we have ${\rm H}^{n}(F,\ZZ/p(n)) \simeq \Omega^n_{F,\rm log} \subset \Omega^n_F$, i.e. the subsheaf of forms on $F$ that can be written as

\[
\frac{{\rm d}b_1}{b_1}\wedge\frac{{\rm d}b_2}{b_2}\wedge \ldots \wedge\frac{{\rm d}b_n}{b_n}
\]
with $b_1,\ldots,b_n \in F^*$. 

Let $\lbrace b_1, \ldots, b_n\rbrace$ denote the element $b_1 \otimes \ldots \otimes b_n \in \K_{\textnormal{Mil}}^n$. Recall that Milnor's $n$-th $\K$-theory group $\K^F_{\textnormal{Mil}}(F)$ is defined as the quotient of $(F^*)^{\otimes n}$ by the subgroup generated by elements $\lbrace b_1, \ldots, b_n \rbrace$ such that $b_i + b_j =1$ for some $i \neq j$. Define $\K^n_{p^r}=\K^n_{\textnormal{Mil}}/p^r\K^n_{\textnormal{Mil}}$. Bloch and Kato \cite{BlKaP}*{Corollary 2.8} proved that there is an isomorphism
\[
\K^n_{p^r}(F) \to W_r\Omega^n_{F,{\rm log}} = {\rm H}^{n}(F,\ZZ/p^r(n)) 
\]
When $r=1$ we can explicitly see the map as
\[
\lbrace b_1,\ldots,b_n\rbrace \mapsto \frac{{\rm d}b_1}{b_1}\wedge\frac{{\rm d}b_2}{b_2}\wedge \ldots \wedge\frac{{\rm d}b_n}{b_n}
\]
Note that the operation on the left is multiplication while on the right it is addition.
We will sometimes shorten the notation $\lbrace b_1,\ldots, b_n\rbrace$ to $\lbrace \underline{b} \rbrace$ when $n$ is clear from context.

We have the following description by Izhboldin for ${\rm H}^{n+1}(F,\ZZ/p(n))$: there is an exact sequence
\[
\Omega^n_{F,\rm log} \to \Omega_F^{n} \xrightarrow{\mathcal{P}} \Omega_F^{n}/{\rm d}\Omega_F^{n-1} \to {\rm H}^{n+1}(F,\ZZ/p(n)) \to 0
\]
where 
\[
\mathcal{P}(a\frac{{\rm d}b_1}{b_1}\wedge\frac{{\rm d}b_2}{b_2}\wedge \ldots \wedge\frac{{\rm d}b_n}{b_n})=(a^p-a)\frac{{\rm d}b_1}{b_1}\wedge\frac{{\rm d}b_2}{b_2}\wedge \ldots \wedge\frac{{\rm d}b_n}{b_n}.
\]
The group of forms $\Omega^{n}_F$ has a description in symbols given by
\[
\Omega^{n}_F = \left[a, b_1, \ldots, b_n\right\rbrace, a \in F, b_i \in F^*
\]
given by the map 
\[
\left[a, b_1,\ldots,b_n\right\rbrace \mapsto a\frac{{\rm d}b_1}{b_1}\wedge\frac{{\rm d}b_2}{b_2}\wedge \ldots \wedge\frac{{\rm d}b_n}{b_n}
\]
subject to the relations that $\left[a, b_1, \ldots, b_i \right\rbrace=0$ if $b_i=b_j$ for some $i\neq j$ and 
\[
\left[u+v, u+v, \ldots, b_n \right\rbrace = \left[u, u, \ldots, b_n \right\rbrace + \left[v, v, \ldots, b_n \right\rbrace
\]
for $u, v, u+v \in F^*$. With the extra relations coming from the exact sequence above, we get the following description of ${\rm H}^{n+1}(F,\ZZ/p(n))$:
\[
{\rm H}^{n+1}(F,\ZZ/p(n)) = F \otimes (F^*)^{\otimes n} / J
\]
where $J$ is the subgroup generated by elements in the forms:
\begin{itemize}
    \item $\left[a, b_1,\ldots,b_n\right\rbrace$ where $b_i=b_j$ for some $i\neq j$
    \item $\left[a^p-a, b_1,\ldots,b_n\right\rbrace $
    \item $\left[a, a, b_2, \ldots,b_n\right\rbrace$ where $a \in F^*$
\end{itemize}

Following Izhboldin's notation we well write ${\rm H}^{n+1}_p$ for this functor. There are natural pairings ${\rm H}_p^{j+1}(F) \otimes \K_p^i(F) \to {\rm H}_p^{i+j+1}(F) $ given by
\[ 
\left[a, b'_1,\ldots,b'_j\right\rbrace \otimes \lbrace b_1,\ldots,b_i\rbrace \mapsto \left[a, b'_1,\ldots,b'_j, b_1, \ldots, b_i\right\rbrace.
\]
A description of the cohomology groups and ${\rm H}^{n+1}(F,\ZZ/p^r(n))$ for $r>1$ requires some additional work. Given an $\mathbb{F}_p$-algebra $A$, the group of Witt vectors \cite{Ill}*{Section 1} of length $r$ on $A$, denoted $W_r(A)$, is a group whose underlying set is $A^r$, and whose group structure is defined in terms of a sequence of \emph{universal polynomials} depending on the characteristic. A few things to keep in mind are:

\begin{itemize}
    \item $W_1(A) = A$
    \item The maps $(a_1,\ldots, a_{r}) \mapsto (0,\ldots,0,a_1,\ldots, a_r)$, $(a_1, \ldots, a_r) \mapsto (a_1, \ldots, a_{r-i})$ are group homomorphisms.
    \item We have $p(a_1,\ldots, a_r)=(0,a_1^p,\ldots, a_{r-1}^p)$
\end{itemize}

Using Witt vectors Izhboldin gives the following description of ${\rm H}^n(F,\ZZ/p^r(n))$ and ${\rm H}^{n+1}(F,\underline{\ZZ}/p^r(n))$ (see \cite{IzhK}*{Corollary 6.5}): write $\left[ a_1,\ldots, a_r, b_1, \ldots, b_n \right\rbrace$ for the element $(a_1,\ldots a_r) \otimes b_1 \otimes \ldots \otimes b_n$ in $W_r(F)\otimes (F^*)^{\otimes n}$, and define
\[
Q^n(F, r) = W_r(F) \otimes (F^*)^{\otimes n} / I
\]
where $I$ is the ideal generated by elements of the form
\begin{itemize}
    \item $\left[a_1, \ldots a_r, b_1,\ldots,b_n\right\rbrace$ where $b_i=b_j$ for some $i\neq j$,
    \item $\left[0, \ldots, 0, a, 0, \ldots, 0, a, b_2, \ldots,b_n\right\rbrace$ where $a \in F^*$.
\end{itemize}
By \cite{IzhK}*{Theorem C} we have an isomorphism
\[ \K^n_{p^r}(F) \simeq \ker(Q^n(F,r) \overset{\mathcal{P}}{\longrightarrow} Q^n(F, r)) \]
where
\[\mathcal{P}(\left[a_1, \ldots a_r, b_1,\ldots,b_n\right\rbrace)\overset{\textnormal{def}}{=}\left[a_1^p, \ldots a_r^p, b_1,\ldots,b_n\right\rbrace - \left[a_1, \ldots a_r, b_1,\ldots,b_n\right\rbrace.\]
As $\K^n_{p^r}(F)\simeq {\rm H}^n(F,\ZZ/p^r(n))$, we can regard the latter as a subgroup of $Q^n(F,r)$. The group ${\rm H}^{n+1}(F,\ZZ/p^r(n))$ can be described as the cokernel of this map, that is
\[
{\rm H}^{n+1}(F,\underline{\ZZ}/p^r(n)) \simeq W_r(F) \otimes (F^*)^{\otimes n} / J
\]
where $J$ is the subgroup generated by elements in the forms:
\begin{itemize}
    \item $\left[a_1, \ldots a_r, b_1,\ldots,b_n\right\rbrace$ where $b_i=b_j$ for some $i\neq j$,
    \item $\left[a_1^p, \ldots a_r^p, b_1,\ldots,b_n\right\rbrace - \left[a_1, \ldots a_r, b_1,\ldots,b_n\right\rbrace$,
    \item $\left[0, \ldots, 0, a, 0, \ldots, 0, a, b_2, \ldots,b_n\right\rbrace$ where $a \in F^*$.
\end{itemize}

We will call these functors ${\rm H}^{n+1}_{p^r}$. We will sometimes shorten the notation for a Witt vector $\left[ a_1, \ldots, a_r \right]$ to $\left[ \underline{a} \right]$, and write $\left[ \underline{a}, \underline{b} \right\rbrace$ for $\left[ a_1,\ldots,a_r,b_1,\ldots,b_n\right\rbrace$.

\subsection{Torsion in ${\rm H}^{n+1}_{p^r}(F)$} 
We prove here some technical results on ${\rm H}^{n+1}_{p^r}(F)$ and his torsion subgroups. The main ingredient will be the following exact sequence due to Izhboldin.

\begin{lm}[\cite{Izh}*{Lemma 6.2}]\label{lm:truncation}
For $1\leq s < r$ there is a short exact sequence
\[ 0\longrightarrow {\rm H}^{n+1}_{p^s}(F) \overset{\iota_s}{\longrightarrow} {\rm H}^{n+1}_{p^r}(F) \overset{\pi_s}{\longrightarrow} {\rm H}^{n+1}_{p^{r-s}}(F)\longrightarrow 0, \]
where the homomorphisms are defined as 
\begin{align*}
    \iota_{s}\left[a_1,\ldots,a_s,b_1,\ldots,b_n\right\rbrace &= \left[0,\ldots,0,a_1,\ldots,a_s,b_1,\ldots,b_n\right\rbrace, \\
    \pi_{s}\left[a_1,\ldots,a_r,b_1,\ldots,b_n\right\rbrace &= \left[a_1,\ldots,a_{r-s},b_1,\ldots,b_n\right\rbrace.
\end{align*}
\end{lm}

The image of $\iota_s$ is clearly of $p^s$-torsion, and in fact as the following proposition and corollary show it is exactly equal to the $p^s$-torsion subgroup in ${\rm H}^{n+1}_{p^r}(F)$.

\begin{prop}\label{prop:ptors}
Let $F$ be a field of characteristic $p>0$. Then the subgroup of elements of $p$-torsion in ${\rm H}^{n+1}_{p^r}(F)$ is isomorphic to ${\rm H}^{n+1}_{p}(F)$, embedded via the homomorphism $\left[ a,b_1,\ldots,b_n\right\rbrace\mapsto \left[0,\ldots,0,a,b_1,\ldots,b_n\right \rbrace$.
\end{prop}
\begin{proof}
We argue by induction on $r$. For $r=1$, the statement is obvious. For $r=2$, define 
\[
\alpha = \sum_i \left[a_{i},a_{2,i},\underline{b}_i\right\rbrace.
\] 
Then $p\alpha = \sum_i \left[0, a_{i}^p,\underline{b}_i\right\rbrace$. Now note that taking the inverse image in ${\rm H}^{n+1}_p(F)$:
\[
\sum_i (\left[a_i^p,\underline{b}_i\right\rbrace-\left[a_i,\underline{b}_i\right\rbrace) = 0 \Rightarrow \sum_i \left[a_i,\underline{b}_i\right\rbrace = 0 \Leftrightarrow \sum_i \left[a_i^p,\underline{b}_i\right\rbrace = 0
\]
which proves the statement. Now let $r$ be greater than $2$, and let 
$\alpha =\sum_i \left[\underline{a}_i, \underline{b}_i\right\rbrace$
be $p$-torsion. Then $\pi_{r-1}\alpha $ is $p$-torsion.

By the inductive hypothesis, we know that 
\[
\pi_{r-1}\alpha = \sum_i \left[0, \ldots, 0, a'_i ,\underline{b}'_i\right\rbrace
\Leftrightarrow 
\alpha= \sum_i \left[0, \ldots, a'_i, 0 ,\underline{b}'_i\right\rbrace + \iota_1 \beta
\]
for some $\beta \in {\rm H}^n_{p}(F)$. 

This shows that $\alpha$ belongs to $\iota_2({\rm H}^{n+1}_{p^2}(F))$, so we can conclude by the case $r=2$.
\end{proof}

\begin{cor}\label{cor:pstors}
The subgroup of elements of $p^s$-torsion in ${\rm H}^{n+1}_{p^r}(F)$ is isomorphic to ${\rm H}^{n+1}_{p^s}(F)$, embedded via the homomorphism
\[
\left[a_1,\ldots, a_s,b_1,\ldots,b_n\right\rbrace\mapsto \left[0,\ldots,0,a_1, \ldots, a_s,b_1,\ldots,b_n\right \rbrace.
\]
\end{cor}
\begin{proof}
We proceed by induction on $s$. An element $\alpha \in {\rm H}^{n+1}_{p^r}(F)$ is of $p^s$ torsion if and only if $p^{s-1}\alpha$ is of $p$ torsion, which by \Cref{prop:ptors} means that $\pi_{r-1}(\alpha)$ is of $p^{s-1}$-torsion, because the $p$-torsion of ${\rm H}_{p^r}^{n+1}(F)$ maps to zero in ${\rm H}_{p^{r-1}}^n(F)$. By the inductive hypothesis that means that we can write
\[
\alpha = (\sum_i \left[ 0, \ldots, 0, \underline{a}'_i, 0, \underline{b}_i \right\rbrace) + \iota_1 \beta
\]
with $\underline{a}'_i \in F^{s-1}$, which allows us to conclude imemdiately.
\end{proof}

\subsection{Unramified, tamely ramified and wildly ramified elements}\label{subsec:ramif}
The functor $\K_{p^r}$ has a natural notion of ramification at a point $x \in X^{(1)}$. If $\pi$ is a uniformizer for the DVR $\mathcal{O}_{X,x}$, define:
\[
\partial_x: \K^n_{p^r}(\bfk(X)) \to \K^n_{p^r}(\bfk(x)), \quad \lbrace \pi, b_2, \ldots, b_n\rbrace = \lbrace \overline{b_2}, \ldots , \overline{b_n}\rbrace
\]
where $\overline{b_i}$ is the image of $b_i$ in $\bfk(x)$, and $\partial_x\lbrace b_1,\ldots,b_n\rbrace=0$ if $b_i \in \mathcal{O}^*_{X,x}$ for all $i$. We have the following description of $\K_{p^r}^{\bullet}(\bfk(t))$, from \cite{GilSza}*{Chapter 7}, which shows that Milnor's $\K$-theory mod $p^r$ behaves the way we expect based on the mod $\ell$ case.

\begin{equation}\label{eq:KA1}
0 \to \K^{\bullet}_{p^r}(\bfk)\to \K^{\bullet}_{p^r}(\bfk(t)) \to \bigoplus_{x \in (\bA^1)^{(1)}} \K^{\bullet-1}_{p^r}(\bfk(x)) \to 0.
\end{equation}

The ramification map $\partial_v$ turns out to be equal to the one in \cite{GrSu}*{Corollary 1.6} over perfect fields, and we will see later (Section \ref{sub:CohDVR}) that we can use it to compute ramification in \Cref{eq:ramification}.

On the other hand, there is no obvious residue map for the functor ${\rm H}^{n+1}_{p^r}$ and looking at \Cref{eq:ramification} a problem is immediately apparent: the differential $\partial_x$ does not map to ${\rm H}^{n}_{p^r}(\bfk(x))$ as one expects from the mod $\ell$ case, but instead to a completely different group, the relative cohomology group ${\rm H}^{n+2}_{{\textnormal{\'et}},x}(X,\ZZ/p^r(n))$, and the identification of the two only works with mod $\ell$-coefficients. 

To get a map into ${\rm H}^{n}_{p^r}(\bfk(x))$ suiting our needs we will have to restrict to an appropriate subgroup, the subgroup of \emph{tamely ramified} elements.

First, let us recall the notion of tamely ramified field extensions.
Let $(R,v)$ be a DVR, with fraction field $F$, residue field $\bfk_v$ and maximal ideal $\mathfrak{m}_v$. Let $F\subset L$ be a finite and separable extension, and denote $\mathfrak{m}_i$ the maximal ideals of the integral closure $\overline{R}$ of $R$ in $L$ (there are only finitely many of these maximal ideals).
Then $F\subset L$ is \emph{tamely ramified} at $v$ if either $\bfk_v$ has characteristic zero or $\bfk_v$ has characteristic $p$, the extensions $\bfk_v \subset \overline{R}/\mathfrak{m}_i$ are separable and the ramification indices of $\mathfrak{m}_v$ in $\overline{R}$ are prime to $p$. 

\begin{df}
let $F$ be a field and $v$ a valuation on it. Define $F^{\rm tm}$ as the subfield of $F^{\rm sep}$ generated by extensions that are tamely ramified at $v$. An element $\alpha \in {\rm H}_{p^r}^{\bullet}(F)$ is tamely ramified at $v$ if and only if $\alpha_{F^{\rm tm}} = 0$. We say $\alpha$ is wildly ramified otherwise. 

Given $(F,v)$ as above we write ${\rm H}^{\bullet}_{p^r}(F)_{\rm tm}$ for the subgroup of tamely ramified elements, and $\widetilde{{\rm H}}^{\bullet}_{p^{r}}(F)$ for the wild quotient ${\rm H}^{\bullet}_{p^r}(F)/{\rm H}^{\bullet}_{p^r}(F)_{\rm tm}$.
\end{df}

Note that by \cite{Endl}*{Chapter III, 17.19 and Section 22} given a DVR $(R,v)$ the quotient field of the strict Henselization $\bfk(R^{\rm sh})$ is contained in $\bfk(R)^{\rm tm}$.

For $r=1$, Izhboldin shows that the subgroup of tamely ramified elements is generated by symbols $\left[a, b_1, \ldots, b_n\right\rbrace$ where $v(a)\geq0$, i.e. $a \in \mathcal{O}_v$. The following proposition extends this result to general $r$.

\begin{prop}
Let $(F,v)$ be as above. An element $\alpha \in {\rm H}^{n+1}_{p^r}(F)$ is tamely ramified at $v$ if and only if $\alpha$ can be written as a sum of elements in the form $\left[ a_1, \ldots, a_r , b_1, \ldots, b_n \right \rbrace $ with $v(a_i) \geq 0$ for all $i$.
\end{prop}
\begin{proof}
First we observe that the elements as above form a subgroup of ${\rm H}^{n+1}_{p^r}(F)$ corresponding to the image of $W_r(\mathcal{O}_v) \otimes \K^n_{p_r}(F)$. The case $r=1$ is true by \cite{Izh}*{Corollary 2.6}. We proceed by induction. Assume the statement is true for $r-1$.

First let $\alpha = \left[ a_1, \ldots, a_r , b_1, \ldots, b_n \right \rbrace$ and assume that $v(a_i) \geq 0$ for all $i$. We want to show that $\alpha$ is tamely ramified. The extension $F'/F$ obtained by adding a root of $x^p-x-a_1$ to $F$ is unramified, and $\alpha_{F'} = \left[ 0, a'_2 \ldots, a'_{r} , b_1, \ldots, b_n \right \rbrace$ for some $a'_2, \ldots, a'_r \in \mathcal{O}_{v'}$. Repeating the process up to $r-1$ more times we see that $\alpha$ becomes zero after passing to a finite unramified extension.

Now let $\alpha = \sum_i \left[ \underline{a}_i , \underline{b}_i \right \rbrace$ be a tamely ramified element and let the morphisms $\pi_s$ and $\iota_s$ be defined as in \Cref{lm:truncation}. Then $\pi_{r-1}(\alpha)$ is tamely ramified, which means there is an element $\alpha'=\sum_j\left[ a'_j, \underline{b}'_j \right \rbrace \in {\rm H}^{n+1}_{p}(F)$ which is equal to $\pi_1(\alpha)$ and such that $v(a'_1) \geq 0$. But then $\alpha = \sum_j \left[ a'_j, 0, \ldots, 0 , \underline{b}'_j \right \rbrace + \alpha'$ where $\alpha'$ is a tamely ramified element coming from ${\rm H}^{n+1}_{p^{r-1}}(F)$, and we can conclude by the inductive hypothesis.
\end{proof}

Izhboldin \cite{Izh}*{Corollary 2.6, Proposition 6.6} completely described the tame part of ${\rm H}^{n+1}_{p^r}(F)$ for a complete discrete valuation field $(F,v)$. Totaro \cite{Totp}*{Theorem 4.3} showed that for $r=1$ the same description extends to Henselian fields. Here we show that the description of the \emph{tame} part of ${\rm H}^{n+1}_{p^r}(F)$ (trivially) works for Henselian fields for any $r$.

\begin{prop}\label{prop:DVR}
Let $(R,v)$ be an Henselian DVR. We have a split exact sequence
\[
0 \to {\rm H}^{n+1}_{p^r}(\bfk_v)  \xrightarrow{j^*} {\rm H}^{n+1}_{p^r}(\bfk(R))_{\rm tm} \xrightarrow{\partial_v} {\rm H}^{n}_{p^r}(\bfk_v) \to 0.
\]
We can see $j^*$ as the composition ${\rm H}^{n+1}_{p^r}(\bfk_v)={\rm H}^{n+1}(R,\ZZ/p^r(n)) \to {\rm H}^{n+1}_{p^r}(\bfk(R))$. It is defined by $\left[ a_1, \ldots, a_r , b_1, \ldots, b_n \right \rbrace  \mapsto \left[ \tilde{a}_1, \ldots, \tilde{a}_r , \tilde{b}_1, \ldots, \tilde{b}_n \right \rbrace $, where $\tilde{\varphi}$ is any lifting of $\varphi$ to $R$.

The residue map $\partial_v$ is $\left[ a_1, \ldots, a_r , \pi, b_2, \ldots, b_n \right \rbrace \mapsto \left[ \overline{a}_1, \ldots, \overline{a}_r , \overline{b}_2, \ldots, \overline{b}_n \right \rbrace$, where $\overline{\varphi}$ is the image of $\varphi$ in $\bfk_v$.
\end{prop}
\begin{proof}
The case $r=1$ is proven in \cite{Totp}*{Thm 4.3}, and the general case can be proven by induction exactly as in \cite{Izh}*{Proposition 6.6}, by considering the commutative diagram with exact rows
\[
     \xymatrix@C=1em{
      & 0 \ar[d] & 0 \ar[d] & 0 \ar[d] & \\
     0 \ar[r] & {\rm A}_{r-1}^{n+1} \ar[r] \ar[d] & {\rm A}_{r}^{n+1} \ar[r] \ar[d] & {\rm A}_{1}^{n+1} \ar[r] \ar[d] & 0 \\
     0 \ar[r] & {\rm H}^{n+1}_{p^{r-1}}(\bfk(R)) \ar[r]^{\iota_{r-1}} \ar[d] & {\rm H}^{n+1}_{p^{r}}(\bfk(R)) \ar[r]^{\pi_1} \ar[d] & {\rm H}^{n+1}_{p}(\bfk(R))  \ar[d] \ar[r] & 0 \\
    0 \ar[r] & \widetilde{{\rm H}}^n_{p^{r-1}}(\bfk(R)) \ar[r] & \widetilde{{\rm H}}^n_{p^{r}}(\bfk(R)) \ar[r] & \widetilde{{\rm H}}^n_{p^{r}}(\bfk(R)) \ar[r] & 0 
        }
 \]
where $A^{n+1}_{i}={\rm H}^{n+1}_{p^i}(\bfk_v) \oplus {\rm H}^{n}_{p^i}(\bfk_v)$. The left and right columns are exact by the inductive hypothesis, which implies the same for the central one. The map 
\[
\left[ a_1, \ldots, a_r , b_1, \ldots, b_n \right \rbrace  \mapsto \left[ \tilde{a}_1, \ldots, \tilde{a}_r , \tilde{b}_1, \ldots, \tilde{b}_n \right \rbrace 
\]
can be shown to be well defined exactly as in \cite{Totp}*{Theorem 4.3}, because if ${\rm m} \subset R$ is the maximal ideal, any element $a \in {\rm m}$ can be written as $u^p - u$ for some $u \in \bfk(R)$. It is an isomorphism by \Cref{prop:Hunr}.
\end{proof}

\subsection{Izhboldin's description of wild ramification.}
Coming back to the wild part of motivic cohomology, Izhboldin and Totaro give a nice description of $\widetilde{{\rm H}}^{n+1}_p(F)$, which we are going to recall. As before, we denote with $(F,v)$ a discrete valuation field, with associated DVR $R$ and residue field $\bfk_v$. 

There is a filtration 
\[ 0\subset {\rm U}_{0}\subset {\rm U}_1 \subset \ldots\subset {\rm U}_{i-1}\subset {\rm U}_{i}\subset\ldots\subset {\rm H}^{n+1}_p(F) \]
where
\[
{\rm U}_i = \langle \left[a,b_1,\ldots,b_n\right\rbrace \mid v(i) \geq -i \rangle.
\]
The graded pieces ${\rm U}_i/{\rm U}_{i-1}$ of this filtration admit a description in terms of forms \emph{over the residue field}.

Indeed, for $p\nmid i$, there is a morphism
\[ \psi_i:\Omega_{\bfk_v}^n \longrightarrow {\rm U}_i/{\rm U}_{i-1} \]
defined as follows: given an element $\alpha=a\frac{{\rm d}b_1}{b_1}\wedge\ldots\wedge\frac{{\rm d}b_n}{b_n}$ in $\Omega_{\bfk_v}^{n}$, pick liftings $a',b_1',\ldots,b'_n \in R$ of $a,b_1,\ldots,b_n$; then we define $\psi_i(\alpha)$ as $[t^{-i}a',b'_1,\ldots,b'_n\rbrace$ in ${\rm U}_i/{\rm U}_{i-1}$. 

This map is well defined because if we pick a second lifting, the associated element in ${\rm U}_i/{\rm U}_{i-1}$ differs from the previous one by a form in ${\rm U}_{i-1}$, which is of course zero in the quotient.

For $p\mid i$, let $Z^n\subset \Omega_{\bfk_v}^n$ be the subgroup of closed forms. We define
\begin{align*}
    \psi_i:\Omega^n_{\bfk_v}/Z^n\oplus \Omega_{\bfk_v}^{n-1}/Z^{n-1} \rightarrow {\rm U}_i/{\rm U}_{i-1} 
\end{align*}
\begin{align*}
    (a\frac{{\rm d}b_1}{b_1}\wedge\ldots\wedge\frac{{\rm d}b_n}{b_n},0)&\longmapsto [t^{-i}a',b'_1,\ldots,b'_n\rbrace\\
    (0,a\frac{{\rm d}b_2}{b_2}\wedge\ldots\wedge\frac{{\rm d}b_n}{b_n})&\longmapsto [t^{-i}a',t,b'_2,\ldots,b'_n\rbrace
\end{align*}

As before, the $a',b'_1,\ldots,b'_n$ denote liftings of $a,b_1,\ldots,b_n$ to $R$.
Again, picking a different lifting changes the associated element in ${\rm U}_{i}/{\rm U}_{i-1}$ by a form in ${\rm U}_{i-1}$, so the map is well defined.

\begin{thm}[\cite{Izh}*{Theorem 2.5} and \cite{Totp}*{Theorem 4.3}]\label{thm:wild pres}
Let $F$ be a field with a discrete valuation $v$, let $\bfk_v$ be the residue field of the corresponding DVR, and let $0\subset {\rm U}_0\subset {\rm U}_1\subset\ldots\subset {\rm H}^{n+1}_p(F)$ be the filtration defined above. 

Then the subgroup ${\rm U}_0$ is the tame subgroup ${\rm H}_{p}^{n+1}(F)_{\rm tm}$, and the maps $\psi_i$ that we defined before induce isomorphisms
\[
{\rm U}_i/{\rm U}_{i-1} \simeq \begin{cases}
\Omega^n_{\bfk_v}\quad \mbox{if } \, p \nmid i \\
\Omega^n_{\bfk_v}/Z_n \oplus \Omega^{n-1}_{\bfk_v}/Z_{n-1} \quad \mbox{if } \, p \mid i 
\end{cases}
\]
where $Z_n \subset \Omega^n_{\bfk_v}$ is the subgroup of closed forms.
\end{thm}

Therefore, the group $\widetilde{{\rm H}}^{n+1}_p(F)$ is isomorphic to ${\rm H}^{n+1}_p(F)/{\rm U}_0$ and it has a filtration whose associated graded pieces ${\rm U}_i/{\rm U}_{i-1}$ are described by the theorem above. 

Nonetheless, given $\alpha \in {\rm U}_i$, understanding its image in ${\rm U}_{i}/{\rm U}_{i-1}$ and extrapolating information on its class in $\widetilde{{\rm H}}^{n+1}_p(F)$ may be highly nontrivial.

Luckily, when $\alpha$ is one of the standard generators of ${\rm U}_i/{\rm U}_{i-1}$, the situation gets much easier, as shown in the following corollary:

\begin{cor}\label{cor:lifting}
Let $(F,v)$ be as above and let $t$ be a uniformizer for $v$. Assume moreover that $\bfk_v=\bfk$.
\begin{enumerate}
    \item if $p\nmid i$, for every $\varphi\in\Omega^n_\bfk$ we have $t^{-i}\varphi=0$ in ${\rm U}_i/{\rm U}_{i-1} \Rightarrow t^{-i}\varphi=0$ in ${\rm H}^{n+1}_p(F)$;
    \item if $p\mid i$, for every $\varphi\in\Omega^n_\bfk$ we have $t^{-i}\varphi=0$ in ${\rm U}_i/{\rm U}_{i-1} \Rightarrow t^{-i}\varphi=t^{-i/p}\varphi'$ in ${\rm H}^{n+1}_p(F)$ for some $\varphi'\in \Omega_\bfk^n$;
    \item if $p \mid i$, for every $\psi\in \Omega^{n-1}_\bfk$ we have $t^{-i}\frac{{\rm d}t}{t}\wedge \varphi =0$ in ${\rm U}_i/{\rm U}_{i-1} \Rightarrow t^{-i}\frac{{\rm d}t}{t}\wedge\varphi=t^{-i/p}\frac{{\rm d}t}{t}\wedge\varphi'$ for some $\varphi'\in \Omega_{n-1}^{\bfk}$. If we further assume that $p^2\nmid i$, then we can replace $\frac{{\rm d}t}{t}\wedge \varphi'$ in the expression above with $\varphi''\in \Omega_\bfk^n$.
\end{enumerate}
\end{cor}

\begin{rmk}
The hypothesis that $\bfk_v=\bfk$ in the statement above is of key importance, because it provides a canonical way to lift elements from $\Omega_{\bfk_v}^n$ to $\Omega^n_F$ using the embedding $\bfk_v=\bfk\hookrightarrow F$.

Without this hypothesis, we would not be able to define an element in $\Omega^n_F$, because the maps that induce isomorphisms of \Cref{thm:wild pres} only lift to ${\rm U}_i$ up to elements of ${\rm U}_{i-1}$.
\end{rmk}

\begin{proof}[Proof of \Cref{cor:lifting}]
Point (1) above follows from the fact that $\psi_i:\Omega_{\bfk_v}^n\to {\rm U}_i/{\rm U}_{i-1}$ is an isomorphism, hence if $\psi_i(\varphi)=0$ we must have $\varphi=0$.

If $p \mid i$ then $t^{-i}\varphi =0 \in {\rm U}_{i}/{\rm U}_{i-1}$ together with \Cref{thm:wild pres} implies that $\varphi$ is closed. By Cartier's theorem \cite{Izh}*{Lemma 1.5.1} this implies that $\varphi$ is a sum of exact forms and forms that can be written as $a^p\frac{{\rm d}b_1}{b_1}\wedge\ldots\wedge\frac{{\rm d}b_n}{b_n}$. If $\tau = {\rm d}\tau'\in \Omega_{\bfk}$ is exact, then $t^{-i}\tau = {\rm d}(t^{-i}\tau')$ so $t^{-i}\tau = 0 \in {\rm H}^{n+1}_p(F)$. On the other hand, 
\[
(t^{-i} a^p-t^{-i/p}a)\frac{{\rm d}b_1}{b_1}\wedge\ldots\wedge\frac{{\rm d}b_n}{b_n} = 0 \in {\rm H}_p^{n+1}(F) 
\]
which immediately implies
\[
(t^{-i} a^p)\frac{{\rm d}b_1}{b_1}\wedge\ldots\wedge\frac{{\rm d}b_n}{b_n}=t^{-i/p}a\frac{{\rm d}b_1}{b_1}\wedge\ldots\wedge\frac{{\rm d}b_n}{b_n} \in {\rm U}_{i/p}.
\]
This proves (2).

To prove (3), proceeding as in (2), we easily get that $t^{-i}\frac{{\rm d}t}{t}\wedge\varphi=t^{-i/p}\frac{{\rm d}t}{t}\wedge \varphi'$. Indeed, if $\varphi={\rm d}\tau$, then 
\[{\rm d}(t^{-i-1}dt\wedge\tau)=t^{-i}\frac{{\rm d}t}{t}\wedge {\rm d}\tau + (i-1)t^{-i-2}dt\wedge dt\wedge \tau, \]
and both the left hand side of the equality and the last term of the sum on the right hand side of the equality are zero in ${\rm H}^{n+1}_p(F)$. If $\varphi=a^p\frac{{\rm d}b_1}{b_1}\wedge\ldots\wedge\frac{{\rm d}b_n}{b_n}$, then the same argument used to prove (2) shows that $t^{-i}\frac{{\rm d}t}{t}\varphi=t^{-i/p}\frac{{\rm d}t}{t}\wedge\varphi'$.
If we further assume that $p^2\nmid i$, then we have that
\begin{align*} {\rm d}(at^{-i/p}\wedge\frac{{\rm d}b_1}{b_1}\wedge\ldots\wedge\frac{{\rm d}b_n}{b_n}) = &-\frac{i}{p} a t^{-i/p} \frac{{\rm d}t}{t}\wedge \frac{{\rm d}b_1}{b_1}\wedge\ldots\wedge\frac{{\rm d}b_n}{b_n} \\
&+ t^{-i/p} \frac{{\rm d}a}{a}\wedge \frac{{\rm d}b_1}{b_1}\wedge\ldots\wedge\frac{{\rm d}b_n}{b_n},\end{align*}
which implies that in ${\rm H}^{n+1}_p(F)$ we have
\[ t^{-i/p}\frac{{\rm d}t}{t}\wedge \varphi' = (-i/p)^{-1}t^{-i/p} \frac{{\rm d}a}{a}\wedge \frac{{\rm d}b_1}{b_1}\wedge\ldots\wedge\frac{{\rm d}b_n}{b_n},\]
where $i/p$ is invertible by hypothesis. This proves (3).
\end{proof}

Another consequence of \Cref{thm:wild pres} is that every element in ${\rm H}_{p}^{n+1}(\bfk(t))$ can, up to subtracting a tamely ramified element, be written in a standard form.

\begin{lm}\label{lm:simple form}
Let $(F,v)$ be a field with a discrete valuation such that $\bfk_v=\bfk$, and let $\beta$ be an element of ${\rm H}_{p}^{n+1}(F)$. Let $t$ be a uniformizer for $v$. Then we have
\begin{align*} \beta &= \sum_{i=0}^{m} t^{-i} \varphi_i + t^{-i} \frac{{\rm d}t}{t}\wedge\varphi'_i + \beta_{{\rm tm}}\\
&= \sum_{i=0}^{m} [ c_i/t^i,\underline{f_i} \} + [ c'_i/t^i,t,\underline{f'_i} \} + \beta_{{\rm tm}} \\
 \end{align*}
where $\varphi_i\in \Omega_\bfk^n$, the $\varphi_i'\in \Omega_\bfk^{n-1}$ and $\beta_{{\rm tm}}$ is tamely ramified. Moreover, we can pick $\varphi'_i=0$ if $p\nmid i$ and $\varphi, \varphi'$ not closed if $p \mid i$.
\end{lm}
\begin{proof}
Assume that $\beta \in {\rm U}_m$. By \Cref{thm:wild pres} there are $\varphi_m, \varphi'_m$ respectively in $\Omega^n_{\bfk}$ and $\Omega^{n-1}_{\bfk}$ such that
\[
\beta = t^{-m}\varphi + t^{-m}\frac{{\rm d}t}{t}\wedge \phi' \in {\rm U}_m/{\rm U}_{m-1},
\]
so that the element $\beta-t^{-m}\varphi - t^{-m}\frac{{\rm d}t}{t}\wedge \phi'$ belongs to ${\rm U}_{m-1}$. If $p \mid m$ we can pick $\varphi, \varphi'$ not closed, while if $p\nmid m$ we can pick $\varphi'=0$.

Repeating this process up to $m-1$ times we find that

\[
\beta - \sum_{i\leq m,\, p\mid i} \left(t^{-i}\varphi + t^{-i}\frac{{\rm d}t}{t}\wedge \phi'\right) - \sum_{i \leq m,\, p\nmid i} t^{-i}\varphi 
\]
belongs to ${\rm U}_0$, proving our claim.

\end{proof}

Finding a presentation similar to \Cref{thm:wild pres} for $r>1$ would probably be rather challenging, but the following proposition describes a set of generators for $\widetilde{{\rm H}}^n_{p^r}(F)$ and shows that to check if an element is wildly ramified we can reduce to the case $r=1$.

\begin{prop}\label{prop:wild gen}
Let $(F,v)$ be a field with a discrete valuation. The group 
\[\widetilde{{\rm H}}^{n+1}_{p^r}(F)={\rm H}^{n+1}_{p^r}(F)/{\rm H}^{n+1}_{p^r}(F)_{\rm tm}\]
is generated by elements in the form 
\[
\left[0,\ldots,0,a_i, \ldots, a_r, b_1, \ldots, b_n \right\rbrace
\]
where $\left[a_i, b_1, \ldots, b_n \right\rbrace \in {\rm H}^n_p(F)$ is wildly ramified.

Moreover, given a wildly ramified element $\alpha$ we can always write
\[
\alpha = \iota_i \alpha' + \alpha_{\rm tm}
\]
where $\alpha_{\rm tm}$ is tamely ramified and $\pi_{i-1}\alpha' \in {\rm H}^{n+1}_p(F)$ is wildly ramified.
\end{prop}
\begin{proof}
To prove our first claim we proceed by induction on $r$. The case $r=1$ is trivial. Now assume the claim for $r-1$. Write $\alpha = \left[a_1, \ldots, a_r, b_1, \ldots, b_n \right\rbrace$. If $\left[a_1, b_1, \ldots, b_n \right\rbrace$ is wildly ramified there's nothing to prove. 

If $\left[a_1, b_1, \ldots, b_n \right\rbrace$ is tamely ramified we must have 
\[
\left[a_1, b_1, \ldots, b_n \right\rbrace = \mathop{\Sigma}_{s} \left[{a}_{1,s}, {b}_{1,s}, \ldots, {b}_{n,s}\right\rbrace
\]
with $v({a}_{1,s}) \geq 0$. Now let 
\[
\tilde{\alpha}=\mathop{\Sigma}_{s} \left[{a}_{1,s}, 0, \ldots, 0, {b}_{1,s}, \ldots, {b}_{n,s}\right\rbrace \in {\rm H}^{n+1}_{p^r}(F)
\] 
then $\alpha -  \tilde{\alpha}$ belongs to the kernel of $\pi_1$, i.e. the image of ${\rm H}^{n+1}_{p^{r-1}}(F)$, and $\tilde{\alpha}$ is clearly tamely ramified, so we can conclude by the inductive hypothesis.

Now take a wildly ramified element $\alpha$ and consider the minimum $s$ such that $\alpha=\iota_s \alpha' + \alpha_{\rm tm}$ with $\alpha_{\rm tm}$ a tamely ramified element. If we had that $\pi_{1}\alpha'$ is tamely ramified then as above there is a tamely ramified element $\tilde{\alpha}'$ such that $\alpha' - \tilde{\alpha}'$ is in the image of ${\rm H}^{n+1}_{p^{s-1}}(F)$, contradicting the minimality of $s$. 
\end{proof}

An immediate consequence is that checking if an element is tamely ramified or unramified con be done on the Henselization or even completion of $(F,v)$.

\begin{cor}\label{cor:unramcompl}
Let $(F,v)$ be a discretely valued field, and let $\alpha$ be an element of ${\rm H}^{n+1}_{p^r}(F)$. Then $\alpha$ is tamely ramified (resp. unramified) if and only if the same is true for the pullback of $\alpha$ to the Henselization or completion of $(F,v)$.
\end{cor}
\begin{proof}
Let $F'$ be either the Henselization or completion or $F$ at $v$. \Cref{prop:wild gen} shows that $\alpha$ is wildly ramified if a certain element $\alpha' \in {\rm H}^{n+1}_p(F)$ is, and \Cref{thm:wild pres} shows that the pullback $\widetilde{{\rm H}}^{n+1}_{p}(F) \to \widetilde{{\rm H}}^{n+1}_{p}(F')$ is an isomorphism, so $\alpha$ is wildly ramified on $F$ if and only if it is on $F'$.

Now assume that $\alpha$ is tamely ramified. The differential $\partial_v$ clearly commutes with passing to $F'$ and as $\bfk_v=\bfk_{v'}$ we have that $\partial_v(\alpha)=0 \Leftrightarrow \partial_{v'}(\alpha)=0$.
\end{proof}

\section{\'{E}tale motivic cohomology of DVRs}\label{sub:CohDVR}
Recall that by the work of Kato \cite{Kato} and Izhboldin \cite{IzhK}*{Corollary 6.5} for a field $F$ of characteristic $p>0$ we have isomorphisms
\[ {\rm H}^{n}(F,\underline{\ZZ}/p^r(n))\simeq \K^n_{p^r}(F),\quad {\rm H}^{n+1}(F,\underline{\ZZ}/p^r(n))\simeq {\rm H}^{n+1}_{p^r}(F). \]
With this identification in mind, given a DVR $(R,v)$ with fraction field $\bfk(R)$ and residue field $\bfk_v$, from \Cref{eq:KA1} and \Cref{prop:DVR} we can derive residue homomorphisms
\[ \partial_v\colon {\rm H}^{n}(\bfk(R),\underline{\ZZ}/p^r(n)) \longrightarrow {\rm H}^{n-1}(\bfk_v,\underline{\ZZ}/p^r(n-1)),\textnormal{and}  \]
\[ \partial_v\colon {\rm H}^{n+1}(\bfk(R),\underline{\ZZ}/p^r(n))_{\rm tm} \longrightarrow {\rm H}^{n}(\bfk_v,\underline{\ZZ}/p^r(n-1)).  \]

The main goal of this section is to show that for $i$ and $j$ in this range the image of 
\[
{\rm H}^i(R,\underline{\ZZ}/p^r(j)) \longrightarrow {\rm H}^i(\bfk(R),\underline{\ZZ}/p^r(j))
\]
coincides with the kernel of the residue map $\partial_v$.

\subsection{The case $i=j=n$}
We begin with the easier case of ${\rm H}^n(R,\underline{\ZZ}/p^r(n))$. Define
\[
\mathcal{K}^{\bullet}(R)={\rm Ker}(\partial_v): \K^{\bullet}_{\textnormal{Mil}}(\bfk(R)) \to \K^{\bullet}_{\textnormal{Mil}}(\bfk_v)\textnormal{ and}
\]
\[
\mathcal{K}^{\bullet}_{p^r}(R)={\rm Ker}(\partial_v): \K^{\bullet}_{\textnormal{Mil}}(\bfk(R))/p^r \to \K^{\bullet}_{\textnormal{Mil}}(\bfk_v)/p^r
\]
where $\partial_v$ is the usual residue in Milnor's $\K$-theory. Note that in general $\Kcal^{\bullet}(R)$ is \emph{not} equal to the group $\K^{\bullet}_{\textnormal{Mil}}(R)$ that we obtain by taking the quotient $(R^*)^{\otimes n}/J$, with the subgroup $J$ generated by the same types of elements as for a field.

Now let $D$ be a Dedekind domain which is the semi-localization of a regular finite type algebra over $\bfk$. We  can extend the functor $\Kcal^{\bullet}_{p^r}$ to $D$ by defining  
\[
\Kcal^{\bullet}(D) = \bigcap_v {\rm Ker} (\partial_v) : \K^{\bullet}_{\textnormal{Mil}}(\bfk(D)) \to \K^{\bullet}_{\textnormal{Mil}}(\bfk_v)
\]
and correspondingly
\[
\mathcal{K}^{\bullet}_{p^r}(D)=\bigcap_v {\rm Ker} (\partial_v): \K^{\bullet}_{\textnormal{Mil}}(\bfk(D))/p^r \to \K^{\bullet}_{\textnormal{Mil}}(\bfk_v)/p^r
\]
where $v$ runs among all valuations $\bfk(D)$ that define a maximal ideal of $D$, i.e. the closed points of $\Spec(D)$. 

First, we show that $\Kcal^{\bullet}_{p^r}$ behaves well when changing the exponent $r$.
\begin{cor}\label{lm:exactK}
Let $(R,v)$ be a DVR. There is a natural isomorphism
\[
\mathcal{K}^n(R)/p^r \simeq \mathcal{K}^n_{p^r}(R)
\]
and the sequence
\[
0 \to \Kcal_{p^{r-i}}^{\bullet}(R) \xrightarrow{i_{r-i}} \Kcal_{p^r}^{\bullet}(R) \to \Kcal_{p^i}^{\bullet}(R) \to 0
\]
where the first map is given by $\lbrace b_1, \ldots, b_n \rbrace\mapsto p^{i} \lbrace b_1, \ldots, b_n \rbrace$, is exact.
\end{cor}
\begin{proof}
There is an obvious map $\mathcal{K}^n(R)/p^r \to \mathcal{K}^n_{p^r}(R)$. it suffices to prove that it is surjective. If $\beta \in \K_{\textnormal{Mil}}^{n}(\bfk(R))$ represents an element $\beta \in \Kcal_{p^r}^{n}(R)$ then $\partial_v\beta = p^r\beta'$ for some $\beta' \in \K_{\textnormal{Mil}}^{n-1}(\bfk_v)$. Pick an element $\gamma \in (R^*)^{\otimes n-1}$ which restricts to $\beta'$ on the residue field; now the element $\beta - p^r\lbrace\pi \rbrace \cdot \beta'$ belongs to $\Kcal^n(R)$ and maps to $\beta \in \Kcal_{p^r}^{n}(R)$.

This shows moreover that the last map in the exact sequence is surjective. The only nontrivial thing left to show for the sequence to be exact is that the map $i_{r-1}$ is injective. If $i_{r-1}\alpha = 0$ then in $\K^{\bullet}_{\textnormal{Mil}}(\bfk(R))$ we must have $p^i \alpha' = p^r\beta$ for some lifting $\alpha'$ of $\alpha$ and some $\beta$.

Therefore, we must have $p^i(\alpha - p^{r-i}\beta)=0$. It follows then that $\iota_{r-1}$ is injective if and only if $\alpha - p^{r-i}\beta=0$ in this case, i.e. if there is no $p$-torsion in $\K^{\bullet}_{\textnormal{Mil}}(\bfk(R))$. This is true by \cite{IzhK}*{Theorem A}.
\end{proof}

What follows shows that the functor we defined is equal to ${\rm H}^n(D,\underline{\ZZ}/p^r(n))$ for all such $D$, and in particular for a DVR $(R,v)$.
\begin{prop}\label{prop:KMH}
Let $D$ be as above. There is a natural isomorphism 
\[
{\rm H}^n(D,\underline{\ZZ}/p^r(n)) \simeq \Kcal_{p^r}^n(D).
\]
\end{prop}
\begin{proof}
The case $r=1$ is proven by Geisser and Levine in \cite{GeisLev}*{Proposition 3.1}. 

In general, if the base field is perfect the proposition is immediate by \cite{GeisLev}*{Theorem 8.3}. To get around this hypothesis we apply Quillen's method: there exists a subfield $\bfk_0$ of $\bfk$ that is finitely generated over $\mathbb{F}_p$ and such that $R=R' \otimes_{\bfk_0} \bfk$. As all functors in the statement commute with direct limits we just have to prove the proposition for the fields $\bfk_0 \subseteq \bfk' \subseteq \bfk$ which are finitely generated over $\mathbb{F}_p$. For such a field we can always see $D_{\bfk'}$ as a partial localization around a subset of codimension one of a smooth variety over $\mathbb{F}_p$. Thus we have reduced to the case of $\bfk$ being perfect.

 \end{proof}

\subsection{The case $i=j+1=n+1$}
Now we move on to the groups ${\rm H}^{n+1}(R,\underline{\ZZ}/p^r(n))$. Note that we always have 
\[{\rm H}^{n+1}(R,\underline{\ZZ}/p^r(n)) \subseteq {\rm H}^{n+1}(\bfk(R),\underline{\ZZ}/p^r(n))_{\rm tm}\]
as $\bfk(R^{\rm sh}) \subset \bfk(R)^{\rm tm}$
where $R^{\rm sh}$ is the strict henselization of $R$. The following lemma and proposition deal with the case $r=1$. \Cref{prop:Tot sec 4} is stated without proof in \cite{Totp}*{Section 4}; Burt Totaro kindly explained a way to derive it to the authors.

\begin{lm}\label{lm:exact Omega}
Let $(R,v)$ be a DVR. Then the following sequence is exact in the small \'etale site of $R$:
\[
0 \longrightarrow {\rm H}^{n}(R,\underline{\ZZ}/p(n)) \longrightarrow \Omega^n \overset{1-\Phi}{\longrightarrow} \Omega^n/{\rm d}\Omega^{n-1} \longrightarrow 0 
\]
where $\Phi$ is the inverse Cartier operator defined by
\[
\Phi(a {\rm d}f_1 \wedge \ldots {\rm d}f_1) = (a^p f_1^{p-1} \ldots f_n^{p-1}) {\rm d}f_1 \wedge \ldots {\rm d}f_1\]
\end{lm}
\begin{proof}
This statement is a direct consequence of \cite{GeisLev}*{Proposition 3.1} and \cite{Mor}*{Corollary 4.1(ii) and 4.2(iii)}, but as much of the rest of the paper will require doing explicit computations we believe it is valuable to give a more basic proof here.

The first map is defined using the equivalence ${\rm H}^{n}(R,\underline{\ZZ}/p(n)) = \Omega^n_{{\rm log},R}$ in \cite{GeisLev}*{Proposition 3.1}, and its image is by definition the kernel of $(1-\Phi)$. It remains to show that the map $1-\Phi$ is \'etale locally surjective. Let $t$ be a uniformizer for $v$. Then $\Omega^n(R)$ is generated by two types of elements, namely $a {\rm d}b_1 \wedge \ldots \wedge {\rm d}b_n$ with $b_1, \ldots, b_n \in R^*$ and $a {\rm d}t \wedge {\rm d} b_2 \wedge \ldots \wedge {\rm d}b_n$ with $b_2, \ldots, b_n \in R^*$.

For elements of the first type write $a'= a b_1 \ldots b_{n}$, so that 
\[
a \,{\rm d}b_1 \wedge \ldots \wedge {\rm d}b_n = a' \frac{{\rm d}b_1}{b_1} \wedge \ldots \wedge \frac{{\rm d}b_n}{b_n}.
\]
We can take the Artin-Schreier extension $R\left[x\right]/(x^p-x-a')$, which is \'etale over $R$. Then it is immediate to check that
\[
(1 - \Phi)\left( x \frac{{\rm d}b_1}{b_1} \wedge \ldots \wedge \frac{{\rm d}b_n}{b_n}\right) = a' \frac{{\rm d}b_1}{b_1} \wedge \ldots \wedge \frac{{\rm d}b_n}{b_n}.
\]

Now consider the second type of elements. First assume that $v(a)=n$ is not congruent to $-1$ mod $p$, and write $a = t^n b$ with $b$ invertible in $R$. Then
\[
{\rm d}(t^{n+1}b \, {\rm d}t \wedge {\rm d} b_2 \wedge \ldots \wedge {\rm d}b_n) = (n+1) a \, {\rm d}t \wedge {\rm d} b_2 \wedge \ldots \wedge {\rm d}b_n + t^{n+1} \, {\rm d} b \wedge  {\rm d} b_2 \wedge \ldots \wedge {\rm d}b_n
\]

which shows that $a {\rm d}t \wedge {\rm d} b_2 \wedge \ldots \wedge {\rm d}b_n$ is equivalent to an element of the second type in $\Omega^n_R/{\rm d}\Omega^{n-1}_R$.

Now assume that $a = t^{mp-1}b $. Write $a' = t^{pm-p} b\, b_1 \ldots b_n$. It is immediate to check that after passing to $R\left[x\right]/(x^p-x-a')$ we have
\[
(1-\Phi)\left(x\, {\rm d}t \wedge \frac{{\rm d} b_2}{b_2} \wedge \ldots \wedge \frac{{\rm d}b_n}{b_n}\right) = a' t^{p-1} {\rm d}t \wedge \frac{{\rm d}b_1}{b_1} \wedge \ldots \wedge \frac{{\rm d}b_n}{b_n} = a\, {\rm d}t \wedge {\rm d} b_2 \wedge \ldots \wedge {\rm d}b_n.
\]

Finally, the surjectivity of $1-\Phi$ on DVRs easily imples the surjectivity on the whole small \'etale site by passing to stalks.
\end{proof}

\begin{prop}\label{prop:Tot sec 4}
The image of ${\rm H}^{n+1}(R,\underline{\ZZ}/p(n))$ in ${\rm H}(\bfk(R),\underline{\ZZ}/p(n))_{\rm tm}$ is equal to the kernel of the residue map $\partial_v$.
\end{prop}
\begin{proof}
As part of the proof of \cite{Totp}*{Theorem 4.3} Totaro shows that the kernel of $\partial_v$ is equal to the subgroup generated by the diffierentials $a\frac{{\rm d}b_1}{b_1}\wedge\cdots\frac{{\rm d}b_n}{b_n}$ where $a\in R$ and $b_1,\ldots,b_n \in R^{*}$. Following his notation, we will refer to this subgroup as ${\rm H}^{n+1}(\bfk(R))_{\rm nr}$.

We want to show that the image of ${\rm H}^{n+1}(R,\underline{\ZZ}/p(n))$ is exactly ${\rm H}^{n+1}(\bfk(R))_{\rm nr}$. 
For this, observe that the short exact sequence in \Cref{lm:exact Omega} induces a long exact sequence in cohomology, which by \cite{GeisLev}*{Thm 8.3} and, as usual, Quillen's method, gives us
\[ \Omega^n_R \xrightarrow{1-\Phi} \Omega^n_R/{\rm d}\Omega_R^{n-1} \rightarrow {\rm H}^1(R,\Omega^n_{{\rm log},R}) \rightarrow {\rm H}^1(R,\Omega^n_{R})=0, \]
where the last term vanishes because $\Omega^n_{R}$ is an $R$-module. This gives us a presentation of ${\rm H}^{n+1}(R,\underline{\ZZ}/p(n))$ as  a group of differentials which is compatible with the presentation of ${\rm H}^{n+1}(\bfk(R),\underline{\ZZ}/p(n))$.
By looking at the composition
\[ \Omega^n_R/d\Omega_R^{n-1} \longrightarrow {\rm H}^{n+1}(R,\underline{\ZZ}/p(n)) \longrightarrow {\rm H}^{n+1}(\bfk(R),\underline{\ZZ}/p(n))_{\rm tm} \]
we deduce that the image of ${\rm H}^{n+1}(R,\underline{\ZZ}/p(n))$ is generated, as above, by differentials of one of the two types
\[
a \, {\rm d} b_1 \wedge \ldots \wedge {\rm d} b_n \,\; \textnormal{or} \,\; a \, {\rm d} t \wedge {\rm d} b_2 \wedge \ldots \wedge {\rm d} b_n .
\]
Differentials of the first type are clearly unramified and surject to ${\rm H}^{n+1}(\bfk(R))_{\rm nr}$. Now observe that symbols of the second form can be rewritten as
\[
a\, {\rm d} t \wedge {\rm d} b_2 \wedge \ldots \wedge {\rm d} b_n  = a \, b_2 \ldots b_n \, t \frac{{\rm d} t}{t} \wedge \frac{{\rm d} b_2}{b_2} \wedge \ldots \wedge \frac{{\rm d} b_n}{b_n}
\]
whose ramification is by definition equal to 
 \[\overline{a \, b_2 \ldots b_n} \cdot \overline{t} \cdot \frac{{\rm d} \overline{b_2}}{\overline{b_2}} \wedge \ldots \wedge \frac{{\rm d} \overline{b_n}}{\overline{b_n}} = 0 \cdot \left( \frac{{\rm d} \overline{b_2}}{\overline{b_2}} \wedge \ldots \wedge \frac{{\rm d} \overline{b_n}}{\overline{b_n}}\right) =0.\]
\end{proof}

To generalize this result to $r>1$ we start by proving that the conclusion of \Cref{lm:truncation} works in this setting as well.

\begin{prop}
Let $(R,v)$ be a DVR. We have an exact sequence
\[
0 \to {\rm H}^{n+1}(R,\underline{\ZZ}/p^{r-1}(n)) \to {\rm H}^{n+1}(R,\underline{\ZZ}/p^{r}(n)) \to {\rm H}^{n+1}(R,\underline{\ZZ}/p(n)) \to 0
\]
which is compatible with the exact sequence on quotient fields
\[
0 \to {\rm H}^{n+1}_{p^{r-1}}(\bfk(R)) \to {\rm H}^{n+1}_{p^{r}}(\bfk(R)) \to {\rm H}^{n+1}_{p}(\bfk(R)) \to 0.
\]
\end{prop}
\begin{proof}
We begin by applying Quillen's method to reduce to $\bfk$ perfect (see the beginning of the proof of \Cref{prop:KMH}). Now, the exact sequence from \Cref{lm:exactK} implies that we have an exact sequence of the sheafifications of $\Kcal_{p^r}^{\bullet}$ on the site of $\Spec(R)$. By \cite{GeisLev}*{Thm 8.3} the sheafification of $\Kcal_{p^r}^{n}$ is equal to the sheaf $W_r\Omega^{n}_{R,{\rm log}}$, and the cohomology of the latter is equal to ${\rm H}^{i}(R,\underline{\ZZ}/p^r(n))$, which we know to be zero except when $i=n,n+1$. Therefore, we obtain the following long exact sequence in cohomology
\begin{align*}
0 \to {\rm H}^n(R,\underline{\ZZ}/p^{r-1}(n)) \to {\rm H}^n(R,\underline{\ZZ}/p^{r}(n)) \to {\rm H}^n(R,\underline{\ZZ}/p(n)) \to
{\rm H}^{n+1}(R,\underline{\ZZ}/p^{r-1}(n))\\ \to {\rm H}^{n+1}(R,\underline{\ZZ}/p^{r}(n)) \to {\rm H}^{n+1}(R,\underline{\ZZ}/p^{1}(n)) \to {\rm H}^{n+2}(R,\underline{\ZZ}/p^{r-1}(n)).
\end{align*}
By \Cref{eq:local ramification} we know that the last term is zero, as it injects into 
\[{\rm H}^{n+2}(\bfk(R),\underline{\ZZ}/p^{r-1}(n))=0.\] Moreover, we know by \Cref{lm:exactK} that the sequence splits into two five-terms exact sequences. So we have the exact sequence:
\[
0 \to {\rm H}^{n+1}(R,\underline{\ZZ}/p^{r-1}(n)) \to {\rm H}^{n+1}(R,\underline{\ZZ}/p^{r}(n)) \to {\rm H}^{n+1}(R,\underline{\ZZ}/p(n)) \to 0.
\]
Compatibility with the exact sequence on Izhboldin's presentation of the \'etale cohomology of $\bfk(R)$ is immediate from \cite{IzhK}*{Corollary 6.5}.
\end{proof}

To understand ramification in relation to the inclusion ${\rm H}^{n+1}(R,\underline{\ZZ}/p^r(n)) \to {\rm H}^{n+1}_{p^r}(\bfk(R))$ we want to describe its image in terms of symbols. We start from the exact sequence of \'{e}tale sheaves
\[
 0 \to \ZZ/p^r \to W_r(\Ga) \xrightarrow{\phi - 1} W_r(\Ga) \to 0
 \]
where $\phi \left[ a_1, \ldots, a_r\right]=\left[ a^p_1, \ldots, a^p_r\right]$. This induces a map 
\[
 W_r(R)/(\phi-1) \to {\rm H}^1(R,\ZZ/p^r).
\]
Combining this with the isomorphism $\Kcal_{p^r}^{\bullet}(R)={\rm H}^n(R,\underline{\ZZ}/p^r(n))$ and taking the cup product we get a map
\[
W_r(R) \otimes \Kcal_{p^r}^{n}(R) \to {\rm H}^{n+1}(R,\underline{\ZZ}/p^r(n))
\]
\begin{lm}\label{lm:surjH}
The map 
\[
W_r(R) \otimes \Kcal_{p^r}^{n}(R) \to {\rm H}^{n+1}(R,\underline{\ZZ}/p^r(n))
\]
is surjective for all $r, n$.
\end{lm}
\begin{proof}
The case $r=1$ is a consequence of \Cref{prop:Tot sec 4}. Now note that due to the equality $R=W_1(R) = W_r(R)/W_{r-1}(R)$ we have a natural exact sequence
\[
W_{r-1}(R)\otimes \Kcal_{p^{r-1}}(R) \to W_{r}(R)\otimes \Kcal_{p^{r}}(R) \to R\otimes \Kcal_{p}(R) \to 0
\]
which commutes with the morphisms above. Consider the commutative diagram
\[
\xymatrix@C=1em{
 0 \ar[r] & {\rm H}^{n+1}(R, \underline{\ZZ}/p^{r-1}(n))  \ar[r] & {\rm H}^{n+1}(R, \underline{\ZZ}/p^{r}(n))  \ar[r] & {\rm H}^{n+1}(R, \underline{\ZZ}/p(n))  \ar[r] & 0 \\
    & W_{r-1}(R)\otimes \Kcal_{p^{r-1}}(R) \ar[r] \ar[u] & W_{r}(R)\otimes \Kcal_{p^{r}}(R) \ar[r] \ar[u] & R\otimes \Kcal_{p}(R) \ar[r] \ar[u] & 0
   }
\]
We can assume by induction on $r$ that the left and right vertical maps are surjective. Then a standard diagram chase proves the surjectivity of the central vertical map.
\end{proof}

\begin{prop}\label{prop:Hunr}
Let $(R,v)$ be a DVR. We have an exact seqence
\[ 0 \to {\rm H}^{n+1}(R,\underline{\ZZ}/p^r(n)) \to {\rm H}^{n+1}_{p^r}(\bfk(R))_{\rm tm} \xrightarrow{\partial_v} {\rm H}^{n}_{p^r}(\bfk_v) \to 0  \]
\end{prop}
\begin{proof}
The statement is true for $r=1$ by \Cref{prop:Tot sec 4}. Now consider the commutative diagram 
\[
     \xymatrix@C=1em{
    0 \ar[r] & {\rm H}^n_{p^{r-1}}(\bfk_v) \ar[r] & {\rm H}^n_{p^{r}}(\bfk_v) \ar[r] & {\rm H}^n_{p^{r}}(\bfk_v) \ar[r] & 0 \\
   0 \ar[r] & {\rm H}^{n+1}_{p^{r-1}}(\bfk(R))_{\rm tm} \ar[r]^{\iota_{r-1}} \ar[u]^{\partial_v} & {\rm H}^{n+1}_{p^{r}}(\bfk(R))_{\rm tm} \ar[r]^{\pi_1} \ar[u]^{\partial_v}& {\rm H}^{n+1}_{p}(\bfk(R))_{\rm tm}  \ar[u]^{\partial_v}&  \\
      0 \ar[r] & {\rm H}^{n+1}(R, \underline{\ZZ}/p^{r-1}(n)) \ar[r] \ar[u] & {\rm H}^{n+1}(R, \underline{\ZZ}/p^{r}(n)) \ar[r] \ar[u] & {\rm H}^{n+1}(R, \underline{\ZZ}/p(n))\ar[u] \ar[r] & 0 \\
     }
 \]
Each row is exact, the lower vertical arrows are injective and the upper vertical arrows are surjective. We want to show that all columns are exact as well.

First, note that by \Cref{lm:surjH} each element of ${\rm H}^{n+1}_{p^r}(\bfk(R))$ coming from the group ${\rm H}^{n+1}(R, \underline{\ZZ}/p^{r}(n))$ is in the form $\left[ a_1, \ldots, a_r, b_1, \ldots, b_n \right \rbrace $ where $a_1, \ldots, a_r \in R$ and $\lbrace b_1, \ldots, b_n \rbrace \in \Kcal^n_{p^r}(R)$. These elements are unramified by construction, so all we need to prove is the exactness of 
\[
{\rm H}^{n+1}(R, \underline{\ZZ}/p^{r}(n)) \to {\rm Ker}(\partial_v) \to 0.
\]

By induction we can assume the left and right column are exact, and we have to show that the central one is. Pick $\gamma\in {\rm H}^{n+1}_{p^r}(\bfk(R))_{\rm tm}$ such that $\partial_v(\gamma)=0$. As $\partial_v(\pi_1(\gamma))=0$ there is $\gamma' \in {\rm H}^{n+1}(R,\underline{\ZZ}/p^r(n))$ such that $\pi_1(\gamma-\gamma')=0$. Then $\gamma-\gamma' = \iota_{r-1}(\beta)$ with $\partial_v(\beta)=0$, which implies that $\beta$ is in the image of ${\rm H}^{n+1}(R,\underline{\ZZ}/p^{r-1}(n))$, allowing us to conclude by the commutativity of the diagram.
\end{proof}

With this it is easy to check that the two notions of ramification coming from \Cref{eq:ramification} and Izhboldin's description are equivalent.

\begin{prop}
Let $(R,v)$ be a DVR with fraction field $F$ and let $y$ be the closed point of $\Spec(R)$. An element $\alpha \in {\rm H}^{n+1}_{p^r}(F)$ is unramified at $y$ in the sense of \Cref{lm:unraMerk} if and only if it is unramified at $v_y$ as above.
\end{prop}
\begin{proof}
Comparing \Cref{prop:DVR} and \Cref{eq:local ramification} we immediately see that both conditions amount to the element coming from the cohomology of $\Spec(R)$.
\end{proof}

\begin{cor}\label{cor:invariants are tame and unramified}
Let $(R,v)$ be a DVR with a map $\Spec(R) \to \Xcal$. Then given any invariant $\alpha \in \Inv(\Xcal, {\rm H}_{p^r})$ the element $\alpha(\bfk(R))$ is tamely ramified and unramified at $v$.

Given a scheme $X$ smooth over $\bfk$, the map $\alpha \mapsto \alpha(\bfk(X))$ is an isomorphism between the cohomological invariants of $X$ and the subgroup of ${\rm H}^{\bullet}_{p^r}(\bfk(X))_{\rm tm}$ of elements unramified at all $x \in X^{(1)}$.
\end{cor}

\section{Some computations of cohomological invariants}\label{sec:some comp}
In this section we come back to cohomological invariants, and present some explicit computations that will be useful for our study of the invariants of $\Mcal_{1,1}$.

To start, we recall Izhboldin's description of ${\rm H}^{n+1}_{p^r}(\bfk(t))$, which we state here using Totaro's notation:
\begin{thm}[\cite{Izh}*{Theorem 4.5, Theorem 6.10}]\label{thm:tame over P1}
Let $S$ be the set of closed points of $\PP^1_\bfk$. Write $v_y$ for the valuation corresponding to the point $y$, and $\bfk(t)_{v_y}$ for the completion of $\bfk(t)$ at $v_y$. We have an exact sequence:
\[
{\rm H}^{n+1}_{p^r}(\bfk(t)) \rightarrow \bigoplus_{y\in S}  \widetilde{{\rm H}}^{\bullet}_{p^{r}}(\bfk(t)_{v_y}) \to 0.
\]

The kernel of this homomorphism, which we call ${\rm H}^{n+1}_{p^r}(\bfk(t))_{{\rm tm}/\PP^1}$, fits in the following exact sequence
\[
0 \to {\rm H}^{n+1}_{p^r}(\bfk)  \to {\rm H}^{n+1}_{p^r}(\bfk(t))_{{\rm tm}/\PP^1} \to \bigoplus_{S \setminus \infty} {\rm H}^{n+1}_{p^r}(\bfk(y)) \to 0 
\]
\end{thm}

\begin{rmk}
Note that by \cite{Totp}*{Theorem 4.3} given a discrete valuation field $(F,v)$ we have $\widetilde{{\rm H}}^{\bullet}_{p}(F)=\widetilde{{\rm H}}^{\bullet}_{p}(F_v)$, where $F_v$ is the completion of $F$ at $v$, simplifying the formula.
\end{rmk}

As an immediate corollary we get the cohomological invariants of $\bA^1$ and $\bA^1\smallsetminus \lbrace 0 \rbrace$

\begin{cor}\label{cor:inv A1}
We have 
\begin{gather*}
0 \to {\rm H}^{\bullet}_{p^r}(\bfk) \to \Inv(\bA^1,{\rm H}_{p^r}) \to \widetilde{{\rm H}}^{\bullet}_{p^r}(\bfk((1/t))) \to 0 \\[2ex]
0 \to {\rm H}^{\bullet}_{p^r}(\bfk)\oplus {\rm H}^{\bullet-1}_{p^r}(\bfk) \to \Inv(\bA^1\smallsetminus\{0\},{\rm H}_{p^r}) \to \widetilde{{\rm H}}^{\bullet}_{p^r}((\bfk(t)))\oplus \widetilde{{\rm H}}^{\bullet}_{p^r}(\bfk((1/t))) \to 0
\end{gather*}

Here the left hand map sends a pair $[a,b_1,\ldots,b_n\},[a',b'_2,\ldots,b'_n\}$ of degree $n+1$ to the element $[a,b_1,\ldots,b_n\}+[a',t,b'_2\ldots,b'_n\}$, and its image is contained in ${\rm H}^{\bullet}_{p^r}(k(t))_{{\rm tm}/\PP^1}$. The component on the right comes from the wild part of ${\rm H}^{\bullet}_{p^r}(k(t))$.
\end{cor}

The cohomological invariants of ${\Brm}\ZZ/p$ with coefficients in ${\rm H}_p$ were computed by Totaro \cite{Totp}*{Proposition 8.1}. Here we give a different proof which we think has independent interest, and then we'll extend the description to coefficients in ${\rm H}_{p^r}$.

\begin{thm}[(Totaro)]\label{ZpTot}
We have
\[
\Inv({\Brm}\ZZ/p,{\rm H}_{p}) = {\rm H}^{\bullet}_{p}(\bfk) \oplus \K^{\bullet-1}_{p}(\bfk)\cdot\alpha
\]
where $\alpha \in {\rm H}^1({\Brm}\ZZ/p,\ZZ/p)$ is the identity element, i.e. given a map $\Spec(F) \to {\Brm}\ZZ/p$ induced by a torsor $\pi:E\to \Spec(F)$, we have $\alpha(F)= \pi \in {H}^1(F,\ZZ/p)$
\end{thm}
\begin{proof}

Consider the map $\Ga \to \Ga$ given by $t \mapsto t^p - t$. It induces an additive action of $\Ga$ on $\bA^1$ which is transitive and has stabilizers equal to $\ZZ/p$, so $\left[ \bA^1/\Ga\right] = {\Brm}\ZZ/p$. We have 
\[\bA^1 \times_{{\Brm}\ZZ/p} \bA^1 = \bA^1 \times  \Ga = \Spec(\bfk\left[t,s\right])\]
where the two projections are the first projection ${\rm pr}_1$ and multiplication map ${\rm m}$.

The map $\bA^1 \to {\Brm}\ZZ/p$ is smooth-Nisnevich, so we can conclude that the cohomological invariants of ${\Brm}\ZZ/p$ are the invariants of $\bA^1$ such that ${\rm pr}_1^*\alpha = {\rm m}^*\alpha$. 

We claim that the subgroup of elements that descend to invariants of ${\Brm}\ZZ/p$ is generated by elements coming from the base field and elements in the form $t\frac{{\rm d}b_1}{b_1}\wedge\ldots\frac{{\rm d}b_n}{b_n}$. The elements coming from the base field form a direct summand in $\Inv(\bA^1,{\rm H}_p)$, and they trivially glue, so let $\alpha$ be an element of ${\rm Inv}^n(\bA^1,{\rm H}_p)$ that restricts to zero at $t=0$. These elements form a subgroup isomorphic to ${\rm Inv}^n(\bA^1,{\rm H}_p)/{\rm H}^{n+1}_p(\bfk)$.

By \Cref{cor:inv A1} and \Cref{lm:simple form} we can write such an element as

\begin{equation}\label{eq:simple form alpha}
\alpha = \sum_i t^i\varphi_i + t^i\frac{{\rm d}t}{t}\wedge\varphi'_i, \quad \varphi_i \in \Omega^n_k,\varphi'_i\in\Omega^{n-1}_\bfk.
\end{equation}
where $\varphi'_i=0$ if $p\nmid i$.
We have
\begin{align*}
    {\rm pr}_1^*(t^i\varphi_i)=t^i\varphi_i&, \quad {\rm pr}_1^*(t^i\frac{{\rm d}t}{t}\wedge \varphi'_i)=t^i\frac{{\rm d}t}{t}\wedge \varphi'_i,\\
    {\rm m}^*(t^i\varphi_i)=(t+s^p-s)^i\varphi_i&, \quad {\rm m}^*(t^i\frac{{\rm d}t}{t}\wedge \varphi'_i)=(t+s^p-s)^i\frac{{\rm d}(t+s^p-s)}{t+s^p-s}\wedge \varphi'_i.
\end{align*}
We first show that for an element to descend we must have $i=1$.  Observe that if $\alpha$ glues then the pullback of ${\rm pr}_1^*\alpha-{\rm m}^*\alpha$ to the copy of $\Ga$ given by $t=0$ must be zero, that is

\[
\alpha' = \sum_i (s^p-s)^i (\varphi_i) - s(s^p-s)^{i-1} \frac{{\rm d}s}{s}\wedge\varphi'_i = 0
\]

Pick the highest $i$ such that $\alpha\in {\rm U}_i/{\rm U}_{i-1}$, which we can assume to be the highest $i$ appearing in equation \eqref{eq:simple form alpha}. 
If $i$ is divisible by $p$ we can assume that at least one among $\varphi_i$ and $\varphi'_i$ is not closed, otherwise we could apply \Cref{lm:simple form} to replace the term $t^{i}(\varphi_i+ \frac{{\rm d}t}{t}\wedge \varphi'_i)$ with something of the form $t^{i/p}\varphi''_i$, and this would imply that $\alpha\in {\rm U}_{i-1}$. 

If $\varphi_i$ is not closed, then the element $s^{pi}\varphi_i$ is the highest degree term in the sum and by \Cref{thm:wild pres} this implies that the equivalence class of $\alpha'$ in ${\rm U}_{pi}/{\rm U}_{pi-1}$ is non-zero.

If $\varphi_i$ is closed, we can assume that it is actually zero: indeed, as before, we can apply \Cref{lm:simple form} to rewrite $\alpha$ in such a way that the highest power of $t$ appears in the term $t^i\frac{{\rm d}t}{t}\wedge\varphi'_i$ with $\varphi'_i$ not closed. Therefore the element $\alpha'$ is equivalent to $-s^{p(i-1)+1} \frac{{\rm d}s}{s}\wedge\varphi'_i$ in ${\rm U}_{i}/{\rm U}_{i-1}$. Observe that in ${\rm H}^{n+1}_p(k(s))$ we have
\begin{align*}
    0 = {\rm d}(s^{p(i-1)+1}\varphi'_i) = s^{p(i-1)+1}{\rm d}(\varphi'_i) + (p(i-1)+1)s^{p(i-1)}ds\wedge\varphi'_i
\end{align*} 
which immediately implies
\[ -s^{p(i-1)+1} \frac{{\rm d}s}{s}\wedge\varphi'_i =  (p(i-1)+1)^{-1} s^{p(i-1)+1}{\rm d}(\varphi'_i) \]
and the term on the right is non-zero in ${\rm U}_{p(i-1)+1}/{\rm U}_{p(i-1)}$: this follows from \Cref{thm:wild pres}, as ${\rm d}\varphi'_i$ is by hypothesis non-zero in $\Omega_\bfk^{n}$, and $p(i-1)+1$ is not divisible by $p$.

If $i>1$ is not divisible by $p$ (hence $\varphi'_i=0$) we consider two cases: if $\varphi_i$ is not closed we can conclude exactly as above. If $\varphi_i$ is closed then the term $s^{pi}(\varphi_i)$ is equivalent to some element of degree $i$, and the next highest degree term is $-{i}s^{(i-1)p+1}(\varphi_i)$, which is nonzero. As $i$ is not divisible by $p$ we have that $-{i}s^{(i-1)p+1}(\varphi_i)$ is nonzero in ${\rm U}_{(i-1)p+1}/{\rm U}_{(i-1)p}$.

We are left with the case $i=1$. In this case ${\rm pr}_1^*\alpha-{\rm m}^*\alpha = (s^p-s)(\varphi) $. If $\varphi$ is not closed then clearly this element cannot be zero. Now assume $\varphi$ is closed: by Cartier's theorem we can write $\varphi = (\sum_j a_j^p(\psi_j)) + \varphi' $ where the forms $\psi_j$ are logarithmic and $\varphi'$ is exact. Then $(s^p-s)(\varphi)=s(\sum_j (a_j-a_j^p)\psi_j) + \varphi')$. If the element $(\sum_j (a_j-a_j^p)\psi_j)+ \varphi'$ is equal to zero then in particular it is equal to zero in $\Omega^{n}_\bfk/{\rm d}\Omega^{n-1}_\bfk$. But this implies that $\sum_j (a_j-a_j^p)\psi_j = \mathcal{P}(\sum_j -a_j\psi_j)=0$, which implies $\sum_j -a_j\psi_j$ is logarithmic, and in particular $\sum_j (a_j-a_j^p)\psi_j = \mathcal{P}(\sum_j -a_j\psi_j)=0$ in $\Omega^n_\bfk$ as well. Then $\varphi' = 0$ necessarily.

Finally, note that if $\mathcal{P}(\sum_j a_j\psi_j)$ is zero than the same must be true for $\sum_j a_j^p\psi_j$: let $\mathcal{F}$ be the additive morphism defined $a\varphi \mapsto a^p\varphi$. Then $\mathcal{P}$ and $\mathcal{F}$ commute, so 
\[\mathcal{P}(\sum_j a_j^p\psi_j)=\mathcal{P}(\mathcal{F}(\sum_j a_j\psi_j))=\mathcal{F}(\mathcal{P}(\sum_j a_j\psi_j))=0
\]
showing that if $t(\varphi)$ glues we must have $\varphi \in \Omega_{\bfk, {\rm log}}^n=\K^n_p(\bfk)$, as claimed.

This shows that 
\[\Inv({\Brm}\ZZ/p,{\rm H}_p)\simeq {\rm H}^{\bullet}_{p}(\bfk) \oplus \K^{\bullet-1}_{p}(\bfk)\cdot t
\]
and we only have to identify the element $t$ with the pullback of $\alpha \in {\rm H}^1({\Brm}\ZZ/p, \ZZ/p)$. To do this, consider the $\ZZ/p$-torsor $E \to \bA^1$ induced by taking a root of $x^p-x-t$. It's immediate to check that the induced map $\bA^1 \to {\Brm}\ZZ/p$ is the quotient we considered above. Now, the pullback of $\alpha$ to the base of a torsor $E \to \Spec(R)$ defined by adding a root of $x^p-x-r$ is exactly the class $\left[ r \right] \in {\rm H}^1(\Spec(R),\ZZ/p)$, allowing us to conclude.
\end{proof}

\begin{cor}
We have
\[ 
\Inv({\Brm}\ZZ/p, {\rm H}_{p^r})= {\rm H}^{\bullet}_{p^r}(k) \oplus \alpha \cdot \K^{\bullet-1}_{p}(\bfk)
\]
where we identify $\K^{\bullet}_{p} = \K^{\bullet}_{p^r}(k)/p\K^{\bullet}_{p^r}(k)$.
\end{cor}
\begin{proof}
Write $E_{\rm triv}$ for the trivial torsor $\ZZ/p \times \Spec(\bfk) \to \Spec(\bfk)$. An invariant $\alpha \in \Inv({\Brm}\ZZ/p, {\rm H}_{p^r})$ is \emph{normalized} if $\alpha(E_{\rm triv})=0$; given any invariant $\alpha$, the invariant $\alpha - \alpha(E_{\rm triv})$ is normalized, where the second element is seen as a constant invariant, so the group of normalized invariants of $\Brm \ZZ/p$ is equal to $\Inv({\Brm}\ZZ/p, {\rm H}_{p^r})/{\rm H}^{\bullet}_{p^r}(\bfk)$. 

The claim is equivalent to showing that all normalized invariants of ${\Brm}\ZZ/p$ with mod $p^r$ coefficients come from normalized invariants with mod $p$ coefficients. The map $\iota_1:{\rm H}_{p} \to {\rm H}_{p^r}$ induces by composition a map
\[\Inv({\Brm}\ZZ/p, {\rm H}^{\bullet}_{p})/{\rm H}^{\bullet}_{p}(\bfk)\to \Inv({\Brm}\ZZ/p, {\rm H}^{\bullet}_{p^r})/{\rm H}^{\bullet}_{p^r}(\bfk).\] 

Now, observe that if $\alpha$ is a normalized cohomological invariant of ${\Brm}\ZZ/p$ with any coefficients then for any $\ZZ/p$-torsor $E \to \Spec(F)$ the pullback of $\alpha$ to $E$ must be trivial, as $E\times_K E \xrightarrow{\pi} E$ is a trivial torsor, and thus by a standard transfer argument we must have $0=\pi_* \pi^* \alpha(F) = p\alpha(F)$, i.e. $\alpha$ is of $p$-torsion. By \Cref{prop:ptors} we know that the $p$-torsion of ${\rm H}^{\bullet}_{p^r}(F)$ is exactly ${\rm H}^{\bullet}_{p}(F)$: using this identification we get a map 
\[\Inv({\Brm}\ZZ/p, {\rm H}^{\bullet}_{p^r})/{\rm H}^{\bullet}_{p^r}(\bfk)\to \Inv({\Brm}\ZZ/p, {\rm H}^{\bullet}_{p})/{\rm H}^{\bullet}_{p}(\bfk).\]
It's immediate that the two maps are inverse to each other, proving our claim.
\end{proof}

Now we consider the relative case, following the same idea as \cite{Totp}*{Proposition 8.3}.

\begin{cor}\label{cor:trivial action}
Let $X$ be a smooth scheme. Then
\[
\Inv(X \times {\Brm}\ZZ/p,{\rm H}_{p^r}) = \Inv(X,{\rm H}_{p^r}) \oplus \alpha \cdot \Inv(X,\K_{p})
\]
\end{cor}
\begin{proof}
Let $\beta$ be a cohomological invariant of $\left[ X/(\ZZ/p)\right]$. Given a point $x:\Spec(F) \to X$, restricting $\beta$ to the fiber of $x$ we get an invariant of ${\Brm}\ZZ/p\times \Spec(F)$. In particular we get unique elements $a_\beta(x) \in {\rm H}^{\bullet}_{p^r}(F), b_\beta(x) \in \K^{\bullet}_p(F)$. We claim that the associations $x \mapsto a_\beta(x), \, x \mapsto b_\beta(x)$ are cohomological invariants of $X$ and that $\beta = a_\beta + b_\beta \cdot \alpha$.

First, it is easy to see that $a_\beta$ is a cohomological invariant of $X$, as it is equal to the pullback of $\beta$ through the map $X \to X \times \Brm \ZZ/p\ZZ$ induced by $\Spec(\bfk) \to \Brm \ZZ/p\ZZ$. Now consider $b_\beta$; we want to show that it is functorial and that it satisfies the continuity condition. Functoriality is an immediate consequence of uniqueness. 

Given an Henselian DVR $(R, v)$ with a map $\Spec(R) \to X$ we get a map $\Spec(R) \times {\Brm}\ZZ/p \to X \times {\Brm}\ZZ/p$. Taking the usual cover $\bA^1 \to {\Brm}\ZZ/p$ we get $\Spec(R) \times \bA^1 = \Spec(R\left[ t\right]) \to X \times {\Brm}\ZZ/p$, and the pullback of $\beta$ to the generic point is $a_\beta(\bfk(R)) + b_\beta(\bfk(R))\cdot \lbrace t\rbrace$. On the other hand, we have $\beta(\bfk_v(t)) = a_\beta(\bfk_v) + b_\beta(\bfk_v)\cdot \lbrace t\rbrace$, and 
\[
j(a_\beta(\bfk_v) + b_\beta(\bfk_v)\cdot \lbrace t\rbrace) = a_\beta(\bfk(R)) + b_\beta(\bfk(R))\cdot \lbrace t \rbrace.
\]
Now, as $j$ is a composition of cohomology pullbacks and their inverses it respects both the group structure and cup product, which implies that 
\[
j(a_\beta(\bfk_v) + b_\beta(\bfk_v)\cdot \lbrace t\rbrace)=j(a_\beta(\bfk)) + j(b_\beta(\bfk_v))\cdot t
\]
which in turn gives us $j(b_\beta(\bfk_v))=b_\beta(\bfk(R))$, showing that $b$ is a cohomological invariant of $X$ as well.

Finally, consider the inclusion of the subgroup 
\[\Inv(X,{\rm H}_{p^r})\oplus \Inv(X, \K_{p}) \cdot \alpha \subseteq \Inv(\left[ X/(\ZZ/p)\right],{\rm H}_{p^r}).\] 
Given an invariant $\beta$ the element $\beta - a_\beta - b_\beta\cdot \alpha$ is by construction zero at all points, concluding our proof.
\end{proof}
As one would expect, the invariants of $X\times{\Brm}\ZZ/q$ with coefficients in ${\rm H}_{p^r}^{\bullet}$, for $q$ a prime different from $p$, are only those coming from $X$.

\begin{lm}\label{lm:invariants Zq}
Let $X$ be a smooth scheme over a field $\bfk$ of characteristic $p$ and let $q\neq p$ be a prime integer. Then the pullback along the projection $X\times \Brm\ZZ/q\to X$ induces an isomorphism
\[ \Inv(X,{\rm H}_{p^r})  \simeq \Inv(X\times\Brm\ZZ/q,{\rm H}_{p^r}).\]
\end{lm}
\begin{proof}
We can regard $X\times\Brm\ZZ/q$ as the quotient stack $[X/(\ZZ/q)]$, where $\ZZ/q$ acts trivially on $X$, so that we have a quotient map $\pi:X\to X\times\Brm\ZZ/q$. Observe that the composition $X\to X\times\Brm\ZZ/q \overset{\pr_1}{\to} X$ is equal to the identity. This implies that $\pi^*$ is a section of $\pr_1^*:\Inv(X,{\rm H}_{p^r})\to\Inv(X\times\Brm\ZZ/q,{\rm H}_{p^r})$, hence it induces a decomposition
\[ \Inv(X\times\Brm\ZZ/q,{\rm H}_{p^r}) \simeq \Inv(X,{\rm H}_{p^r}) \oplus \ker(\pi^*). \]
For every field $F\supset \bfk$ and for every morphism $f:\Spec(F)\to X\times\Brm\ZZ/q$ we can form the diagram
\[
\begin{tikzcd}
E \ar[r, "f'"] \ar[d, "\pi'"] & X \ar[d, "\pi"] \\
\Spec(F) \ar[r, "f"] \ar[rd, "g"] & X\times\Brm\ZZ/q \ar[d] \\
& X.
\end{tikzcd}
\]
In particular we have that $E\to\Spec(F)$ is a $\ZZ/q$-torsor.
Let $\beta=\pr_1^*\beta_1+\beta_0$ be an invariant of $X\times\Brm\ZZ/q$, with $\pi^*\beta_0=0$. We claim that $f^*\beta = g^*\beta_1$; this would imply that the pullback morphism $\pr_1^*$ induces an isomorphism of cohomological invariants.
We have
\begin{align*}
    q f^*\beta &= \pi'_*{\pi'}^* f^*\beta = \pi'_*{f'}^*\pi^*(\pr_1^*\beta_1 + \beta_0) \\
    &=\pi'_*{f'}^*(\beta_1) = \pi'_*{\pi'}^*g^*\beta_1 = q g^*\beta_1.
\end{align*}
As $q$ is invertible in ${\rm H}_{p^r}^{\bullet}(F)$, we get the claimed equality.
\end{proof}
\begin{lm}\label{lm:invariants X x BG}
Let $X$ be a smooth scheme over a field $\bfk$ of characteristic $p>0$ and let $G$ be a smooth affine group over $\bfk$. Then $\Inv(X\times\Brm G,\K_{p^r})\simeq \Inv(X,\K_{p^r})$.
\end{lm}
\begin{proof}
First, observe that $X\times\Brm G\simeq [X/G]$, where the groups action trivial. Let $\pi:X\to X\times\Brm G$ be the quotient map: the same argument used in the proof of \Cref{lm:invariants Zq} gives us a decomposition
\[ \Inv(X\times\Brm G,\K_{p^r}) \simeq \Inv(X,\K_{p^r}) \oplus \ker(\pi^*). \]
Let $\alpha$ be an invariant in $\ker(\pi^*)$. For every field $F$ with a map $f:\Spec(F)\to X\times\Brm G$, we claim that $f^*\alpha=0$ in $\K^{\bullet}_{p^r}(F)$. If so, this implies that $\ker(\pi^*)=0$ and we are done.

To prove the claim, observe that if the induced $G$-torsor $X_F\to F$ is trivial, then we have a factorization
\[
\begin{tikzcd}
 & X \ar[d, "\pi"] \\
 \Spec(F) \ar[ur, "g"] \ar[r, "f"] & X\times\Brm G,
\end{tikzcd}
\]
hence $f^*\alpha=g^*\pi^*\alpha=0$. If $X_F\to \Spec(F)$ is not trivial, let $F^s$ be the separable closure of $F$ and denote $h:\Spec(F^s)\to\Spec(F)$ the induced morphism. Then $X_{F^s}\to\Spec(F^s)$ is a trivial $G$-torsor because $G$ is smooth and affine, thus $h^*f^*\alpha=0$. To conclude, observe that $h^*$ is injective because it coincides with the injective homomorphism $W_r\Omega^{\bullet}_{F,{\rm log}}\to W_r\Omega^{\bullet}_{F^s,{\rm log}}$.
\end{proof}

\section{Cohomological invariants of $\mathcal{M}_{1,1}$ in positive characteristic}\label{sec:inv M11}
Let $\mathcal{M}_{1,1}$ be the stack of elliptic curves over a base field $\bfk$ of characteristic $p>0$. In this section we compute the cohomological invariants of $\mathcal{M}_{1,1}$ with coefficients in ${\rm H}_{p^r}$ and $\K_{p^r}$. Our main results for coefficients in ${\rm H}_{p^r}$ can be summarized as follows.
\begin{thm}\label{thm:inv M11 H}
\[ \Inv(\mathcal{M}_{1,1},{\rm H}_{p^r})\simeq 
\begin{cases}
\Inv(\bA^1,{\rm H}_{2^r})\oplus {\rm J}_{2^r}^{\bullet-1}(\bfk) & \textnormal{ if }p=2, \\
\Inv(\bA^1,{\rm H}_{3^r})\oplus {\rm H}^{\bullet-1}_{3}(\bfk)\cdot\{\Delta\} & \textnormal{ if }p=3, \\
\Inv(\bA^1,{\rm H}_{p^r}) & \textnormal{ if }p>3.
\end{cases}
\]
\end{thm}
The group ${\rm J}_{2^r}^{\bullet}(\bfk)$ is defined in \Cref{rmk:J} (also see the paragraph above the remark), and $\{\Delta\}$ is the invariant that sends an elliptic curve $C\to\Spec(F)$ to its discriminant, regarded as an element in $\K^1_{p^r}(F)$ (for those $p$ for which this quantity is well defined).

The computation is divided in three cases, namely the case $p=2$ is discussed in \ref{sec:char 2}, the case $p=3$ in \ref{sec:char 3} and finally the case $p>3$ in \ref{sec:char p}. The case $p=2$ is the most complicated to deal with and the most interesting. 

For what concerns invariants with coefficients in $\K_{p^r}$, we prove the following.
\begin{thm}\label{thm:inv M11 K}
\[
\Inv(\Mcal_{1,1},\K_{p^r})\simeq
\begin{cases}
\K^{\bullet}_{2}(\bfk)\oplus \K^{\bullet-1}_2\cdot\{\Delta\} & \textnormal{if }p=2,r=1, \\
\K^{\bullet}_{2^r}(\bfk)\oplus\K^{\bullet-1}_4\cdot\{\Delta\}&\textnormal{if }p=2,r>1,\\
\K^{\bullet}_{3^r}(\bfk)\oplus\K^{\bullet-1}_3\cdot\{\Delta\}&\textnormal{if }p=3, \\
\K^{\bullet}_{p^r}(\bfk) &\textnormal{if }p>3.
\end{cases}
\]
\end{thm}
This computation is contained in \ref{sec:inv K} and is much easier than the previous ones.
\subsection{Setup}\label{sub:setup}
Here we recall some basic facts about $\mathcal{M}_{1,1}$.
Following \cites{FulOl, Sil}, we adopt the following notation: given variables $a_1,a_2,a_3,a_4,a_6$ we set
\[b_2=a_1^2+4a_4,\quad b_4=2a_4+a_1a_3,\quad b_6=a_3^2+4a_6 \]
and
\begin{align*}
    b_8&=a_1^2a_6+4a_2a_6-a_1a_3a_4+a_2a_3^2-a_4^2,\\
    c_4&=b_2^2-24b_4,\\
    c_6&=-b_2^3-36b_2b_4-216b_6,\\
    \Delta&=-b_2^2 b_8-8b_4^3-27b_6^2+9b_2b_4b_6,\\
    j&=c_4^3/\Delta.
\end{align*}
We also have $4b_8=b_2b_6-b_4^2$ and $1728\Delta = c_4^3-c_6^2$.

We denote $U\overset{\textnormal{def}}{=}\Spec(\ZZ[a_1,a_2,a_3,a_4,a_6,\Delta^{-1}])$ and we define the group scheme $G$ over $\Spec(\ZZ)$ as the scheme $\Spec(\ZZ[u^{\pm 1},r,s,t])$ together with the group structure
\[ (u,r,s,t)\cdot (u',r',s',t')\overset{\textnormal{def}}{=}(uu',u^2r'+r,us'+s,u^3t'+u^2r's+t).  \]
There is an action of $G$ on $U$ defined via the following coaction:
\begin{align*}
    a_1&\longmapsto u^{-1}(a_1+2s), \\
    a_2&\longmapsto u^{-2}(a_2-sa_1+3r-s^2),\\
    a_3&\longmapsto u^{-3}(a_3+ra_1+2t),\\
    a_4&\longmapsto u^{-4}(a_4-sa_3+2ra_2-(t+rs)a_1+3r^2-2st),\\
    a_6&\longmapsto u^{-6}(a_6+ra_4+r^2a_2+r^3-ta_3-t^2-rta_1)
\end{align*}
The quotient stack $[U/G]$ is isomorphic to $\mathcal{M}_{1,1,\ZZ}$, the stack of elliptic curves over $\Spec(\ZZ)$.

Let $j:\mathcal{M}_{1,1,\ZZ}\to M_{1,1,\ZZ}$ be the morphism to the coarse moduli space. The scheme $M_{1,1,\ZZ}$ is isomorphic to $\bA^1_{\ZZ}$, and the induced map $U\to \bA^1_{\ZZ}$ corresponds to the element $j$ defined above.

\subsection{Characteristic 2}\label{sec:char 2}
In this subsection, we work over a ground field $\bfk$ of characteristic 2.
Our goal is to compute the group $\Inv(\mathcal{M}_{1,1,\bfk},{\rm H}_{2^r})$.
In what follows, as we only deal with stacks and schemes over $\bfk$, we use the simpler notation $\mathcal{M}_{1,1}\overset{\textnormal{def}}{=}\mathcal{M}_{1,1,\bfk}$.

The stack ${\mathcal{M}}_{1,1}\smallsetminus j^{-1}(0)$ is isomorphic to $(\bA^1\smallsetminus\{0\})\times \Brm\ZZ/2$: one way to see this is to observe that $\mathcal{M}_{1,1}\smallsetminus j^{-1}(0) \to \bA^1\smallsetminus\{0\}$ is a gerbe banded by $\ZZ/2$ and that it has a section, as in \cite{Shi}*{Lemma 3.2}. 

Another one consists in showing that $\mathcal{M}_{1,1}\smallsetminus j^{-1}(0)\simeq [\Spec(\bfk[a'_2,a'_6,{a'}_6^{-1}])/\Ga]$, where the additive group acts by $(a'_2,a'_6)\mapsto (a'_2+s(s+1),a'_6)$. This follows from \cite{Sil}*{Appendix A, Proposition 1.1} or, more explicitly, by considering the map
\[\bfk[a'_2,{a'}_6^{\pm 1}]\longrightarrow \bfk[a_1^{\pm 1},a_2,a_4,a_6,\Delta^{-1}]\]
given by
\begin{align*}
    a_2'&\longmapsto \frac{a_1a_2+a_3}{a_1^3},\\
    a_6'&\longmapsto \frac{a_1^4 a_2 a_3^2 + a_1^3 a_3^3 + a_3^4 + a_1^5 a_3 a_4 + a_1^4 a_4^2 + a_1^6 a_6}{a_1^{12}}=\frac{\Delta}{a_1^{12}}=j^{-1}.
\end{align*}
Then it is immediate to check that $[\Spec(\bfk[a'_2,{a'}_6^{\pm 1}])/\Ga]\simeq \Spec(\bfk[{a'}_6^{\pm 1}])\times{\Brm}\ZZ/2$.

\begin{lm}\label{lm:inv open part}
We have
\[ \Inv(\mathcal{M}_{1,1}\smallsetminus j^{-1}(0),{\rm H}_{2^r})\simeq \Inv(\bA^1\smallsetminus\{0\},{\rm H}_{2^r})\oplus \Inv(\bA^1\smallsetminus\{0\},\K_2)\cdot [\alpha],   \]
where $[\alpha]$ is the cohomological invariant pulled back from $\Brm\ZZ/2$.
\end{lm}
\begin{proof}
We apply \Cref{cor:trivial action} to $\mathcal{M}_{1,1}\smallsetminus j^{-1}(0)\simeq(\bA^1\smallsetminus \{0\})\times\Brm\ZZ/2$.
\end{proof}

\begin{rmk}\label{rmk:alpha}
More concretely, using the presentation 
\[\mathcal{M}_{1,1}\smallsetminus j^{-1}(0)\simeq [\Spec(\bfk[a'_2,{a'}_6^{\pm 1}])/\Ga],\]
we can look at the pullback of $[\alpha]$ to $\Spec(\bfk[a'_2,{a'}_6^{\pm 1}])$. Given a map $\Spec(F)\to \Spec(\bfk[a'_2,{a'}_6^{\pm 1}])$, there is an induced map $\varphi:\bfk[a'_2,{a'}_6^{\pm 1}]\to F$. Then the pullback of $[\alpha]$  coincides with the invariant
\[ (\Spec(F)\to \Spec(\bfk[a'_2,{a'}_6^{\pm 1}])) \longmapsto [(0,\ldots,0,\varphi(a'_2))] \in \rm H^1_{2^r}(F).   \]
Observe that this cohomological invariant is $\Ga$-invariant, so it descends to $\mathcal{M}_{1,1}\smallsetminus j^{-1}(0)$, as expected.

Moreover we deduce that the pullback of $[\alpha]$ to ${\rm H}^1_{p^r}(\bfk(a_1,\ldots,a_6))$ is equal to $[(0,\ldots,0,\phi(\frac{a_1a_2+a_3}{a_1^3})]$.
\end{rmk}

Knowing the cohomological invariants of $\mathcal{M}_{1,1}\smallsetminus\lbrace j=0 \rbrace$, the next step is to understand which ones extend to cohomological invariants of $\mathcal{M}_{1,1}$. For this, we first investigate what are the invariants of $\mathcal{M}_{1,1}\smallsetminus\lbrace j=0 \rbrace$ that are tamely ramified over $\mathcal{M}_{1,1}$.

Recall that $\mathcal{M}_{1,1}\simeq [U/G]$, where $U\overset{\textnormal{def}}{=}\Spec(\ZZ[a_1,a_2,a_3,a_4,a_6,\Delta^{-1}])$ and $G$ is a special group. In particular, the morphism $\pi:U\to\mathcal{M}_{1,1}$ is smooth-Nisnevich, hence we can check if an element is wildly ramified, tamely ramified or unramified by looking at its pullback to $U$. In this presentation, the substack $\mathcal{M}_{1,1}\smallsetminus\lbrace j=0 \rbrace$ corresponds to the quotient of the complement of the divisor $\lbrace a_1=0\rbrace$.

\begin{df}
We define $\Inv(\mathcal{M}_{1,1}\smallsetminus\lbrace j=0 \rbrace,{\rm H}_{2^r})_{{\rm tm}/U}$ as the subgroup of cohomological invariants of $\mathcal{M}_{1,1}\smallsetminus\lbrace j=0 \rbrace$ that, once pulled back to ${\rm H}^{\bullet}_{2^r}(\bfk(U))$, are tamely ramified on $U$ or, equivalently, are tamely ramified at $\lbrace a_1=0 \rbrace$. We refer to these elements as \emph{invariants that are tame on} $U$.

Similarly, we define
$\Inv(\bA^1\smallsetminus\lbrace 0 \rbrace, {\rm H}_{2^r})_{{\rm tm}/U}$ as the intersection
\[\Inv(\mathcal{M}_{1,1}\smallsetminus\lbrace j=0 \rbrace,{\rm H}_{2^r})_{{\rm tm}/U}\cap \Inv(\bA^1\smallsetminus\lbrace 0 \rbrace, {\rm H}_{2^r})\]
and $(\Inv(\bA^1\smallsetminus\lbrace 0 \rbrace, \K_{2})\cdot [\alpha])_{{\rm tm}/U}$ as the intersection
\[\Inv(\mathcal{M}_{1,1}\smallsetminus\lbrace j=0 \rbrace,{\rm H}_{2^r})_{{\rm tm}/U}\cap \Inv(\bA^1\smallsetminus\lbrace 0 \rbrace, \K_{2})\cdot [\alpha].\]
\end{df}

Let $F$ be the field $\bfk(a_1,\ldots, a_6)$ and let $v$ be the discrete valuation determined by $a_1$. The map $U \xrightarrow{\pi} \bA^1$ induces a map
\[
\tilde{\pi}^*: \widetilde{{\rm H}}_{2^r}^{\bullet}(\bfk(j)) \to \widetilde{{\rm H}}_{2^r}^{\bullet}(\bfk_v(a_1)),
\]
where $\bfk_v=\bfk(a_2,a_3,a_4,a_6)$.
\begin{lm}\label{lm:tame Gm}
The map $\tilde{\pi}^*$ is injective and we have
\[ \Inv(\bA^1 \smallsetminus \lbrace 0 \rbrace,{\rm H}_{2^r})_{{\rm tm}/U} = \Inv(\bA^1,{\rm H}_{2^r})\oplus {\rm H}_{2^r}^{\bullet -1}(\bfk)\cdot\lbrace j \rbrace. \]
Moreover, if $\beta\in {\rm H}^{n+1}_p(\bfk(j))$ belongs to ${\rm U}_{s}\smallsetminus {\rm U}_{s-1}$, then $\pi^*\beta$ belongs to ${\rm U}_{12s}\smallsetminus {\rm U}_{12s-4}$.
\end{lm}
\begin{proof}
Let $\alpha \in {\rm H}_{p^r}^{\bullet}(\bfk(j))$ be a wildly ramified element. Let the morphisms $\pi_s$ and $\iota_s$ be defined as in \Cref{lm:truncation}. \Cref{prop:wild gen} tells us that there must be an $i$ such that $\alpha=\iota_i \alpha' + \alpha_{\rm tm}$ and $\beta=\pi_{i-1}\alpha' \in {\rm H}_{p}^{\bullet}(\bfk(j))$ is wildly ramified. Now, if the pullback of $\beta$ is wildly ramified so is the pullback of $\alpha$. Thus we have reduced our proof to the case $r=1$.

Applying \Cref{lm:simple form} we can rewrite $\beta$, up to a tamely ramified element, as 
\begin{equation}\label{eq:beta writing} \beta = \sum_{i=1}^s j^{-i}\varphi_i + j^{-i}\frac{{\rm d}j^{-1}}{j^{-1}}\wedge\varphi'_i, \quad \varphi_i \in \Omega^n_\bfk,\varphi'_i\in\Omega^{n-1}_\bfk \end{equation}
where $\varphi'_i=0$ if $p\nmid i$.
On $U$, we have 
\[j^{-1} =\frac{a_1^4 a_2 a_3^2 + a_1^3 a_3^3 + a_3^4 + a_1^5 a_3 a_4 + a_1^4 a_4^2 + a_1^6 a_6}{a_1^{12}} = \frac{a_2 a_3^2}{a_1^8} + \frac{a_3^3}{a_1^9} + \frac{a_3^4}{a_1^{12}} + \frac{a_3 a_4}{a_1^7} + \frac{a_4^2}{a_1^8} + \frac{a_6}{a_1^6}. \]

Note that $dj^{-1}$ equals
\[
(\frac{a_3^3}{a_1^9}+\frac{a_3a_4}{a_1^7})\frac{{\rm d}a_1}{a_1} + \frac{a_2a_3^2}{a_1^8}\frac{{\rm d}a_2}{a_2} + (\frac{a_3^3}{a_1^9}+\frac{a_3a_4}{a_1^7})\frac{{\rm d}a_3}{a_3}+\frac{a_3a_4}{a_1^7}\frac{{\rm d}a_4}{a_4}+\frac{a_6}{a_1^6}\frac{{\rm d}a_6}{a_6},
\]
so by direct computation
\[\pi^*\left(j^{-i}\frac{{\rm d}j^{-1}}{j^{-1}}\wedge \varphi'_i\right) = (\pi^*j)^{-i+1}{\rm d}\pi^*{j^{-1}} \wedge \varphi_i'\]
belongs to ${\rm U}_{12i-3}$.

Let $-s$ be the lowest power of $j$ appearing in \eqref{eq:beta writing}. We can assume that $j^{-s}$ appears coupled with an element $\varphi_s+\frac{{\rm d}j^{-1}}{j^{-1}}\wedge\varphi'_s$ that is non-zero if $s$ is odd, and $(\varphi_s,\varphi'_s)$ is not closed in $\Omega_\bfk^n\oplus\Omega_\bfk^{n-1}$ if $s$ is even.

We first consider the case where $s$ is even. If $\varphi_s$ is not closed, then the lowest degree term in $\pi^*\beta$ is $a_1^{-12s}(a_3^{4s}\varphi_s)$. Observe that ${\rm d}(a_3^{4s}\varphi_s)=a_3^{4s}{\rm d}(\varphi_s) \neq 0$.  This shows that $\pi^*\beta$ is not $0$ in ${\rm U}_{12s}/{\rm U}_{12s-1}$. 

If $\varphi_s$ is closed, then $\varphi_s'$ is not closed. By \Cref{lm:simple form} up to elements of degree $-s/2$ we may assume that $\varphi_s=0$. Then $\pi^*\beta$ is equal in ${\rm U}_{12s-3}/{\rm U}_{12s-4}$ to 
\[
\pi^*{j}^{-s+1}(\frac{a_3^3}{a_1^9}\frac{{\rm d}a_1}{a_1} +  \frac{a_3^3}{a_1^9}\frac{{\rm d}a_3}{a_3})\wedge \varphi'_s= \frac{a_3^{4s+3}}{a_1^{12s-3}} (\frac{{\rm d}a_1}{a_1}+\frac{{\rm d}a_3}{a_3})\wedge \varphi'_s.
\]
Observe that as
\begin{align*}
    0 = {\rm d}\left(\frac{a_3^{4s+3}}{a_1^{12s-3}}\varphi'_s\right) = \frac{a_3^{4s+3}}{a_1^{12s-3}}{\rm d}(\varphi'_s) + \frac{a_3^{4s+3}}{a_1^{12s-3}}\frac{{\rm d}a_1}{a_1}\wedge\varphi'_s + \frac{a_3^{4s+3}}{a_1^{12s-3}}\frac{{\rm d}a_3}{a_3}\wedge\varphi'_s
\end{align*} 
we get that 
\[
\pi^*\beta = \frac{a_3^{4s-1}}{a_1^{12s-3}}{\rm d}(\varphi'_s) \neq 0
\]
in ${\rm U}_{12s-3}/{\rm U}_{12s-4}$.

Now assume $s$ is odd. If $\varphi_s$ is not closed the same argument as above allows us to conclude. Assume that ${\rm d}(\varphi_s)=0$. In ${\rm U}_{12s}/{\rm U}_{12s-1}$ we have
\[
\pi^*\beta = \frac{a_3^{4s}}{a_1^{12s}}\varphi
\]
which is closed, and by \Cref{cor:lifting} this shows that this term actually belongs to ${\rm U}_{6s}$. This shows that in ${\rm U}_{12s-3}/{\rm U}_{12s-4}$ we have
\[
\pi^*\beta = \frac{a_3^{4s-1}}{a_1^{12s-3}}\varphi
\]
which is nonzero as $a_3^{4s-1}\varphi \neq 0$.

By \Cref{lm:inv open part} we can regard $\Inv(\bA^1\smallsetminus\lbrace 0 \rbrace, {\rm H}_{2^r})$ as a subgroup of $\Inv(\mathcal{M}_{1,1}\smallsetminus\lbrace j=0 \rbrace,{\rm H}_{2^r})$. Then the invariants in this subgroup that are tame on $U$ are those that are sent to zero by the map \[\Inv(\bA^1\smallsetminus\lbrace 0 \rbrace, {\rm H}_{2^r})\longrightarrow \widetilde{{\rm H}}^{\bullet}_{2^r}(\bfk(j)).\]
By looking at the exact sequence for the invariants of $\bA^1$ (\Cref{cor:inv A1}), we see that the kernel of this map coincides with the image of the injective homomorphism 
\[ \Inv(\bA^1, {\rm H}_{2^r}) \longrightarrow \Inv(\bA^1\smallsetminus\lbrace 0 \rbrace, {\rm H}_{2^r}).\]
This concludes the proof.
\end{proof}

\begin{lm}\label{lm:tame alpha}
We have  
\[(\Inv(\bA^1\smallsetminus\lbrace 0 \rbrace, \K_{2})\cdot [\alpha])_{{\rm tm}/U}=\K_2^{\bullet-1}(\bfk)\cdot [\alpha,j\rbrace. \]
\end{lm}

\begin{proof}
We have $\Inv(\bA^1\smallsetminus\lbrace 0 \rbrace, \K_2)\simeq \K_2^{\bullet}(\bfk)\oplus \K_2^{\bullet-1}(\bfk)\cdot\{j\}$ due to \Cref{eq:KA1}. We first check that $[\alpha,\Delta \}$ is tame on $U$: this implies that $[\alpha, j \rbrace$ is tame, because $\{\Delta\}=\{j^{-1}\}= -\{ j\}$. 

Using the formula for the pullback of $[\alpha]$ given in \Cref{rmk:alpha} and the formula for $\Delta$ in characteristic two, we can write the pullback of $[\alpha,\Delta\}$ to $U$ as a differential form as
\begin{small}
  \[
       \left(\frac{a_1a_2+a_3}{a_1^3} \right) \Biggl( \frac{(a_1^2a_3)(a_1^2a_4+a_3^2)\cdot da_1}{\Delta} + \frac{a_1^4(a_1^2+a_3^2)\cdot da_2}{\Delta}  
    + \frac{a_1^3(a_1^2a_4+a_3^2)\cdot da_3}{\Delta} + \frac{a_1^4a_3\cdot da_4 }{\Delta} \Biggr)
\]  
\end{small}
which, after some manipulations, can be written as
\begin{small}
 \[ \left(\frac{a_1a_2+a_3}{\Delta} \right) \Biggl( a_3(a_1^2a_4+a_3^2) \frac{{\rm d}a_1}{a_1} + a_1a_2(a_1^2+a_3^2)\cdot \frac{{\rm d}a_2}{a_2}  
    + a_3(a_1^2a_4+a_3^2)\cdot \frac{{\rm d}a_3}{a_3} + a_1a_3a_4\cdot \frac{{\rm d}a_4 }{a_4} \Biggr).
\]   
\end{small}
    
All the coefficients in the formula have non negative valuation at $\lbrace a_1=0 \rbrace$, hence this element is tamely ramified by \Cref{thm:wild pres}.

To conclude the proof, we only need to check that $\delta \cdot \left[\alpha\right]$ is wildly ramified at $U$ for any $\delta \in \K_2^{\bullet}(\bfk)$. Again, using the formula of \Cref{rmk:alpha}, we get
\[ \pi^*[\alpha] = [(0,\ldots,0,(a_1a_2+a_3)/a_1^3)] \]
which is wildly ramified at $\lbrace a_1=0 \rbrace$: this follows from \Cref{thm:wild pres}, as the form 
\[\frac{a_1a_2+a_3}{a_1^3} \varphi, \, \varphi \in \Omega^n_{\bfk, {\rm log}}\]
is sent to $a_3/a_1^3 \varphi \neq 0$ in ${\rm U}_3/{\rm U}_2$, which coincides with the class of $\pi^*[\alpha]$ in this quotient (we are implicitly using the fact that $\pi^*[\alpha]$ belongs by definition to the subgroup ${\rm H}_2^1(\bfk(a_1,\ldots,a_6))$ contained in ${\rm H}_{2^r}^1(\bfk(a_1,\ldots,a_6))$).
\end{proof}

Before proceeding with the computation of the cohomological invariants of $\mathcal{M}_{1,1}$ in characteristic $2$, we need some technical results.

\begin{lm}\label{lm:wild elements}
Let $\pi:U\smallsetminus\{a_1=0\} \to \bA^1\smallsetminus\{0\}$ be the map given by the $j$-invariant and let $\beta$ be a cohomological invariant of $\bA^1\smallsetminus\{0\}$ that is wildly ramified at $\{j=0\}$. Then for every $\delta\in \K_2^n(\bfk)$ the element $\pi^*\beta + \delta\cdot [\alpha]$ is wildly ramified at $\{a_1=0\}$, where $\alpha=(a_2a_1+a_3)/a_1^3$.
\end{lm}

\begin{proof}
Suppose that the statement holds true for $r=1$ and let $\beta$ be an invariant contained in ${\rm H}_{p^r}^{n+1}(\bfk(j))$ with $r>1$. Let $\pi_1$ and $\iota_1$ be defined as in \Cref{lm:truncation}. By definition, the invariant $\delta\cdot[\alpha]$ belongs to the image of $\iota_1$, hence $\pi_1(\pi^*\beta + \delta\cdot[\alpha]) = \pi_1(\pi^*\beta) = \pi^*(\pi_1(\beta))$, where with a little abuse of notation we denoted as $\pi_1$ both the projection on ${\rm H}_p^{n+1}(\bfk(j))$ and the one on ${\rm H}_p^{n+1}(\bfk(a_2,\ldots,a_6)(a_1))$.

If $\pi^*\beta+\delta\cdot[\alpha]$ is tamely ramified at $\{a_1=0\}$, then so it is its image via $\pi_1$, hence by the computation above $\widetilde{\pi}^*(\pi_1(\beta))=0$ in $\widetilde{{\rm H}}_{p}^{n+1}(\bfk(a_2,\ldots,a_6)(a_1))$. By \Cref{lm:tame Gm} this implies that $\pi_1(\beta)=0$ in $\widetilde{H}_{p}^{n+1}(\bfk(j))$, i.e. $\beta$ is equal to $\iota_1(\beta')$ up to tame elements. As the tameness of $\pi^*\beta+\delta\cdot[\alpha]$ is equivalent to the tameness of $\pi^*(\iota_1(\beta'))+\delta\cdot[\alpha]$, we can assume $\beta=\iota_1(\beta')$, that is $\beta$ belongs to ${\rm H}_p^{n+1}(\bfk(j))$. Then we have reduced ourselves to the case $r=1$.

Observe that $\delta\cdot[\alpha]$ belongs to ${\rm U}_3\smallsetminus {\rm U}_2$ (the ${\rm U}_i$ are the subgroups appearing in \Cref{thm:wild pres}), so that if $\beta$ is tame we are done.

Suppose that $\beta$ is in ${\rm U}_s\smallsetminus {\rm U}_{s-1}$, with $s\geq 1$. Then by \Cref{lm:tame Gm} we have that $\pi^*\beta \in {\rm U}_{12s}\smallsetminus {\rm U}_{12s-4}$. If $\pi^*\beta + \delta\cdot[\alpha]$ is tame then it belongs to ${\rm U}_0\subset {\rm U}_{12s-4}$ by \Cref{thm:wild pres}, hence $\pi^*\beta=(\pi^*\beta + \delta\cdot[\alpha])-\delta\cdot [\alpha]$ is in ${\rm U}_{12s-4}$, because $3<12s-4$ for $s\geq 1$. This contradicts the fact that $\pi^*\beta \in {\rm U}_{12s}\smallsetminus {\rm U}_{12s-4}$, and concludes the proof. 
\end{proof}

\begin{prop}\label{pr:tame elements}
We have
\[ \Inv(\mathcal{M}_{1,1}\smallsetminus\{j=0\},{\rm H}_{2^r})_{{\rm tm}/U}\simeq \Inv(\bA^1,{\rm H}_{2^r})\oplus {\rm H}_{2^r}^{\bullet -1}\cdot\{j\}\oplus \K_2^{\bullet - 1}\cdot [\alpha,j\} \]
\end{prop}

\begin{proof}
Let $\beta$ be an invariant of $\mathcal{M}_{1,1}\smallsetminus \{j=0\}$ with coefficients in ${\rm H}_{2^r}$, tame over $U$. From \Cref{lm:inv open part} we can write $\beta=\beta' + \beta''\cdot [\alpha]$, where $\beta'$ is (the pullback of) an invariant of $\bA^1\smallsetminus\{0\}$ with coefficients in ${\rm H}_{2^r}$ and $\beta''$ is (the pullback of) an invariant of $\bA^1\smallsetminus\{0\}$ with coefficients in $\K_2$.

We can further rewrite $\beta''$ as $\delta + \delta'\cdot\{j\}$, where $\delta$ and $\delta'$ belongs to $\K^{\bullet}_2(\bfk)$. 

If the image of $\beta'$ is zero in $\widetilde{{\rm H}}^{n+1}_{2^r}(\bfk(a_2,\ldots,a_6)(a_1))$, then $\beta'$ is tamely ramified and hence it belongs to $ \Inv(\bA^1,{\rm H}_{2^r})\oplus {\rm H}_{2^r}^{\bullet -1}\cdot\{j\}$ by \Cref{lm:tame Gm}; moreover, we have then that $\beta$ is tame if and only if $\delta=0$ because $[\alpha,j\}$ is tame and $\delta\cdot [\alpha]$ is not (\Cref{lm:tame alpha}).

This implies that $\beta$ belongs to $\Inv(\bA^1,{\rm H}_{2^r})\oplus {\rm H}_{2^r}^{\bullet -1}\cdot\{j\}\oplus \K_2^{\bullet - 1}\cdot [\alpha,j\}$, as claimed.

Suppose now that the image of $\beta'$ is not zero in $\widetilde{{\rm H}}^{n+1}_{2^r}(\bfk(a_2,\ldots,a_6)(a_1))$. As $\beta$ is tamely ramified by hypothesis, this implies that $\beta'+\delta\cdot [\alpha]$ must be zero in $\widetilde{{\rm H}}^{n+1}_{2^r}(\bfk(a_2,\ldots,a_6)(a_1))$, which contradicts \Cref{lm:wild elements}. This finishes the proof.
\end{proof}

Let $\mu:\K_2^{\bullet-1} (\bfk) \to {\rm H}^{\bullet}_{2^r}(\bfk)$ be the map that sends 
\[\{x_1,\dots,x_n\}\longmapsto [(0,\ldots,0,1),x_1,\ldots,x_n\}\]
and let ${\rm H}_{2^r}^{\bullet}(\bfk)\overset{4}{\to} {\rm H}_{2^r}^{\bullet}(\bfk)$ be the multiplication by $4$.
In the category of abelian groups, we can form the cartesian diagram
\[
\begin{tikzcd}
{\rm J}_{2^r}^{\bullet}(\bfk) \ar[r] \ar[d] & \K_2^{\bullet-1} (\bfk) \ar[d, "\mu"]\\
{\rm H}_{2^r}^{\bullet}(\bfk) \ar[r, "4"] & {\rm H}_{2^r}^{\bullet}(\bfk).
\end{tikzcd}
\]
More concretely, one can think of ${\rm J}_{2^r}^{\bullet}(\bfk)$ as the subgroup of ${\rm H}_{2^r}^{\bullet}(\bfk)\oplus\K_2^{\bullet-1}(\bfk)$ formed by those pairs $([w,\underline{x}\},\{\underline{y}\})$ such that $4[w,\underline{x}\}=[(0,\ldots,0,1),\underline{y}\}$.
\begin{rmk}\label{rmk:J}
If $r\leq 2$, then the multiplication by $4$ is equal to the zero map, hence in this range ${\rm J}^{\bullet}_{2^r}(\bfk)$ is equal to ${\rm H}_{2^r}^{\bullet}(\bfk)\oplus \ker(\mu)$.

If $r\geq 3$, then ${\rm J}_{2^r}^{\bullet}(\bfk)$ is equal to ${\rm J}_8^{\bullet}(\bfk)$: indeed we have by definition that for every element in ${\rm J}_{2^r}^{\bullet}(\bfk)$ we have $4[w,\underline{x}\}=[(0,\ldots,0,1),\underline{y}\}$, which immediately implies that $8[w,\underline{x}\}=0$. By \Cref{prop:ptors} we deduce that $[w,\underline{x}\}$ belongs to ${\rm H}_8^{\bullet}(\bfk)$, thus the pair $([w,\underline{x}\},\{\underline{y}\})$ belongs to ${\rm J}_8^{\bullet}(\bfk)$.

Finally, observe that if $x^2+x+1$ has a solution in $\bfk$, we have $[(0,\ldots,0,1)]=0$ in ${\rm H}^1_{2^r}(\bfk)$, thus ${\rm J}^{\bullet}_{2^r}(\bfk)$ is actually equal to ${\rm H}^{\bullet}_{4}(\bfk) \oplus \K^{\bullet -1}_2(\bfk)$.
\end{rmk}

\begin{thm}\label{thm:InvM11H2}
Suppose that the base field $\bfk$ has characteristic two. Then
\[\Inv(\mathcal{M}_{1,1},{\rm H}_{2^r})\simeq \Inv(\bA^1,{\rm H}_{2^r})\oplus {\rm J}_{2^r}^{\bullet-1}(\bfk)\]
\end{thm}

\begin{proof}
We know from \Cref{pr:tame elements} that
\[ \Inv(\mathcal{M}_{1,1}\smallsetminus\{j=0\},{\rm H}_{2^r})_{{\rm tm}/U}\simeq \Inv(\bA^1,{\rm H}_{2^r})\oplus {\rm H}_{2^r}^{\bullet -1}\cdot\{j\}\oplus \K_2^{\bullet - 2}\cdot [\alpha,j\}.  \]
The cohomological invariants of $\mathcal{M}_{1,1}$ are those elements in the group above whose pullback to $U$ is unramified along $\{a_1=0\}$ (\Cref{cor:invariants are tame and unramified}).

The elements coming from $\bA^1$ are obviously unramified, so we only have to check when the pullback to $U$ of an element of the form $[w,j,\underline{x}\}+[\alpha,j,\underline{y}\}$ is unramified at $\{a_1=0\}$.

We have seen in the proof of \Cref{lm:tame alpha} that can rewrite the pullback of $[\alpha,j\}$ to $U$ as
 \begin{small}
 \[
       \left(\frac{a_1a_2+a_3}{\Delta} \right) \Biggl( a_3(a_1^2a_4+a_3^2) \frac{{\rm d}a_1}{a_1} + a_1a_2(a_1^2+a_3^2)\cdot \frac{{\rm d}a_2}{a_2}  
    + a_3(a_1^2a_4+a_3^2)\cdot \frac{{\rm d}a_3}{a_3} + a_1a_3a_4\cdot \frac{{\rm d}a_4 }{a_4} \Biggr),
\]
\end{small}
from which we deduce by direct computation that the residue at $\{a_1=0\}$ of $[\alpha,j,\underline{y}\}$ is equal to $[(0,\ldots,0,1),\underline{y}\}$.

On the other hand, the pullback of $\{j\}$ to $U$ is equal to $\{a_1^{12}/\Delta\}$, whose residue at $\{a_1=0\}$ is $12$. This implies that the residue of $[w,j,\underline{x}\}$ is equal to $12[w,\underline{x}\}=4[w,\underline{x}\}$. 

We deduce that $[w,j,\underline{x}\}+[\alpha,\underline{y}\}$ is unramified if and only if the element $4[w,\underline{x}\}+[(0,\ldots,0,1),\underline{y}\}$ is equal to zero. This concludes the proof.
\end{proof}

\begin{rmk}
If $\bfk$ has a primitive cubic root of unity, i.e. the polynomial $x^2+x+1$ is not irreducible in $\bfk$, then it follows from \Cref{rmk:J} that the invariants of $\mathcal{M}_{1,1}$ are 
\[ \Inv(\bA^1,{\rm H}_{2^r})\oplus {\rm H}^{\bullet-1}_4(\bfk)\cdot \{\Delta\} \oplus \K^{\bullet-2}_2\cdot [\alpha,\Delta\}. \]
Indeed, for any $\beta\in {\rm H}_{4}^{\bullet-1}(\bfk)$, we have that the pullback of $\beta\cdot\{j\}$ to $U$ is equal to $\beta\cdot (12\{a_1\}-\{\Delta\})=-\beta\cdot\{\Delta\}$, from which follows that $\{\Delta\}$ is a generator for this subgroup.
Similarly, for $\beta\in \K^{\bullet-2}_2(\bfk)$, we have $\beta'\cdot[\alpha,j\} = \beta'\cdot 12[\alpha,a_1\} - \beta'\cdot [\alpha,\Delta\}= \beta'\cdot [\alpha,\Delta\}$.
\end{rmk}

\begin{rmk}\label{rmk:essM11}
\Cref{thm:InvM11H2}, in conjunction with \Cref{prop:essential}, show that when ${\rm char}(\bfk)=2$ we have
\[
{\rm ed}(\Mcal_{1,1}) \geq {\rm ed}_2(\Mcal_{1,1}) \geq 2
\]
as the group ${\rm J}^{1}\left(\overline{\bfk}\right)$ is isomorphic to $\K^0_2\left(\overline{\bfk}\right)=\ZZ/2$.

The essential dimension of $\Mcal_{1,1}$ is known to be $2$ over any field \cite{BrReVi}*{Section 9.5}, so \Cref{thm:inv M11 H} shows that ${\rm char}(\bfk)=2$ is the only characteristic where cohomological invariants provide a sharp lower bound for the essential dimension of $\Mcal_{1,1}$.
\end{rmk}

\subsection{Characteristic 3}\label{sec:char 3}
In this subsection we assume that the base field $\bfk$ has characteristic three, and our goal is to compute $\Inv(\mathcal{M}_{1,1},{\rm H}_{3^r})$.
Recall \cite{Sil}*{Section III.1, pg. 42} that if ${\rm{char}}(\bfk)=3$ then
\[ \mathcal{M}_{1,1} \simeq [\Spec(\bfk[b_2,b_4,b_6,\Delta^{-1}])/\Gm \ltimes \Ga]  \]
where the the group $\Gm\ltimes\Ga$ acts on $\Spec(\bfk[b_2,b_4,b_6,\Delta^{-1}])$ as:
\begin{align*}
    &b_2 \longmapsto u^{-2}b_2, \\
    &b_4 \longmapsto u^{-4}(b_4+rb_2), \\
    &b_6 \longmapsto u^{-6}(b_6-rb_4+r^2b_2+r^3).
\end{align*}
The isomorphism $[U/G]\simeq [\Spec(\bfk[b_2,b_4,b_6,\Delta^{-1}])/\Gm \ltimes \Ga]$ is given by expressing the coordinates $b_i$ in terms of the coordinates $a_i$ via the formulas of \ref{sub:setup}. In particular, we have
\begin{align*}
    \Delta = -b_2^3b_6 + b_2^2b_4^2 + b_4^3,\quad j= b_2^6/\Delta.
\end{align*}

Similarly to what happens in characteristic two, also in characteristic three we have
\[ \mathcal{M}_{1,1}\smallsetminus \{j=0\} \simeq (\bA^1\smallsetminus\{0\})\times\Brm\ZZ/2. \]
This can be proven exactly as in the characteristic two case, using the fact that for every elliptic curve $C$ with $j(C)\neq 0$, the only non-trivial automorphism is the involution. We can also write
\[ \mathcal{M}_{1,1}\smallsetminus \{j=0\} \simeq [\Spec(\bfk[b_2^{\pm 1},b_4,b_6,\Delta^{-1}])/\Gm\ltimes\Ga]. \]
In other terms, the invariant subscheme in $\Spec(\bfk[b_2,b_4,b_6,\Delta^{-1}])$ that parametrizes curves with $j$-invariant equal to zero is the divisor $\{b_2=0\}$.

We can apply \Cref{lm:invariants Zq} to deduce
\[ \Inv(\mathcal{M}_{1,1}\smallsetminus \{j=0\}, {\rm H}_{3^r}) \simeq \Inv(\bA^1\smallsetminus\{0\},{\rm H}_{3^r}). \]
Set $U'\overset{\textnormal{def}}{=} \Spec(\bfk[b_2,b_4,b_6,\Delta^{-1}])$, so that $U'\to \mathcal{M}_{1,1}$ is a $\Gm\ltimes\Ga$-torsor.

\begin{df}
We define $\Inv(\mathcal{M}_{1,1}\smallsetminus\lbrace j=0 \rbrace,{\rm H}_{3^r})_{{\rm tm}/U'}$ as the subgroup of cohomological invariants of $\mathcal{M}_{1,1}\smallsetminus\lbrace j=0 \rbrace$ that, once pulled back to ${\rm H}^{\bullet}_{3^r}(\bfk(U'))$, are tamely ramified on $U'$ or, equivalently, are tamely ramified at $\lbrace b_2=0 \rbrace$. We refer to these elements as \emph{invariants that are tame on} $U'$.
\end{df}
Analogously to what we did in characteristic two, our strategy consists in first finding the invariants of $\mathcal{M}_{1,1}\smallsetminus\{j=0\}$ that are tamely ramified on $U'$, and then checking what are the ones that are unramified using the residue homomorphism. 

Let $\pi:(U'\smallsetminus\{b_2=0\}) \to \bA^1$ be the composition of the quotient map and the $j$-invariant, so that there is an induced morphism 
\[\widetilde{\pi}^*: \widetilde{{\rm H}}_{3^r}(\bfk(j)) \longrightarrow \widetilde{{\rm H}}_{3^r}(\bfk(b_4,b_6)(b_2)). \]
\begin{lm}\label{lm:tame Gm char 3}
The homomorphism $\widetilde{\pi}^*$ is injective and 
\[ \Inv(\mathcal{M}_{1,1} \smallsetminus\{j=0\},{\rm H}_{3^r})_{{\rm tm}/U'} \simeq \Inv(\bA^1,{\rm H}_{3^r})\oplus {\rm H}^{\bullet-1}_{3^r}(\bfk)\cdot\{j\}. \]
\end{lm}
\begin{proof}
The proof adopts the same strategy of \Cref{lm:tame Gm}, and the computations are quite similar.
First observe that
\[ j^{-1} = -\frac{b_6}{b_2^3} + \frac{b_4^2}{b_2^4} + \frac{b_4^3}{b_2^6},\quad dj^{-1} = \left(\frac{2b_6}{b_2^3}\right)\frac{{\rm d}b_6}{b_6} + \left(\frac{2b_4^2}{b_2^4}\right)\frac{{\rm d}b_4}{b_4} + \left(\frac{2b_4^2}{b_2^4}\right)\frac{{\rm d}b_2}{b_2} \]
To prove that $\widetilde{\pi}^*$ is injective, we can reduce to the case $r=1$ using exactly the same strategy adopted in the proof of \Cref{lm:tame Gm}. 
Let $\beta$ be a wildly ramified element of ${\rm H}_{3}^{n+1}(\bfk(j))$. We can apply \Cref{lm:simple form} to write it down, up to a tamely ramified element, as
\begin{equation}\label{eq:simple form 3} \beta = \sum_{i\geq 1} j^{-i}\wedge\varphi_i + j^{-i}\frac{{\rm d}j^{-1}}{j^{-1}}\wedge\varphi_i', \end{equation}
where $\varphi_i$ and $\varphi_i'$ come respectively from $\Omega^n_{\bfk}$ and $\Omega^{n-1}_{\bfk}$. Moreover, we can assume that if $3\nmid i$ then $\varphi_i\neq 0$ and $\varphi'_i=0$, whether if $3\mid i$ then at least one of $\varphi_i$ and $\varphi'_i$ is not closed.

Let $s$ be the highest power of $j^{-1}$ appearing in \eqref{eq:simple form 3}. Suppose that $3\mid s$, and that $\varphi_s$ is not closed.
Then in ${\rm U}_{6s}/{\rm U}_{6s-1}$ we have $\pi^*\beta = b_2^{-6s}(b_4^{3s}\varphi_s)$, hence if $b_4^{3s}\varphi_s$ is not closed, we can conclude by \Cref{thm:wild pres} that $\pi^*\beta$ is non-zero in ${\rm U}_{6s}/{\rm U}_{6s-1}$. We have ${\rm d}(b_4^{3s}\phi_s)=b_4^{3s}{\rm d}\phi_s \neq 0$ because ${\rm d}\phi_s\neq 0$ by hypothesis.

Suppose now that $3\mid s$, the element $\varphi_s=0$ but $\varphi'_s$ is not closed. Then in ${\rm U}_{6s-2}/{\rm U}_{6s-3}$ we have
\begin{equation}\label{eq:char 3 s even 2}
    b_2^{-(6s-2)}(2b_4^2\frac{{\rm d}b_4}{b_4}\wedge\phi'_s) + b_2^{-(6s-2)}(2b_4^2\frac{{\rm d}b_2}{b_2}\wedge\phi'_s).
\end{equation}
Observe that in ${\rm H}^{n+1}_3(\bfk(b_2,b_4,b_6))$ we have
\begin{equation*}
    0={\rm d}(b_2^{-(6s-2)}b_4^2\varphi'_s)=b_2^{-(6s-2)}(2b_4^2\frac{{\rm d}b_4}{b_4}\wedge\phi'_s) + b_2^{-(6s-2)}(2b_4^2\frac{{\rm d}b_2}{b_2}\wedge\phi'_s) + b_2^{-(6s-2)}(b_4^2{\rm d}\phi'_s)
\end{equation*}
hence \eqref{eq:char 3 s even 2} is equal to $-b_2^{-(6s-2)}(b_4^2{\rm d}\phi'_s)$, and the latter is non-zero in ${\rm U}_{6s-2}/{\rm U}_{6s-3}$ by \Cref{thm:wild pres}, because $b_4^2{\rm d}\phi'_s\neq 0$.

Finally, suppose $3\nmid s$, so that we only know that $\varphi_s\neq 0$. The highest term in the pullback of $\beta$ is $b_2^{-6s}(b_4^{3s}\varphi_s)$. If $\varphi_s$ is not closed, then we can conclude as before, so suppose the contrary. In this case, applying \Cref{cor:lifting} we see that this term can be rewritten as $b_2^{-3s}\varphi_s$, so that now $\pi^*\beta$ is equal in ${\rm U}_{6s-2}/{\rm U}_{6s-3}$ to $b_2^{-6s+2}(sb_4^{3s-1}\varphi_s)$. As $-6s+2$ is not divisible by $3$, the form $b_4^{3s-1}\varphi_s$ is non-zero and $s$ is invertible, we can again conclude that the pullback of $\beta$ is non-zero in ${\rm U}_{6s-2}/{\rm U}_{6s-3}$.

Putting all together, we proved that $\widetilde{\pi}^*$ is injective.

From \Cref{cor:inv A1} we have that the invariants of $\bA^1\smallsetminus\{j=0\}$ fits into the exact sequence
\[ 0\to \Inv(\bA^1,{\rm H}_{3^r})\oplus {\rm H}_{3^r}^{\bullet-1}\cdot\{j\} \to \Inv(\bA^1\smallsetminus\{0\},{\rm H}_{3^r}) \to \widetilde{{\rm H}}^{\bullet}_{3^r}(\bfk(j)) \to 0. \]
Let $\pi:(U'\smallsetminus\{b_2=0\}) \to \bA^1$ be the composition of the quotient map and the $j$-invariant. Then 
\[ \pi^*[w,j,\underline{x}\}= [w,b_2^6/\Delta,\underline{x}\} = [w,b_2^6,\underline{x}\} - [w,\Delta,\underline{x}\} \]
from which we deduce that $\pi^*(\delta\cdot\{j\})$ is tamely ramified at $\{b_2=0\}$ for every $\delta\in {\rm H}_{3^r}^{\bullet-1}(\bfk)$.
The invariants of $\bA^1$ are obviously tamely ramified on $U'$. This, together with the injectivity of $\widetilde{\pi}^*$, gives the claimed description of the subgroup of invariants that are tame on $U'$.
\end{proof}

\begin{thm}\label{thm:InvM11H3}
Suppose that the base field $\bfk$ has characteristic three. Then we have
\[ \Inv(\mathcal{M}_{1,1},{\rm H}_{3^r}) \simeq \Inv(\bA^1,{\rm H}_{3^r}) \oplus {\rm H}_3^{\bullet-1}(\bfk)\cdot \{\Delta\}. \]
\end{thm}
\begin{proof}
We know from \Cref{lm:tame Gm char 3} that the cohomological invariants of $\mathcal{M}_{1,1}$ correspond to the unramified elements in 
\[ \Inv(\bA^1,{\rm H}_{3^r})\oplus {\rm H}^{\bullet-1}_{3^r}(\bfk)\cdot\{j\}. \]
Of course, the invariants coming from $\bA^1$ are unramified, so that we only have to check the invariants of the form $\delta\cdot\{j\}$, where $\delta\in {\rm H}^{\bullet-1}_{3^r}$. 

This can be done on the cover $U'=\Spec(\bfk[b_2,b_4,b_6,\Delta^{-1}])\to\mathcal{M}_{1,1}$, where we only have to check the unramifiedness at $\{b_2=0\}$. When we pullback $\{j\}$ to $U'$ we get $\{ b_2^6/\Delta \}$, so that the residue of $\delta\cdot\{j\}$ at $\{b_2=0\}$ is equal to $6\delta$, which is zero if and only if $\delta$ is of $3$-torsion.

By \Cref{prop:ptors} the elements of $3$-torsion in ${\rm H}_{3^r}^{\bullet-1}(\bfk)$ are those that belongs to ${\rm H}_{3}^{\bullet-1}(\bfk)$. This, together with the fact that $\{j\}=\{b_2^6/\Delta\}=-\{\Delta\}$ in $\K_3^1(\bfk)$, concludes the proof.
\end{proof}

\subsection{Characteristic $p>3$}\label{sec:char p}
In this section we assume that the base field $\bfk$ has characteristic $p>3$. Our goal is to compute the cohomological invariants of $\mathcal{M}_{1,1}$ with coefficients in ${\rm H}_{p^r}$.

Recall \cite{Sil}*{Chapter III, Section 1} that with these hypotheses on $\bfk$ the stack $\mathcal{M}_{1,1}$ has the following presentation as a quotient stack
\[ \mathcal{M}_{1,1} \simeq [ \Spec(\bfk[c_4,c_6,\Delta^{-1}])/\Gm ], \]
where the multiplicative group acts by $(c_4,c_6)\mapsto (u^{-4}c_4,u^{-6}c_6)$ and $\Delta=\frac{c_4^3-c_6^2}{1728}$. The $j$-invariant in this case is $j=1728 \frac{c_4^3}{c_4^3-c_6^2}$, and we have 
\[\mathcal{M}_{1,1}\smallsetminus \{j=0,1728\}\simeq(\bA^1\smallsetminus\{0,1728\})\times\Brm\ZZ/2.\]
As we are in characteristic $p>3$, we can apply \Cref{lm:invariants Zq} to deduce
\[ \Inv(\mathcal{M}_{1,1}\smallsetminus\{j=0,1728\},{\rm H}_{p^r}) \simeq \Inv(\bA^1\smallsetminus\{0,1728\},{\rm H}_{p^r}). \]
Using \Cref{thm:tame over P1} we immediately deduce that there is a short exact sequence
\[ 0 \longrightarrow {\rm N}^{\bullet} \longrightarrow \Inv(\bA^1\smallsetminus\{0,1728\},{\rm H}_{p^r}) \longrightarrow {\rm P}^{\bullet} \longrightarrow 0 \]
where:
\begin{align*}
    {\rm N}^{\bullet} &\overset{\textnormal{def}}{=} \Inv(\bA^1,{\rm H}_{p^r})\oplus {\rm H}^{\bullet-1}_{p^r}(\bfk)\cdot\{j\}\oplus {\rm H}_{p^r}^{\bullet-1}(\bfk)\cdot\{j-1728\}, \\
    {\rm P}^{\bullet} &\overset{\textnormal{def}}{=} \widetilde{{\rm H}}_{p^r}^{\bullet}(\bfk(j))\oplus \widetilde{{\rm H}}_{p^r}^{\bullet}(\bfk(j-1728)).
\end{align*}
Set $U''\overset{\textnormal{def}}{=}\Spec(\bfk[c_4,c_6,\Delta^{-1}])$. As before, we introduce the definition of elements \emph{tame on }$U''$.
\begin{df}
We define $\Inv(\mathcal{M}_{1,1}\smallsetminus\lbrace j=0,1728 \rbrace,{\rm H}_{p^r})_{{\rm tm}/U''}$ as the subgroup of cohomological invariants of $\mathcal{M}_{1,1}\smallsetminus\lbrace j=0,1728 \rbrace$ that, once pulled back to ${\rm H}^{\bullet}_{p^r}(\bfk(U''))$, are tamely ramified on $U''$. We refer to these elements as \emph{invariants that are tame on} $U''$.
\end{df}
To compute the invariants of $\mathcal{M}_{1,1}$, we first determine which invariants are tame on $U''$. For this, let us denote $\pi:U''\to \bA^1$ the composition of the quotient map together with the morphism given by the $j$-invariant. Observe that the preimage of $0$ in $U''$ corresponds to the divisor $\{c_4=0\}$, and the preimage of $1728$ corresponds to $\{ c_6 = 0 \}$, so that we have pullback homomorphisms
\[ \widetilde{\pi}_{1728}^*: \widetilde{{\rm H}}_{p^r}^{\bullet}(\bfk(j-1728)) \to \widetilde{{\rm H}}_{p^r}^{\bullet}(\bfk(c_4)(c_6)), \quad \widetilde{\pi}_0^*: \widetilde{{\rm H}}_{p^r}^{\bullet}(\bfk(j)) \to \widetilde{{\rm H}}_{p^r}^{\bullet}(\bfk(c_6)(c_4)). \]
We use these maps to determine the invariants that are tame on $U''$.
\begin{lm}\label{lm:tame Gm char p}
The pullback homomorphisms $\widetilde{\pi}_0$ and $\widetilde{\pi}_{1728}$ are both injective and
\begin{align*} \Inv(\mathcal{M}_{1,1}\smallsetminus\lbrace j=0,&1728 \rbrace,{\rm H}_{p^r})_{{\rm tm}/U''} \simeq \\
&\simeq\Inv(\bA^1,{\rm H}_{p^r})\oplus {\rm H}^{\bullet-1}_{p^r}(\bfk)\cdot\{j\}\oplus {\rm H}_{p^r}^{\bullet-1}(\bfk)\cdot\{j-1728\}. \end{align*}
\end{lm}
\begin{proof}
Using exactly the same argument adopted in the proof of \Cref{lm:tame Gm}, we reduce to the case $r=1$. First we deal with the morphism $\widetilde{\pi}_0^*$. Using \Cref{lm:simple form}, given a wildly ramified element $\beta$ in ${\rm H}_{p^r}^{\bullet}(\bfk(j))$ we can rewrite it as
\begin{equation}\label{eq:simple form p} \beta = \sum_{i\geq 1} j^{-i}\wedge\varphi_i + j^{-i}\frac{{\rm d}j^{-1}}{j^{-1}}\wedge\varphi_i', \end{equation}
where $\varphi_i$ and $\varphi_i'$ come respectively from $\Omega^n_{\bfk}$ and $\Omega^{n-1}_{\bfk}$. Moreover, we can assume that if $p\nmid i$ then $\varphi_i\neq 0$ and $\varphi'_i=0$, whether if $p\mid i$ then at least one of $\varphi_i$ and $\varphi'_i$ is not closed.
Observe that
\[ j^{-1} = (1728)^{-1}\left( 1 - \frac{c_6^2}{c_4^3} \right), \quad {\rm d}j^{-1} = -(1728)^{-1}\left(\left(\frac{2c_6^2}{c_4^3}\right)\frac{{\rm d}c_6}{c_6} - \left(\frac{3c_6^2}{c_4^3}\right)\frac{{\rm d}c_4}{c_4}\right). \]

Let $s$ be the highest power of $j^{-1}$ appearing in \eqref{eq:simple form p} and suppose $p\mid s$, then in ${\rm U}_{3s}/{\rm U}_{3s-1}$ we have
\[\pi_0^*\beta = (1728)^{-s} c_4^{-3s} (c_6^{2s}(\varphi_s-2{\rm d}c_6\wedge \varphi'_s)) + (-1728)^{-s}c_4^{-3s}(3c_6^{2s}\frac{{\rm d}c_4}{c_4}\wedge\varphi'_s) \]
which by \Cref{thm:wild pres} is non-zero if and only if the pair
\begin{equation*}\label{eq:pair} ((1728)^{-s}c_6^{2s}(\varphi_s-2{\rm d}c_6\wedge \varphi'_s),(1728)^{-s}3c_6^{2s}\varphi'_s) \in \Omega^n/Z^n \oplus \Omega^{n-1}/Z^{n-1} \end{equation*}
is non-zero. We know by hypothesis that $\varphi_s$ and $\varphi'_s$ cannot be both closed. If $\varphi'_s$ is not closed, then we deduce that the element in \eqref{eq:pair} is not zero, because the second entry is not zero. Suppose that $\varphi'_s$ is closed: then the first entry of \eqref{eq:pair} is equal to $(1728)^s c_6^{2s}\varphi_s$, and we have ${\rm d}((1728)^s c_6^{2s}\varphi_s)=(1728)^s c_6^{2s}{\rm d}\varphi_s\neq 0$ by hypothesis. We have shown that the pullback of $\beta$ is not zero in ${\rm U}_{3s}/{\rm U}_{3s-1}$ when $p\mid s$.

If $p\nmid s$, then $p\nmid 3s$, the form $\varphi_s$ is not zero and again by \Cref{thm:wild pres} we can conclude that in ${\rm U}_{3s}/{\rm U}_{3s-1}$ the element $\pi_0^*\beta = (1728)^{-s} c_4^{-3s} (c_6^{2s}\varphi_s)$ is not zero in ${\rm U}_{3s}/{\rm U}_{3s-1}$ because $(1728)^s c_6^{2s}\varphi_s\neq 0$. This implies that $\widetilde{\pi}^*_0\beta\neq 0$, as claimed.

We deal with $\widetilde{\pi}^*_{1728}$ basically in the same way. Observe that
\[ (j-1728)^{-1} = (1728)^{-1}\left( \frac{c_4^3}{c_6^2} - 1 \right), \quad dj^{-1}=(1728)^{-1}\frac{c_4^3}{c_6^2}\left(3\frac{{\rm d}c_4}{c_4}-2\frac{{\rm d}c_6}{c_6}\right). \]
Write
\begin{equation}\label{eq:simple form p tris} \beta = \sum_{i\geq 1} (j-1728)^{-i}\varphi_i + (j-1728)^{-i}\frac{{\rm d}(j-1728)^{-1}}{(j-1728)^{-1}}\wedge\varphi_i', \end{equation}
and pull it back to $U''$.

Let $s$ be the highest power of $(j-1728)^{-1}$ appearing in \eqref{eq:simple form p tris}, then in ${\rm U}_{2s}/{\rm U}_{2s-1}$ we have
\[\pi_{1728}^*\beta = 1728^{-s} c_6^{-2s} (c_4^{3s}(\varphi_s+3{\rm d}c_4\wedge \varphi'_s)) - (1728)^{-s}c_6^{-2s}(2c_4^{3s}\frac{{\rm d}c_6}{c_6}\wedge\varphi'_s) \]
If $p\mid s$, by \Cref{thm:wild pres} the element above is non-zero if and only if the pair
\begin{equation*}\label{eq:pair bis} ((1728)^s c_4^{3s}(\varphi_s+3{\rm d}c_4\wedge \varphi'_s),-(1728)^s 2c_6^{2s}\varphi'_s) \in \Omega^n/Z^n \oplus \Omega^{n-1}/Z^{n-1} \end{equation*}
is non-zero. Arguing exactly as in the case of $\widetilde{\pi}^*_0$, we can conclude that $\pi_{1728}^*\beta\neq 0$ in ${\rm U}_{2s}/{\rm U}_{2s-1}$.
If $p\nmid s$, then $p\nmid 2s$, the form $\varphi_s$ is not zero and again by \Cref{thm:wild pres} we can conclude that $\pi_{1728}^*\beta\neq 0$ in ${\rm U}_{2s}/{\rm U}_{2s-1}$. This implies that $\widetilde{\pi}^*_{1728}\beta\neq 0$, as claimed.
\end{proof}

We are ready to compute the cohomological invariants of $\mathcal{M}_{1,1}$ with coefficients in ${\rm H}_{p^r}$.
\begin{thm}\label{thm:InvM11Hp}
Suppose that the base field $\bfk$ has characteristic $p>3$. Then we have
\[ \Inv(\mathcal{M}_{1,1},{\rm H}_{p^r}) \simeq \Inv(\bA^1,{\rm H}_{p^r}). \]
\end{thm}
\begin{proof}
The cohomological invariants of $\mathcal{M}_{1,1}$ coincide with the invariants of $\mathcal{M}_{1,1}\smallsetminus\{j=0,1728\}$ that are unramified on $U''$. By \Cref{lm:tame Gm char p} the tamely ramified elements are of the form $\beta + \delta \cdot \{j\} + \delta'\cdot\{j-1728\}$, where $\beta$ is an invariant of $\bA^1$ and $\delta$, $\delta'$ belong to ${\rm H}^{n+1}_{p^r}(\bfk)$. The term $\beta$ is then unramified, so our claim would follow by showing that $\delta \cdot \{j\} + \delta'\cdot\{j-1728\}$ is unramified on $U''$ if and only if $\delta=\delta'=0$.

The only divisors where the ramification can be non-zero is $\{c_4=0\}$ and $\{c_6=0\}$. The pullback of $\{j\}$ is equal to $\{c_4^3/\Delta\}$, whereas the pullback of $\{j-1728\}$ is equal to $\{c_6^2/\Delta\}$. This implies that the ramification at $\{c_4=0\}$ of the pullback of $\delta \cdot \{j\} + \delta'\cdot\{j-1728\}$ is equal to $3\delta$, whereas the ramification at $\{c_6=0\}$ is $2\delta'$. As both $2$ and $3$ are invertible in ${\rm H}^{\bullet}_{p^r}(F)$ for every field $F$, we deduce that $\delta \cdot \{j\} + \delta'\cdot\{j-1728\}$ is unramified if and only if $\delta=\delta'=0$.
\end{proof}

\subsection{Invariants with coefficients in $\K_{p^r}$}\label{sec:inv K}
We conclude this section with the computation of the invariants of $\mathcal{M}_{1,1}$ with coefficients in $\K_{p^r}$, contained in \Cref{thm:inv M11 K} stated at the beginning of the section. 
\begin{proof}[Proof of \Cref{thm:inv M11 K}]
We divide the proof in three cases, depending on the characteristic of the ground field $\bfk$. Basically all the arguments used here already appeared in the computation of the cohomological invariants with coefficients in ${\rm H}_{p^r}^{\bullet}$.

For $\bfk$ of characteristic two or three, we have that $\Mcal_{1,1}\smallsetminus\{j=0\}\simeq (\bA^1\smallsetminus\{0\})\times\Brm\ZZ/2$. It follows from \Cref{lm:invariants X x BG} that 
\[\Inv((\bA^1\smallsetminus\{0\})\times\Brm\ZZ/2,\K_{p^r})\simeq \Inv(\bA^1\smallsetminus\{0\},\K_{p^r}).\]
By \eqref{eq:KA1}, the invariants of $\bA^1\smallsetminus\{0\}$ are equal to 
\[\K^{\bullet}_{p^r}(\bfk)\oplus \K^{\bullet-1}_{p^r}(\bfk)\cdot\{j\}.\]
Suppose that $\bfk$ has characteristic two. Let $U\to\Mcal_{1,1}$ be the cover introduced in \ref{sec:char 2}, and let $\pi:U\to\bA^1$ be the map obtained by composing with the $j$-invariant $\Mcal_{1,1}\to\bA^1$. Then $\pi^*(\beta\cdot\{j\})=\beta\cdot\{a_1^12/\Delta\}$, whose ramification along $\{a_1=0\}$ is equal to $12\beta$. We deduce that $\pi^*\beta\cdot\{j\}$ is unramified if and only if $12\beta=0$ which implies that
\[ 
\Inv(\Mcal_{1,1},\K_{2^r})\simeq 
\begin{cases}
\K^{\bullet}_{2^r}(\bfk)\oplus \K^{\bullet-1}_2\cdot\{\Delta\} &\textnormal{if }r=1, \\
\K^{\bullet}_{2^r}(\bfk)\oplus\K^{\bullet-1}_4\cdot\{\Delta\} &\textnormal{if }r>1.
\end{cases}
\]
Suppose now that $\bfk$ has characteristic three. Then we have a cover $U'\to\Mcal_{1,1}$ (see \ref{sec:char 3}) and an induced map $\pi:U'\to\bA^1$ such that $\pi^*(\beta\cdot\{j\})=\beta\cdot \{b_2^6/\Delta\}$. The residue of this element at $\{b_2=0\}$ is $6\beta$, hence we deduce
\[ \Inv(\Mcal_{1,1},\K_{3^r}) \simeq  \K^{\bullet}_{3^r}(\bfk)\oplus \K^{\bullet-1}_3\cdot\{\Delta\}.\]
Finally, we deal with the case $p>3$. We have 
\[\Mcal_{1,1}\smallsetminus\{j=0,1728\}\simeq\bA^1\smallsetminus\{0,1728\}\times\Brm\ZZ/2.\]
Again by \eqref{eq:KA1} we have
\[ \Inv(\bA^1\smallsetminus\{0,1728\},\K_{p^r})\simeq \K^{\bullet}_{p^r}(\bfk)\oplus \K^{\bullet-1}_{p^r}(\bfk)\cdot\{j\}\oplus \K^{\bullet-1}_{p^r}(\bfk)\cdot\{j-1728\}. \]
Denote $U''\to\Mcal_{1,1}$ the cover introduced in \ref{sec:char p}, and set $\pi:U''\to\bA^1$ the composition of the cover with the map given by the $j$-invariant. We have 
\[\pi^*(\beta_0\cdot\{j\} + \beta_1\cdot\{j-1728\}) = \beta_0\cdot\{c_4^3/\Delta\} + \beta_1 \cdot \{c_6^2/\Delta\}. \]
For this element to be unramified both at $\{c_4=0\}$ and $\{c_6=0\}$ we must have $\beta_0=\beta_1=0$, from which we conclude
\[ \Inv(\Mcal_{1,1},\K_{p^r}) = \K^{\bullet}_{p^r}(\bfk). \]
This finishes the proof.
\end{proof}
\section{Mod $\ell$ computations} \label{sec:mod l}

In this section we complete the computation of mod $\ell$ cohomological invariants of $\Mcal_{1,1}$ from \cite{DilPirBr}*{Section 3}. This is needed for our description of ${\rm Br}(\Mcal_{1,1})$. As in \cite{DilPirBr} we will be working with cohomological invariants with coefficients in a general cycle module, which were developed in the classical case by Gille and Hirsch \cite{GilHir}. The reader can refer to \cite{DilPirBr}*{Section 2,4,5} for an introduction to the mod $\ell$ theory. 

The computations that follow will look quite similar to what we do elsewhere in the paper, with the crucial difference that mod $\ell$ cohomological invariants are homotopy invariant, i.e. if $f:\Ycal \to \Xcal$ is a vector or affine bundle the pullback $f^*$ is an isomorphism.

We begin with a simple lemma, which in a way mirrors \Cref{lm:invariants Zq}.

\begin{lm}\label{lm:Zpell}
Let $\ell$ be a positive integer that is coprime to $p$, and let ${\rm M}$ be an $\ell$-torsion cycle module. Then for any smooth scheme $X$ we have
\[
\Inv(X\times {\Brm}\ZZ/p,{\rm M})=\Inv(X,{\rm M}). 
\]
\end{lm}
\begin{proof}
This is an immediate consequence of homotopy invariance:
\[ 
X\times \bA^1 \to X \times {\Brm}\ZZ/p
\]
is a $\Ga$-torsor, and consequently a smooth-Nisnevich cover, but on the other hand 
\[
\Inv(X \times \bA^1, {\rm M}) \simeq \Inv(X,{\rm M}).
\]
\end{proof}

The presentations of $\Mcal_{1,1}$ as a quotient stack in characteristic $2$ and $3$ that we used for the mod $p$ invariants will be enough to go through our computations in the mod $\ell$ case as well. As we will see, here it is the characteristic $3$ case that is slightly trickier.

\begin{prop}
Let $\bfk$ be a field of characteristic $2$ and let $\ell>2$ be coprime to $2$. Let ${\rm M}$ be an $\ell$-torsion cycle module. Then 
\[
\Inv(\Mcal_{1,1}, {\rm M}) = {\rm M}^{\bullet}(\bfk) \oplus {\rm M}^{\bullet}(\bfk)_{3}\cdot  \lbrace \Delta \rbrace
  \]
\end{prop}
\begin{proof}
First observe that by \Cref{lm:Zpell} we have $\Inv(\Mcal_{1,1}\smallsetminus j^{-1}(0),{\rm M})\simeq \Inv(\bA^1\smallsetminus\{0\},{\rm M})$, which has been computed in \cite{DilPirBr}*{Lemma 2.22}. Specifically, if we write $\bA^1\smallsetminus\{0\}=\Spec(\bfk[j,j^{-1}])$, we get that
\[
\Inv(\bA^1\smallsetminus\lbrace 0\rbrace,{\rm M})\simeq {\rm M}(\bfk)\oplus {\rm M}^{\bullet}(\bfk)\cdot \lbrace j\rbrace.
\]
After identifying $(a'_6)^{-1}$ with $j$, we deduce that
\[ \Inv({\Mcal}_{1,1}\smallsetminus j^{-1}(0),{\rm M})\simeq {\rm M}^\bullet(\bfk)\oplus{\rm M}^{\bullet}(\bfk)\cdot\{j\}. \]

Consider the smooth-Nisnevich cover $U\to \Mcal_{1,1}$ (notation as in \ref{sub:setup}): an invariant $\gamma$ of $\Mcal_{1,1}\smallsetminus j^{-1}(0)$ pulls back to an invariant of $U\setminus \lbrace j=0 \rbrace$. If it is unramified of $\lbrace j=0 \rbrace$ then it must come from $\Mcal_{1,1}$, as it trivially glues. On the other hand if $\gamma$ comes from $\Mcal_{1,1}$ it has to be unramified on $\lbrace j=0 \rbrace$ by definition.

The element $\{j\}\cdot \beta$ with $\beta\in {\rm M}^{\bullet-1}(\bfk)$ pulls back to $\{a_1^{12}/\Delta\cdot\beta\}$, whose residue at $\{a_1=0\}$ is equal to $12\beta$. As $\ell$ is coprime to $2$ we get that $\{j\}\cdot \beta$ is unramified if and only if $\beta$ is of $3$-torsion. Finally by the same reasoning we have $\lbrace j \rbrace \cdot \beta = \lbrace \Delta \rbrace \cdot \beta$
\end{proof}

\begin{prop}
Let $\bfk$ be a field of characteristic $3$ and let $\ell >1$ be coprime to $3$. Let ${\rm M}$ be an $\ell$-torsion cycle module. Then
\[
\Inv(\Mcal_{1,1}, {\rm M}) = {\rm M}^{\bullet}(\bfk) \oplus \lbrace\Delta \rbrace \cdot {\rm M}^{\bullet}(\bfk)_4
  \]
\end{prop}
\begin{proof}
Consider the smooth-Nisnevich cover $U \to \Mcal_{1,1}$. First we observe that $\lbrace \Delta \rbrace \cdot \beta$ is clearly unramified on $U$ and glues whenever $\beta$ is of $4$ torsion as ${\rm m}^*\lbrace\Delta \rbrace = \lbrace u^{12}\Delta \rbrace = 12 \lbrace u \rbrace + \lbrace \Delta \rbrace$.

Now recall that we have an isomorphism

\[
\Mcal_{1,1} \setminus \lbrace j=0\rbrace \simeq (\bA^1 \smallsetminus \lbrace 0 \rbrace) \times {\Brm}\ZZ/2 \simeq (\bA^1 \smallsetminus \lbrace 0 \rbrace) \times {\Brm}\mu_2
\]

as $\ZZ/2 \simeq \mu_2$ when the characteristic of $\bfk$ is not 2. Applying \cite{DilPirBr}*{Proposition 4.3} we get

\[
\Inv(\Mcal_{1,1}, {\rm M}) = ({\rm M}^{\bullet}(\bfk) \oplus \lbrace j \rbrace \cdot {\rm M}^{\bullet}) \oplus \alpha \cdot {\rm M}^{\bullet}(\bfk)_2 \oplus \alpha \cdot \lbrace j \rbrace \cdot {\rm M}^{\bullet}(\bfk)_2.
\]

Using the presentation 
\[
\Mcal_{1,1} \smallsetminus \simeq \left[ \Spec(\bfk\left[b_2, b_4, b_6, \Delta^{-1}\right])/\Gm \rtimes \Ga\right], \, j=b_2^6/\Delta
\]

we see that $\lbrace j\rbrace= 6 \lbrace b_2 \rbrace - \lbrace \Delta \rbrace, \, \alpha=\lbrace b_2 \rbrace$. Note that $\alpha(F)$ should be seen as an element in $\K_2^1(F)=F^*/(F^*)^2$ and can only multiply elements of $2$-torsion in ${\rm M}^{\bullet}(F)$.

Now, the cohomological invariants of $\Mcal_{1,1} \smallsetminus \lbrace j=0 \rbrace$ coming from $\Mcal_{1,1}$ are exactly those that when pulled back to $U$ are unramified at $b_2=0$. 

Consider a general element 
\[
\gamma = \beta_0 + \lbrace j \rbrace \cdot \beta_1 + \alpha \cdot \beta_2 + \alpha \cdot \lbrace j \rbrace \cdot \beta_3
\]
with $\beta_0, \beta_1 \in {\rm M}^{\bullet}(\bfk)$ and $\beta_2, \beta_3 \in {\rm M}^{\bullet}(\bfk)_2$. We have 
\[
\alpha \cdot \lbrace j \rbrace \cdot \beta_3 = 6\lbrace b_2\rbrace \cdot \lbrace b_2 \rbrace \cdot \beta - \lbrace \Delta \rbrace \cdot \lbrace b_2 \rbrace \cdot \beta =  \lbrace \Delta \rbrace \cdot \lbrace b_2 \rbrace \cdot \beta
\]
which shows that the ramification of $\gamma$ at $b_2 = 0$ is
\[
6 \beta_1 + \beta_2 + \lbrace \Delta \rbrace \cdot \beta_3
\]
as there are elliptic curves over $\bfk\left[ t, t^{-1} \right]$ with $\Delta={t^{-1}}$ the element $\lbrace \Delta \rbrace \cdot \beta_3$ can never cancel out with the other two, which come from the base field. Then we must have $\beta_3 =0$. The only possibility for the remaining elements is that $\beta_1$ is of $4$-torsion and $\beta_2 = -6 \beta_1$. In other words, all unramified elements are in the form
\[
\beta_0 + \lbrace j \rbrace \cdot \beta_1 + \alpha \cdot (6 \beta_1)
\]
for some $\beta_0 \in {\rm M}^{\bullet}(\bfk), \, \beta_1 \in {\rm M}^{\bullet}(\bfk)_4$. Finally, we have 
\[
\lbrace j \rbrace \cdot \beta_1 - \alpha \cdot (6 \beta_1) = (6 \lbrace b_2 \rbrace - \lbrace \Delta \rbrace)\cdot \beta_1 - \lbrace b_2 \rbrace \cdot (6 \beta_1) = - \lbrace \Delta \rbrace \cdot \beta_1
\]
proving our claim.
\end{proof}

\begin{cor}\label{cor:Br_l}
Write $c={\rm char}(\bfk)$ and $^c{\rm Br}'(\Mcal_{1,1})$ for the subgroup of ${\rm Br}'(\Mcal_{1,1})$ whose elements have order not divisible by $c$. Then
\[
\begin{cases}
^2{\rm Br}'(\Mcal_{1,1})={^2{\rm Br}}'(\bfk) \oplus {\rm H}^1(\bfk, \ZZ/3) \, \textnormal{ if }\, c=2 \\
^3{\rm Br}'(\Mcal_{1,1})={^3{\rm Br}}'(\bfk) \oplus {\rm H}^1(\bfk, \ZZ/4)\, \textnormal{ if }\, c=3
\end{cases}
\]
\end{cor}
\begin{proof}
This is an immediate consequence of the computations above applied for ${\rm M}={\rm H}_{\ZZ/\ell(-1)}$ and degree $2$.
\end{proof}

\section{The Brauer group of $\mathcal{M}_{1,1}$}\label{sec:brauer}

We now have all the tools to compute the Brauer group of $\Mcal_{1,1}$ over any field.

\begin{thm}
Let $\mathcal{M}_{1,1}$ be the stack over $\Spec(\bfk)$ parametrizing elliptic curves. If ${\rm char}(\bfk)=c$ the group ${\rm Br}(\mathcal{M}_{1,1})$ is:
\[
\begin{cases}
{\rm Br}(\mathbb{A}^1_\bfk) \oplus {\rm H}^{1}(\bfk,\ZZ/12\ZZ) & \mbox{if } c \neq 2,\\
{\rm Br}(\mathbb{A}^1_\bfk) \oplus {\rm H}^{1}(\bfk,\ZZ/3\ZZ) \oplus {\rm J} & \mbox{if } c=2, \, x^2+x+1 \, \mbox{ irreducible over } \, \bfk, \\
{\rm Br}(\mathbb{A}^1_\bfk) \oplus {\rm H}^{1}(\bfk,\ZZ/12\ZZ) \oplus \ZZ/2\ZZ & \mbox{if } c=2, \, x^2+x+1 \, \mbox{ has a root in } \, \bfk,
\end{cases}
\]
where ${\rm H}^1(\bfk,\ZZ/4) \subset {\rm J} \subseteq {\rm H}^1(\bfk,\ZZ/8)$ sits in an exact sequence
\[
0 \to  {\rm H}^1(\bfk,\ZZ/4) \to {\rm J} \to \ZZ/2 \to 0. 
\]
\end{thm}
\begin{proof}
First we note that by \cite{Shi}*{Lemma 3.1} the Brauer map ${\rm Br}(\Mcal_{1,1}) \to {\rm Br}'(\Mcal_{1,1})$ is surjective, so we only need to worry about computing the latter.

The $\ell$-torsion is computed in \cite{DilPirBr}*{Corollary 3.2} for $c \neq 2,3$ and in \Cref{cor:Br_l} for $c=2,3$. The $p$-torsion for $c$ equals to respectively $2$, $3$ and $p>3$ is obtained by restricting \Cref{thm:inv M11 H} to degree $2$.
\end{proof}

Now let us look at the generators. For $c \neq 2$ the group is generated by the elements coming from the base field and elements in the form $\lbrace \Delta \rbrace \cdot \beta$ with $\beta \in {\rm H}^1(\bfk,\ZZ/12)$. Write ${\rm H}^1(\bfk,\ZZ/12)={\rm H}^1(\bfk,\ZZ/4)\oplus{\rm H}^1(\bfk,\ZZ/3)$ and $\lbrace \Delta \rbrace = \lbrace \Delta \rbrace_4 +\lbrace \Delta \rbrace_3 $. Then by either \cite{DilPirBr}*{Lemma 2.18} or \Cref{prop:H1} we know that $\lbrace \Delta \rbrace_4$ and $\lbrace \Delta \rbrace_3$ come from ${\rm H}^1_{\rm fl}(\Mcal_{1,1},\mu_n),\, n=3,4$ (note that for $n$ not divisible by $c$ that's just the regular \'etale cohomology group) and consequently $\lbrace \Delta \rbrace$ comes from ${\rm H}^1_{\rm fl}(\Mcal_{1,1},\mu_{12})$. Thus we can conclude that the elements in the form $\lbrace \Delta \rbrace \cdot \beta$ are all cyclic algebras by \cite{AntMeiEll}*{Lemma 2.10}.

When $c=2$ the generators are in the form $\lbrace \Delta \rbrace \cdot \beta$, $\left[ \alpha, \Delta \right\rbrace$ (here we're identifying the generator with its pullback to $U$) or a sum of the two. Any element of the first type is an cyclic algebra by the same reasoning as above. 

We claim that $\left[ \alpha, \Delta \right\rbrace$ is not a cyclic algebra. First, note that $\left[ \alpha, \Delta \right\rbrace$ does not go to zero if we pass to $\overline{\bfk}$, so it suffices to show that it is not a cyclic algebra when the base field is algebraically closed. It is a $2$-torsion element, so saying it is a cyclic algebra is equivalent to saying there must exist $h \in {\rm H}^1_{\rm fl}(\Mcal_{1,1},\mu_2)$ and $h' \in {\rm H}^1(\Mcal_{1,1},\ZZ/2)$ with $h \cdot h' = \left[ \alpha, \Delta \right\rbrace$. When $\bfk = \overline{\bfk}$ we have
\[
{\rm H}^1_{\rm fl}(\Mcal_{1,1},\mu_2)=\lbrace \Delta \rbrace \cdot \ZZ/2, \quad {\rm H}^1(\Mcal_{1,1},\ZZ/2)={\rm H}^1(\bA^1_{\bfk},\ZZ/2).
\]
If we pull everything back to $\Mcal_{1,1} \smallsetminus \lbrace j=0\rbrace \simeq (\bA^1\smallsetminus\{0\}) \times {\Brm}\ZZ/2$ we immediately see that $\lbrace\Delta \rbrace = \lbrace j \rbrace \in {\rm H}^1_{\rm fl}(\Mcal_{1,1},\mu_2)$, which shows that 
\[
{\rm H}^1_{\rm fl}(\Mcal_{1,1},\mu_2)\subset {\rm H}^1_{\rm fl}(\bA^1\smallsetminus\{0\},\mu_2) \subset {\rm H}^1_{\rm fl}(\Mcal_{1,1}\smallsetminus \lbrace j=0\rbrace,\mu_2)
\]
and
\[
{\rm H}^1(\Mcal_{1,1},\ZZ/2)\subset {\rm H}^1(\bA^1\smallsetminus\{0\},\ZZ/2) \subset {\rm H}^1(\Mcal_{1,1}\smallsetminus \lbrace j=0\rbrace,\ZZ/2).
\]
But then we must have $h \cdot h' \in {\rm Br}'(\bA^1\smallsetminus\{0\})_2$ which is zero when $\bfk = \overline{\bfk}$ by \Cref{cor:inv A1} (note that $\Omega_{\bfk}=0$ as $\Omega_{\bfk, {\rm log}}$ is $2$-divisible), a contradiction.

\begin{rmk}
A natural question left open in Shin's results \cite{Shi} is, given a finite field $\bfk$ of characteristic $2$ not containing a third root of unit $\zeta$, to compute the restriction map 
\[
{\rm Br}(\Mcal_{1,1,\bfk}) \to {\rm Br}(\Mcal_{1,1,\bfk(\zeta)}).
\]
Our description makes the task quite simple: we just need to understand the restriction map 
\[{\rm H}^1(\bfk, \ZZ/8\ZZ) \to {\rm H}^1(\bfk(\zeta), \ZZ/8\ZZ).\]
Recall that ${\rm H}^1(F, \ZZ/8\ZZ)$ is equal to the quotient of $W_3(F)$ by the subgroup of elements in the form $\left[a_1^2,a_2^2,a_3^2\right]-\left[a_1,a_2,a_3\right]$. Using the definition of of Witt vectors we get the formula
\[
\begin{small} \left[x_1,x_2,x_3\right]+\left[y_1,y_2,y_3\right]=\left[x_1+y_1,x_2+y_2+x_1y_1, S_3(\underline{x},\underline{y})\right] \end{small},
\]
\[S_3(\underline{x},\underline{y})=x_3+y_3+x_2y_2 + x_2x_1y_1+y_2x_1y_1 + x_1^2y_1^2+x_1^3y_1+y_1^3x_1.
\]
We know by standard Galois theory that both groups are isomorphic to $\ZZ/8\ZZ$, and the first group is generated by $\left[1,0,0\right]$ as $4\left[1,0,0\right]=\left[0,0,1\right]\neq 0$. So to understand the map we just have to check the class of $\left[1,0,0\right]$ in ${\rm H}^1(\bfk(\zeta), \ZZ/8\ZZ)$. Now, observe that
\[
\left[\zeta^2,\zeta^2,\zeta^2\right] - \left[\zeta, \zeta, \zeta\right] = \left[\zeta+1,\zeta+1,\zeta+1\right] + \left[\zeta, 1, \zeta+1\right]=\left[1,0,0\right]
\]
which implies that the class of $\left[1,0,0\right]$ maps to zero, and as a degree two extension will induce an isomorphism on mod $3$ Galois cohomology we conclude that the map of Brauer groups maps $\ZZ/24\ZZ={\rm Br}(\Mcal_{1,1,\bfk})$ to $\ZZ/3\ZZ \subset {\rm Br}(\Mcal_{1,1,\bfk(\zeta)})=\ZZ/12\ZZ \times \ZZ/2\ZZ$.
\end{rmk}

\end{document}